\documentclass[11pt]{article} 
\usepackage{amsmath,amssymb,amsthm}
\usepackage{hyperref}
\usepackage{xcolor}
\usepackage{graphicx}

\usepackage[letterpaper,top=3.5cm,bottom=3.5cm,left=3.5cm,right=3.5cm,marginparwidth=1.75cm]{geometry}

\title{Inverse Spectral Problems for Collapsing Manifolds II: Quantitative Stability of Reconstruction for Orbifolds} 

\author{Matti Lassas,\, Jinpeng Lu,\, Takao Yamaguchi \\ \vspace{-2mm}
\\
\emph{Dedicated to the memory of Yaroslav Kurylev} }

\theoremstyle{theorem}
\newtheorem{lemma}{Lemma}[section]
\newtheorem{proposition}[lemma]{Proposition}
\newtheorem{coro}[lemma]{Corollary}
\newtheorem{theorem}[lemma]{Theorem}

\theoremstyle{definition}
\newtheorem{definition}[lemma]{Definition}

\newtheorem{remark}[lemma]{Remark}
\newtheorem{problem}[lemma]{Problem}

\numberwithin{equation}{section}

\def\R{\mathbb R}

\def\Z{\mathbb Z}
\def\N{\mathbb N}

\def\hat{\widehat}

\def\tilde{\widetilde}
\def \bfo {\begin {eqnarray*} }
\def \efo {\end {eqnarray*} }
\def \ba {\begin {eqnarray*} }
\def \ea {\end {eqnarray*} }
\def \beq {\begin {eqnarray}}
\def \eeq {\end {eqnarray}}
\def \bequ {\begin {equation}}
\def \eequ {\end {equation}}

\def \diam {\hbox{diam}}

\def\diag{\hbox{diag }}
\def \det {\hbox{det}}
\def\bra{\langle}
\def\cet{\rangle}
\def \e {\varepsilon}
\def \p {\partial}

\def\a{\alpha}

\def \LK {\Lambda_3}

\begin{document}

\AtEndDocument{\bigskip{\footnotesize%

  \textsc{Matti Lassas: Department of Mathematics and Statistics, University of Helsinki, FI-00014 Helsinki, Finland} \par  
  \textit{Email address}: \texttt{matti.lassas@helsinki.fi} \par
  
  \addvspace{\medskipamount}
  \textsc{Jinpeng Lu: Department of Mathematics and Statistics, University of Helsinki, FI-00014 Helsinki, Finland} \par
  \textit{Email address}: \texttt{jinpeng.lu@helsinki.fi} \par

\addvspace{\medskipamount}
  \textsc{Takao Yamaguchi: Institute of Mathematics, University of Tsukuba, Tsukuba 305-8571, Japan} \par
  \textit{Email address}: 
  \texttt{takao@math.tsukuba.ac.jp}
  
}}

\date{}
\maketitle

\begin{abstract}
We consider the inverse problem of determining the metric-measure structure of collapsing manifolds from local measurements of spectral data.
In the part I of the paper, we proved the uniqueness of the inverse problem and a continuity result for the stability in the closure of Riemannian manifolds with bounded diameter and sectional curvature in the measured Gromov-Hausdorff topology.
In this paper we show that when the collapse of dimension is $1$-dimensional, it is possible to obtain quantitative stability of the inverse problem for Riemannian orbifolds.
The proof is based on an improved version of the quantitative unique continuation for the wave operator on Riemannian manifolds by removing assumptions on the covariant derivatives of the curvature tensor.
\end{abstract}

\section{Introduction}

We consider the class of connected, closed, smooth Riemannian manifolds $(M,g)$ defined by
\begin{equation} \label{bounded-intro}
{\rm dim}(M)=n, \quad |R(M)|\leq \Lambda^2, \quad  {\rm diam}(M)\leq D,
\end{equation}
where $R(M), {\rm diam}(M)$ are the sectional curvature and diameter of $M$.
Denote by $\mu_M$ the normalized Riemannian measure,
\begin{equation} \label{normalized-measure}
d\mu_M=\frac{1}{\textrm{Vol}(M)}dV_g,
\end{equation}
where $\textrm{Vol}(M)$ is the Riemannian volume of $(M,g)$, and $dV_g$ is the Riemannian volume element of $(M,g)$.
We denote this class of manifolds $(M,g)$ equipped with their normalized measure $\mu_M$ by ${\frak M}{\frak M} (n,\Lambda,D)$. Let $\overline{{\frak M}{\frak M}}(n,\Lambda,D)$ be its closure with respect to the measured Gromov-Hausdorff topology (\cite{F87}).

It is well-known that a sequence of $n$-dimensional manifolds in the class ${\frak M\frak M}(n,\Lambda, D)$ can collapse to a lower dimensional space when the injectivity radius of the sequence of manifolds goes to zero.
A metric-measure space $(X,\mu)\in \overline{{\frak M\frak M}}(n,\Lambda, D)$ has the stratification 
\bequ
     X = S_0(X) \supset S_1(X) \supset \cdots \supset S_d(X),
       \label{eq:stratif}
\eequ
with the following property: if $S_j(X)\setminus S_{j+1}(X)$ is non-empty, then it is a 
$(d-j)$-dimensional Riemannian manifold, where $d=\textrm{dim}(X)\leq n$. The regular part of $X$ is the set
$X^{reg}:=S_0(X) \setminus S_1(X)$ which is an open $d$-dimensional Riemannian manifold of Zygmund class $C^2_{\ast}$, that is, the transition maps between coordinate charts are of class $C^3_*$ and the metric tensor in these charts is of class $C^2_*$ (\cite{F88,KLLY}).
Recall that the Zygmund space $C^2_*(X^{reg})$ has the relation $C^{1,1}(X^{reg})\subset C^2_*(X^{reg})\subset C^{1,\alpha}(X^{reg})$ for any $0<\alpha<1$.
The singular set of $X$ is the complement of the regular part, denoted by $X^{sing}:=X\setminus X^{reg}$ or $S$. In particular, the singular set $X^{sing}$ has dimension at most $d-1$, and $\mu(X^{sing})=0$.
The measure $\mu$ is absolutely continuous with respect to the Riemannian measure $dV_g$ on $X^{reg}$, with the density function $\rho_X\in C^2_*(X^{reg})$.
In particular, if $\dim(X)=n-1$ then $X$ is a Riemannian orbifold \cite{F90}. 
Recall that a metric space $X$ is called a Riemannian orbifold if, roughly speaking, for each $p\in X$ there exists a neighborhood $U$ of $p$ such that $U$ is isometric to the quotient of a Riemannian manifold by an action of a finite group of isometries (see e.g. \cite{DCGW,Satake,Thurston} or Appendix \ref{appendix_orbifold} for a formal definition).
As a basic example, the product manifold $\mathbb{S}^2\times [0,\epsilon]$ with points $({\bf x}, 0) \in \mathbb{S}^2 \times \{0\}$ and $(e^{2\pi i/m} \cdot {\bf x}, \epsilon) \in \mathbb{S}^2 \times \{\epsilon\}$ identified collapses to a Riemannian orbifold $\mathbb{S}^2/\mathbb{Z}_m$ as $\epsilon\to 0$, where $e^{2\pi i/m} \cdot {\bf x}$ stands for the rotation on the unit sphere $\mathbb{S}^2\subset \mathbb{R}^3$ by the angle $2\pi/m$ along the $z$-axis.
The orbifold $\mathbb{S}^2/\mathbb{Z}_m$, having the shape of a rugby ball, has conic singular points only at the north and south poles. 
Relevant figures and an example of collapsing manifolds in physics can be found in Section \ref{review-physics}.

As shown in \cite {F87}, the Dirichlet's quadratic form $A[u,v]=\bra du,dv\cet _{L^2(X^{reg},d\mu)}^2$,
$u,v\in C^{0,1}(X)$,  defines
a self-adjoint operator $\Delta_X$ on $(X, \mu)$, which we call the weighted Laplacian, denoted by $\Delta_X$.
In a local coordinate $(x^1,\cdots,x^d)$ of $(X^{reg},g_X)$, it has the form
\begin{equation} \label{Delta-intro}
 \Delta_X u= \frac{1}{\rho_X |g_X|^{\frac12}} \sum_{j,k=1}^d \frac \p{\p x^j} \Big(\rho_X |g_X|^{\frac12} g^{jk}_X \frac \p{\p x^k}  u  \Big), \quad |g_X|=\det\big((g_X)_{jk} \big).
\end{equation}
Note that in the part I of the paper \cite{KLLY} we used the nonnegative definite Laplacian, while in this paper we use analyst's nonpostive definite Laplacian. 
Denote by $\lambda_j$ the $j$-th eigenvalue of the weighted Laplacian $\Delta_X$ and by $\phi_j$ the corresponding orthonormalized eigenfunction in $L^2(X,\mu)$,
\begin{equation}
-\Delta_X \phi_j=\lambda_j \phi_j, \quad j=0,1,\cdots.
\end{equation}
Fukaya proved in \cite{F87} that the $j$-th eigenvalue, for any $j$, is a continuous function on  $\overline{{\frak M}\frak M}(n,\Lambda,D)$ with respect to the measured Gromov-Hausdorff topology.

\smallskip
We consider the following generalization of Gel'fand's inverse problem \cite{G54} in the class $\overline{{\frak M}{\frak M}}(n,\Lambda,D)$.

\begin{problem} \label{problem-collapsing}
Let $(X,\mu)\in \overline{{\frak M}\frak M}(n,\Lambda,D)$ and $p\in X$. Suppose that one can measure the spectral data $\{\lambda_j,\phi_j|_{B(p,r)}\}_{j=0}^{\infty}$ in a ball $B(p,r)$ of radius $r>0$. 
Do these data uniquely determine the metric structure of $X$ and the measure $\mu$?
\end{problem}

In the special case of $X$ being \emph{a priori} a Riemannian manifold of dimension $n$, the problem is reduced to the classical case that was solved in \cite{AKKLT,BK,KrKL}, with quantitative stability of $\log$-$\log$ type proved in \cite{BKL3,BILL}. Stronger H\"older type of stability estimates can be obtained in \cite{AG,BD,SU,SU2} if additional geometric assumptions are assumed, e.g., if the metric is close to being simple.
In general, the positive answer to this problem and an abstract continuity for the stability are proved in the part I of this paper \cite{KLLY}. In this part II, we investigate if a quantitative stability estimate for the problem can be obtained in the general case when $X$ is not a manifold.  
We also discuss applications of Gel'fand's inverse problem for collapsing manifolds in manifold learning and physics in Section \ref{sec-applications}.

For the purpose of obtaining a stability estimate, we need to impose bounds for the covariant derivatives of the curvature tensor.
Following similar notations in \cite{CFG}, we say a Riemannian manifold $M$ is $\Lambda_k$-regular if
\begin{equation} \label{higher-curvature-intro}
\| \nabla^i R(M) \| \leq \Lambda_k,\quad \textrm{for }\,i=1,\cdots,k.
\end{equation}
We denote the class of connected closed smooth Riemannian manifolds satisfying \eqref{bounded-intro} and \eqref{higher-curvature-intro} by ${\frak M}{\frak M}(n,\Lambda,\Lambda_k,D)$, and denote its closure by $\overline{{\frak M}{\frak M}}(n,\Lambda,\Lambda_k,D)$ with respect to the measured Gromov-Hausdorff topology.
Our main result is to show that when $\textrm{dim}(X)=n-1$, i.e., $X$ is an orbifold, it is possible to obtain a quantitative stability estimate for the reconstruction of the regular part from incomplete spectral data.

\begin{theorem} \label{thm-GH}
Let $(X,\mu) \in \overline{{\frak M}{\frak M}}(n,\Lambda,\LK,D)$ and $p\in X^{reg}$. Suppose $\dim(X)=n-1$ and ${\rm Vol}_{n-1}(X) \geq v_0$.
Let $\sigma\in (0,1)$, and $U\subset X^{reg}$ be an open subset containing a ball $B_X(p,r_0)$.
Then there exists $\widehat{\delta}=\widehat{\delta}(X,r_0,\sigma)>0$, such that
the finite interior spectral data $\{\lambda_j,\phi_j|_U\}_{j=0}^{\delta^{-1}}$ for $\delta<\widehat{\delta}$ determine a finite metric space $\widehat{X}$ such that
$$d_{GH}(X\setminus S_{\sigma,\delta},\widehat{X}) < C_1(X,\sigma) \Big(\log\log|\log \delta| \Big)^{-C_2},$$
where $S_{\sigma,\delta}$ is a subset (depending on $\sigma,\delta$) of the $C_3\sigma^{1/4}$-neighborhood of the singular set of $X$, and $X\setminus S_{\sigma,\delta}$ is equipped with the restriction of the metric of $X$.
The constant $C_1(X,\sigma)$ depends on $\sigma,r_0,n,\Lambda,\LK,D,v_0$ as well as the space $X$, and $C_1(X,\sigma)\to \infty$ as $\sigma\to 0$. The constant $C_2$ depends on $n$, and $C_3$ depends on $n,\Lambda,D$.

Moreover, suppose that the complete spectral data of $(X,\mu), (X',\mu')\in \overline{{\frak M}{\frak M}}(n,\Lambda,D)$ on open sets $U\subset X$, $U'\subset X'$ are equivalent by a homeomorphism $\Psi_{U}: U\to U'$, namely $\lambda_j=\lambda'_j$ and $\Psi_U^*(\phi'_j|_{U'})=\phi_j|_U$ for all $j\in \N$, where $\lambda'_j,\phi'_j$ are the eigenvalues and orthonormalized eigenfunctions of the weighted Laplacian on $(X',\mu')$.
Then there exists a measure-preserving isometry $F: X\to X'$ such that $F|_{U}=\Psi_U$ and $\mu=F^* (\mu').$
\end{theorem}

Recall that the singular set of $X$ is a union of submanifolds having dimension at most $n-2$.
Note that the incomplete data case of the theorem above is of an asymptotic nature as the reconstruction is done to a fixed orbifold.
Our method can also be extended to the case where the finite interior spectral data are given up to small noise.
The case for complete spectral data in the theorem above, valid for any dimension of collapse, is due to the part I of this paper \cite{KLLY}.

The basic idea for proving Theorem \ref{thm-GH} is as follows. Let $M_i\in {\frak M}{\frak M}(n,\Lambda,D)$ be a sequence of Riemannian manifolds converging to $X$. We consider the orthonormal frame bundles $FM_i$ of $M_i$ endowed with the natural Riemannian metric. Note that $O(n)$ isometrically acts on $FM_i$ and $M_i=FM_i/O(n)$. Passing to a subsequence, we may assume that $FM_i$ converge to a space $Y$, and $X=Y/O(n)$. 
The purpose of this construction is that unlike $X$, the limit space $Y$ is always a Riemannian manifold, see \cite{F87,F88}, and the Riemannian metric on $Y$ is of class $C^2$ under the $\Lambda_k$-regular condition \eqref{higher-curvature-intro}, see Section \ref{section-prelimi}.
Thus, we can apply the quantitative unique continuation for the wave equation and a quantitative version of the boundary control method on $Y$ to reconstruct the interior distance functions on $U\subset X$ in Section \ref{sec-coefficients} and \ref{sec-recon}. Next, using the interior distance functions, we approximately locate the singular set of $X$ and reconstruct the smooth metric structure away from the singular set in Section \ref{sec-singular}. 
However, due to the lack of higher order curvature bounds on the limit space $Y$, we need to improve the earlier results on the quantitative unique continuation \cite{BKL2,BKL3,BILL,DLLO} by removing the assumptions on the covariant derivatives of the curvature tensor.
This result is stated as follows.

\smallskip
Let $(Y,g)$ be a connected, closed, smooth manifold with $C^2$-smooth Riemannian metric $g$, and let $\rho_Y\in C^1(Y)$ be a density function on $Y$.
Consider the wave equation with the weighted Laplacian $\Delta_Y$ on $Y$,
\begin{equation} \label{wave-Y-intro}
(\partial_t^2-\Delta_Y) u= f,
\end{equation}
where $\Delta_Y$ is given by \eqref{Delta-intro} with $\rho_X$ replaced by the density function $\rho_Y$.

\begin{theorem}\label{uc-Y}
Let $(Y,g)$ be a connected, closed, smooth manifold with $C^2$-smooth Riemannian metric $g$, satisfying 
$${\rm dim}(Y)=n, \quad |R(Y)|\leq \Lambda^2, \quad  {\rm diam}(Y)\leq D, \quad {\rm Vol}(Y)\geq c_0.$$
Let $\rho_Y\in C^1(Y)$ be a density function on $Y$ and consider the weighted Laplacian operator $\Delta_Y$.
Suppose $u\in H^1(Y\times[-T,T])$ is a solution of the non-homogeneous wave equation \eqref{wave-Y-intro} with $f\in L^2(Y\times [-T,T])$. Let $V\subset Y$ be an open subset with smooth boundary.
If the norms satisfy
$$\|u\|_{H^1(Y\times[-T,T])}\leq E,\quad \|u\|_{H^{1}(V\times [-T,T])}\leq \varepsilon_0,$$
then for $0<h<h_0$, we have
$$\|u\|_{L^2(\Omega(h))}\leq C_4 \exp(h^{-C_5})\frac{E}{\bigg(\log \big(1+\frac{h^{-1}E}{\|f\|_{L^2(Y\times[-T,T])}+h^{-2}\varepsilon_0}\big)\bigg) ^{\frac{1}{2}}}\, .$$
The domain $\Omega(h)$ is defined by
$$\Omega(h)=\big\{(x,t)\in (Y \setminus V)\times [-T,T]: T-|t|-d_Y(x,V) > \sqrt{h} \big\}.$$
The constants $C_4,h_0$ explicitly depend on $n,T,\Lambda,D,c_0$ and $\|\rho_Y\|_{C^{0,1}}$, and $C_5$ explicitly depends on $n$. The norms are with respect to the Riemannian volume element on $(Y,g)$.
\end{theorem}

The proof of Theorem \ref{uc-Y} is based on the smoothening of Riemannian metric \cite{Ba,BMR} and is done in Section \ref{sec-UC}.

\medskip
\noindent \textbf{Acknowledgements.}
The authors thank Atsushi Kasue for helpful discussions.
 M.L. and J.L. were supported by PDE-Inverse project of the European Research Council of the European Union, 
 the FAME-flagship and the grant 336786 of the Research Council of Finland. Views and opinions expressed are
those of the authors only and do not necessarily reflect those of the European
Union or the other funding organizations.  
Neither the European Union nor the other funding organizations can be held responsible for them. 
T.Y. was supported by JSPS KAKENHI Grant Number 21H00977.

\section{Preliminaries}
\label{section-prelimi}

\subsection{Geometric structure of collapsing manifolds}

Given $n\in \N_+$, $\Lambda,D>0$, let ${\frak M}(n,\Lambda,D)$ denote the class of connected, closed, smooth Riemannian manifolds $(M,g)$ satisfying
\begin{equation}
{\rm dim}(M)=n, \quad |R(M)|\leq \Lambda^2, \quad  {\rm diam}(M)\leq D,
\end{equation}
and let $\overline{{\frak M}}(n,\Lambda,D)$ denote its closure in the Gromov-Hausdorff topology.
Suppose that $X \in \overline{{\frak M}}(n,\Lambda,D)$ is a compact metric space which is the Gromov-Hausdorff limit of closed Riemannian manifolds $M_i \in {\frak M}(n,\Lambda,D)$.
Let $FM_i$ be the orthonormal frame bundle of $M_i$. 
Fixing a Riemannian metric on $O(n)$, $FM_i$ has a natural Riemannian structure with uniformly bounded sectional curvature by O'Neill's formula,
\begin{equation} \label{sectional-FM}
|R(FM)| \leq \Lambda_F^2=C(n,\Lambda).
\end{equation}
The dimension and diameter of $FM$ satisfy
\begin{equation}
\dim(FM)=n+\dim(O(n)),\quad \textrm{diam}(FM)\leq D_F=C(n,D).
\end{equation}

The limit space $X$ has the following geometric structure due to Fukaya.

\begin{theorem}[\cite{F88}] \label{thm-Y}
Let $X\in \overline{{\frak M}}(n,\Lambda,D)$. Then there exists a smooth manifold $Y$ with $C^{1,\alpha}$-Riemannian metric, for any $0<\alpha<1$, 
on which $O(n)$ acts as isometries in such a way that $X$ is isometric to $Y/O(n)$, and $\dim(Y)=\dim(X)+\dim(O(n))$.
\end{theorem}

The manifold $Y$ is constructed as the Gromov-Hausdorff limit of the orthonormal frame bundles $FM_i$ of $M_i$.
Here we briefly explain why $Y$ is a manifold. 
Fix any point $p\in X$ and put  $p_i:= \psi_i(p)$, where $\psi_i:X\to M_i$ is
an $\e_i$-Gromov-Hausdorff approximation with $\lim\limits_{i\to\infty} \e_i = 0$.
Let $B \subset \R^n$ be the open ball around the origin $\it O$ in $\R^n$ of radius $\pi/\Lambda$, and
let $\exp_i: B\to M_i$ be the composition of the exponential map $\exp_{p_i}:
 T_{p_i}(M_i)\to M_i$ and a linear isometric embedding 
$B\to B({\it O},\frac{\pi}{\Lambda})\subset T_{p_i}(M_i)$.
Since $|R(M_i)| \leq \Lambda^2$, the exponential map on $B(O,\frac{\pi}{\Lambda})$ has maximal rank, and 
we have the pull-back metric $\tilde g_i := \exp_{i}^*(g_{i})$ on $B\subset \R^n$.
Moreover, the injectivity radius of $(B,\tilde g_i)$ is uniformly bounded from below, e.g. \cite[Lemma 8.19]{GLP}.
Therefore, we may assume that $(B, \tilde g_i)$ converges to a
$C^{1,\alpha}$-metric $(B, \tilde g_0)$ in the $C^{1}$-topology, 
for any $0<\alpha < 1$ (\cite{A90,AKKLT}).

Let $G_i$ denote the set of all  isometric 
embeddings $\gamma:(B', \tilde g_i) \to (B, \tilde g_i)$ such that
$\exp_i \circ \gamma = \exp_i$ on $B'$, where $B':= B({\it O}, \frac{\pi}{2\Lambda})\subset\R^n$. 
Then $G_i$ forms a local pseudogroup, see e.g. \cite[Section 7]{F90}.
Passing to a subsequence, we may assume that
$G_i$ converges to a local pseudogroup   $G$ consisting of isometric 
embeddings $\gamma:(B', \tilde g_0) \to (B, \tilde g_0)$, and that
$(B,\tilde g_i, G_i)$ converges to $(B, \tilde g_0,G)$ in the 
equivariant Gromov-Hausdorff topology. The quotient space $(B',\tilde g_i)/G_i=B(p_i, \frac{\pi}{2\Lambda})\subset M_i$ 
converges to $(B',\tilde g_0)/G$,
which implies that 
$(B', \tilde g_0)/G=B(p, \frac{\pi}{2\Lambda}) \subset X.$
(See \cite{FY:fundgp} for further details on basic properties of the equivariant Gromov-Hausdorff convergence.)

The fact that $Y$ is a Riemannian manifold can be seen as follows.
The pseudogroup action of $G_i$ on $(B', \tilde g_i)$ induces an isometric pseudogroup
action, denoted by $\hat G_i$, on the frame bundle $F(B', \tilde g_i)$ of
$(B', \tilde g_i)$ defined by differential. Therefore, 
$F(B', \tilde g_i)/\hat G_i = FB(p_i, \frac{\pi}{2\Lambda})$. Passing to a subsequence,
we may assume that  $(F(B', \tilde g_i), \hat G_i)$ 
converges to $(F(B', \tilde g_0), \hat G)$ in the equivariant
GH-topology, where $\hat G$ denotes the
isometric pseudogroup action on $F(B',\tilde g_0)$ induced from that of 
$G$ on $(B', \tilde g_0)$. 
Since the action of $G$ on $(B', \tilde g_0)$ is isometric, the action of $\hat G$ on  $F(B', \tilde g_0)$
is free. 
Therefore, $F(B',\tilde g_0)/\hat G$ is a Riemannian manifold, and
so is $Y$. The regularity of the metric follows from \cite{A90,AKKLT,F88,GW:lipschitz,Pets:conv}.


\smallskip
From the construction above, the diameter of $Y$ is clearly uniformly bounded:
\begin{equation} \label{diam-Y}
\textrm{diam}(Y)\leq C(n,D).
\end{equation}
Furthermore, under the additional $\LK$-regular condition
\begin{equation} \label{higher-curvature}
\| \nabla^i R(M) \| \leq \LK, \quad \textrm{for }\, i=1,2,3,
\end{equation}
the lifted metrics $\tilde g_i$ are of $C^{4,\alpha}$ in harmonic coordinates in the construction above (\cite{A90,HH,K93}), which, passing to a subsequence, converge to a metric $\tilde g_0 \in C^{4,\alpha}$ in the $C^4$-topology. Since the induced pseudogroup action $\hat G$ on the orthonormal frame bundle $F(B',\tilde g_0)$ is isometric and free, the Riemannian manifold $Y$ has $C^{4,\alpha}$-metric tensor and in particular, the sectional curvature of $Y$ is defined.

Let ${\frak F}{\frak M}(n,\Lambda,\LK,D)$ denote the set of orthonormal frame bundles $FM$ of closed Riemannian manifolds $M$ in the class ${\frak M}(n,\Lambda,D)$ satisfying the additional $\LK$-regular condition \eqref{higher-curvature}, and let $\overline{{\frak F}{\frak M}}(n,\Lambda,\LK,D)$ denote its closure in the Gromov-Hausdorff topology. 

\begin{lemma} \label{sec-bound-Y}
The sectional curvature of any $Y\in \overline{{\frak F}{\frak M}}(n,\Lambda,\LK,D)$ is uniformly bounded:
\begin{equation} \label{curvatureY}
|R(Y)| \leq C(n, \Lambda,\LK,D).
\end{equation}
\end{lemma}

\begin{proof}
The lower curvature bound is immediate: $Y$ has curvature bounded from below in the sense of Alexandrov (e.g. \cite[Proposition 10.7.1]{BBI}).
The argument for the upper sectional curvature bound is due to \cite{F88}.
Suppose that the sectional curvature is unbounded at $y_k\in Y_k$, then \cite[Lemma 7.8]{F88} implies that the isotropy group $G_{y_0}$ at $\lim\limits_{k\to\infty} y_k=y_0\in Y_0=\lim\limits_{k\to \infty} Y_k$ has positive dimension. However, the (induced) pseudo-group action $\hat G$ on the orthonormal frame bundle is free, so the isotropy group is trivial, which is a contradiction.
Note that \cite[Lemma 7.8]{F88} relies on the sectional curvature estimate in \cite[Lemma 7.2]{F88}, and the latter requires that $Y$ has $C^4$-smooth Riemannian metric.
\end{proof}

The next lemma concerns the relation between the volumes of $X$ and $Y$.

\begin{lemma} \label{volumebound-Y}
Let $X \in \overline{{\frak M}}(n,\Lambda,D)$ and $Y \in \overline{{\frak F}{\frak M}}(n,\Lambda,D)$ be defined as in Theorem \ref{thm-Y}. Then we have the following identity:
$${\rm Vol}_{d+\dim(O(n))} (Y) = {\rm Vol}(O(n)) \, {\rm Vol}_d(X).$$
\end{lemma}

\begin{proof}


Let $M_i\in {\frak M}(n,\Lambda,D)$ be a sequence of manifolds converging to $X$, and, passing to a subsequence, the orthonormal frame bundles $FM_i\in {\frak F}{\frak M}(n,\Lambda,D)$ converges to $Y$.
We have fibrations
$f_i:M_i\to X$ and $\tilde f_i:FM_i\to Y$ such that 
$f_i\circ \pi_i=\pi\circ \tilde f_i$, where $\pi_i:FM_i\to M_i$ is the projection. On the other hand, for every $x\in X$, as $i\to\infty$,
$(f_i\circ \pi_i)^{-1}(x)$ and $(\pi\circ \tilde f_i)^{-1}(x)$
converge to $O(n)$ and $\pi^{-1}(x)$ respectively.
Thus, $\pi^{-1}(x)$ is isometric to $O(n)$. $\pi:Y\to X=Y/O(n)$ is the projection along 
$O(n)$-orbits, and the distance of $X$ is the orbit 
distance, $d_X(x,x')=d_Y(\pi^{-1}(x),\pi^{-1}(x'))$ for all
$x,x'\in X$.
Thus, $|d\pi(v)|=1$ if $v$ is normal to the $O(n)$-orbit. 
Then the claim follows from the co-area formula.
\end{proof}

In addition, we assume a lower bound for the volume of $X$,
\begin{equation} \label{volume-lower-X}
\textrm{Vol}_d(X)\geq v_0, \quad d=\dim(X).
\end{equation}
Thus, Lemma \ref{volumebound-Y} gives us a lower bound on the volume of $Y$.
Then by Cheeger's injectivity radius estimate (\cite{C2,HK}) and Lemma \ref{sec-bound-Y}, we have
\begin{equation} \label{inj-bound}
\textrm{inj}(Y)\geq i_0=i_0(n,\Lambda,\LK,D,\textrm{dim}(X),v_0).
\end{equation}

\subsection{Density functions and the weighted Laplacian}

In \cite{F87}, Fukaya introduced a topology on the set of compact metric spaces equipped with Borel measures, called the measured Gromov-Hausdorff topology. A sequence of compact metric-measure spaces $(X_{i},\mu_i)$ with $\mu_i(X_i)=1$ is said to converge to $(X,\mu)$ in the measured Gromov-Hausdorff topology, if $X_i$ converges to $X$ in the Gromov-Hausdorff topology and, in addition, the measure $(\psi_i)_{\ast}(\mu_i)$ converges to $\mu$ in the weak$^*$ topology, where $\psi_i:X_i\to X$ is an $\varepsilon$-Gromov-Hausdorff approximation.
When $X_i$ is a Riemannian manifold, the measure $\mu_i$ is taken to be the normalized Riemannian measure \eqref{normalized-measure}.

We denote by ${\frak M}{\frak M} (n,\Lambda,D)$ the set of closed Riemannian manifolds $M$ in the class ${\frak M} (n,\Lambda,D)$ equipped with their normalized Riemannian measure $\mu_M$, and by $\overline{{\frak M}{\frak M}}(n,\Lambda,D)$ its closure with respect to the measured Gromov-Hausdorff topology.
Similarly for the orthonormal frame bundles, we denote by ${\frak F}{\frak M}{\frak M}(n,\Lambda,D)$ the set of orthonormal frame bundles $FM$ of closed Riemannian manifolds $M\in {\frak M} (n,\Lambda,D)$ equipped with the normalized Riemannian measure on $FM$, and denote by $\overline{{\frak F}{\frak M}{\frak M}}(n,\Lambda,D)$ its closure in the measured Gromov-Hausdorff topology.

\smallskip
Let $(X,\mu) \in \overline{{\frak M}{\frak M}}(n,\Lambda,D)$.
By \cite{F87,K93,KLLY}, there exists a density function $\rho_X\in C^2_{\ast}(X^{reg})\subset C^{1,\alpha}(X^{reg})$, for any $0<\alpha<1$, such that 
$$d\mu = \rho_X \frac{dV_X}{\textrm{Vol}(X)},$$
where $dV_X$ is the Riemannian volume element on the regular part $X^{reg}$ of $X$. 
Note that $\rho_X$ can be zero at non-orbifold point of $X$, see \cite[Theorem 0.6]{F87}.
For $(Y,\tilde\mu) \in \overline{{\frak F}{\frak M}{\frak M}}(n,\Lambda,D)$, since $Y$ is always a Riemannian manifold due to Theorem \ref{thm-Y}, there exists a strictly positive density function $\rho_Y\in C^2_{\ast}(Y)\subset C^{1,\alpha}(Y)$ on the Riemannian manifold $Y$ such that 
$$d\widetilde{\mu} = \rho_Y \frac{dV_Y}{\textrm{Vol}(Y)}.$$ 

\begin{lemma}[Lemma 1.9(i) in \cite{K93}]\label{bound-Y-Lip}
Let $(Y,\widetilde{\mu}) \in \overline{{\frak F}{\frak M}{\frak M}}(n,\Lambda,D)$ and $\rho_Y$ be the density function on $Y$. Then
$$\|\log \rho_Y\|_{C^0(Y)} \leq C(n,\Lambda,D,\dim(Y),i_0),$$
$$\|\rho_Y\|_{C^{0,1}(Y)} \leq C(n,\Lambda,D,\dim(Y),i_0),$$
where $i_0$ is the injectivity radius bound for $Y$ given by \eqref{inj-bound}.
\end{lemma}

\smallskip
Let $\pi:Y\to X=Y/O(n)$ be the natural projection. Recall \cite[Section 3]{KLLY} that
\begin{equation} \label{measure-iso}
\pi_{\ast}(\widetilde{\mu})=\mu,
\end{equation}
which implies that 
$$\pi_{\ast}: L^2_O (Y,\widetilde{\mu})\to L^2(X,\mu)$$
is an isometry, where $L^2_O (Y,\widetilde{\mu})$ is the $O(n)$-invariant subspace of $L^2 (Y,\widetilde{\mu})$.
Thus for any $u\in L^2(X,\mu)$ and any open subset $U\subset X$, we have
\begin{equation} \label{norm-iso}
\|u\|_{L^2(U,\mu)}=\|u\circ \pi\|_{L^2(\pi^{-1}(U),\widetilde{\mu})}.
\end{equation}

\smallskip
The weighted Laplacian $\Delta_X$ in a local coordinate $(x^1,\cdots,x^d)$ of $X^{reg}$ has the form
\begin{eqnarray} \label{DeltaX}
\Delta_X u &=&\frac{1}{\rho_X |g|^{1/2}}\frac{\partial}{\partial x^j}\Big( \rho_X |g|^{1/2}g^{jk} \frac{\partial}{\partial x^k}u \Big) \nonumber \\
&=& g^{jk} \frac{\partial^2}{\partial x^j \partial x^k} u + \textrm{lower order terms}.
\end{eqnarray}
The weighted Laplacian $\Delta_Y$ on $Y$ has the same form with $\rho_X$ replaced by $\rho_Y$. By \cite[Appendix A]{KLLY}, the $L^2(Y,\widetilde{\mu})$-normalized eigenfunctions of $\Delta_Y$ consist of $O(n)$-invariant part, denoted by $\phi_j^O$, and the orthogonal complement, denoted by $\phi_j^{\perp}$. Moreover, $\phi_j=\pi_{\ast}(\phi_j^O)=\phi_j^O \circ\pi^{-1}$ is an $L^2(X,\mu)$-normalized eigenfunction of $\Delta_X$ due to \eqref{measure-iso}, and their corresponding eigenvalues coincide: $\lambda_j=\lambda_j^O$. As a consequence, the eigenfunctions $\{\phi_j\}_{j=1}^{\infty}$ of $\Delta_X$ form an orthonormal basis of $L^2(X,\mu)$.

\section{Quantitative unique continuation} \label{sec-UC}

We consider the quantitative unique continuation on the Riemannian manifold $Y$ stated in Theorem \ref{thm-Y}, with observation in an open subset $V \subset Y$, where the Laplacian is weighted with the density function $\rho_Y$ as discussed in the previous section. Recall that $\rho_Y\in  C^{1,\alpha}(Y)$ and $\rho_Y>0$. Furthermore the density function $\rho_Y$ has uniformly bounded $C^{0,1}$-norm by Lemma \ref{bound-Y-Lip}. 

Consider the wave equation with the weighted Laplacian $\Delta_Y$ on $Y$,
\begin{equation} \label{wave-Y}
(\partial_t^2-\Delta_Y) u= f,
\end{equation}
where $\Delta_Y$ is given by \eqref{DeltaX} with $\rho_X$ replaced by $\rho_Y$.
In view of the bounded geometric parameters \eqref{diam-Y}-\eqref{inj-bound} on $Y$, we apply the stable unique continuation theorem in \cite{BKL2,BKL3,BILL,DLLO}. 
This is applicable because $\Delta_Y$ in local coordinates has the same principle term as the Laplace-Beltrami operator, and the density $\rho_Y$ only appears in lower order terms with its $C^1$-norm uniformly bounded.
However, before we can apply those mentioned results, we need to relax the curvature assumptions in those results where higher order curvature bounds were assumed. In the simple case of manifolds without boundary, one can apply results in \cite{Ba,BMR} to smoothen the Riemannian metric.


\begin{proof}[Proof of Theorem \ref{uc-Y}]
This is an improvement of \cite{BKL3,BILL} for manifolds without boundary in terms of regularity assumptions by removing the assumptions on the higher order curvature bounds in those results. The idea is to construct the desired non-characteristic domains using distance functions associated to smoothened Riemannian metrics. Recall from \cite{Ba, BMR} that for sufficiently small $\epsilon$ depending only on $n,\Lambda$, there exist smooth Riemannian metrics $g_{\epsilon}$ on $Y$ such that
\begin{align}
&\|g-g_{\epsilon}\|_{C^0} \leq \epsilon, \label{metric-smoothen} \\
&|R(Y,g_{\epsilon})| \leq 2\Lambda^2, \label{curvature-smoothen} \\
&\|\nabla^k R(Y,g_{\epsilon})\| \leq C(n,\Lambda) (k+1)! \, \epsilon^{-\frac{k}{2}} ,\quad \forall \, k\in \N. \label{curvature-higher-smoothen} 
\end{align}
The metrics $g_{\epsilon}$ are constructed as the solution of the Ricci flow with the initial metric $g$.
The sectional curvature bound \eqref{curvature-smoothen} is due to the maximum principle.
We denote by $d_{\epsilon}$ the distance function of $(Y,g_{\epsilon})$. 
The conditions $\textrm{Vol}_n(Y,g)\geq c_0$ and \eqref{metric-smoothen} imply that $\textrm{Vol}_n(Y,g_{\epsilon}) \geq c_0/2$ for sufficiently small $\epsilon$.
Then Cheeger's injectivity radius estimate and \eqref{curvature-smoothen} yield
\begin{equation} \label{inj-smoothen}
\textrm{inj}(Y,g_{\epsilon}) \geq i_1=i_1(n,\Lambda,D,c_0).
\end{equation}
This shows that the distance function $d_{\epsilon}(\cdot,y)$ of $(Y,g_{\epsilon})$ is smooth on the ball $B^{\epsilon}(y,i_1)\setminus \{y\}$, where $B^{\epsilon}$ denotes a ball of $(Y,g_{\epsilon})$. Note that the two distance functions $d_Y,d_{\epsilon}$ are close to each other with respect to the Lipschitz distance: $d_L\big((Y,g), (Y,g_{\epsilon}) \big)<\epsilon$, which implies that $B(x,i_1/2)\subset B^{\epsilon}(x,i_1)$ for sufficiently small $\epsilon$.

\smallskip
To obtain the bounds for higher order derivatives of $d_{\epsilon}(\cdot,x)$, we need to work in harmonic coordinates instead of geodesic normal coordinates used in \cite{BILL}, since we do not assume higher order curvature bounds for $(Y,g)$.
For closed manifolds, the $C^{1,\alpha}$-harmonic radius can be explicitly estimated in terms of the sectional curvature bound, see \cite{JK}.
Namely, there exists a constant $r_h=r_h(n,\Lambda,i_1)$, such that one can find harmonic coordinates in a ball of radius $r_h/2$ on $(Y,g)$ with the property that
\begin{equation} \label{metric-1alpha}
\frac{1}{2}|v|^2\leq \sum_{i,j=1}^{n} g^{ij} v_{i} v_{j} \leq 2|v|^2\; (v\in\mathbb{R}^{n}), \quad\;
\|g_{ij}\|_{C^{1,\alpha}}\leq C(n,\Lambda,r_h,\alpha),
\end{equation}
for any $\alpha\in (0,1)$.  
Let us fix a point $y\in Y$ and a harmonic coordinate $(x^1,\cdots,x^n)$ on $(Y,g)$ in a ball of radius $r_h/2$ centered at $y$.
Recall that $g_{\epsilon}$ is the solution of the Ricci flow with initial metric $g$.
Then in this harmonic coordinate, the Ricci flow equation, together with curvature bounds \eqref{curvature-higher-smoothen}, and an interpolation argument yield
\begin{equation} \label{metric-smooth-1alpha}
\|(g_{\epsilon})_{ij}\|_{C^{1,\alpha'}}\leq C(n,\Lambda,r_h,\alpha,\alpha'),\quad \forall\, \alpha'\in (0,\alpha).
\end{equation}

In the harmonic coordinate $(x^1,\cdots,x^n)$ for $(Y,g)$ on $B(y,r_h/2)$, 
the covariant derivative $\nabla^{\epsilon}$ has the following form (e.g. Chapter 2 in \cite{PP}, p32),
\begin{equation}\label{Hessian-local}
\big((\nabla^{\epsilon})^2 d_{\epsilon}(\cdot, y)\big)(\frac{\partial}{\partial x^k},\frac{\partial}{\partial x^l})=\frac{\partial^2}{\partial x^k \partial x^l} d_{\epsilon}(\cdot,y)-\sum_{i=1}^n (\Gamma^{\epsilon})_{kl}^i \frac{\partial}{\partial x^i} d_{\epsilon}(\cdot,y), \quad k,l=1,\cdots,n.
\end{equation}
Note that $d_{\epsilon}(\cdot,y)$ is smooth on $B(y,r_h/2)\setminus \{y\}$ in the harmonic coordinate $(x^1,\cdots,x^n)$ for $(Y,g)$ since $g_{\epsilon}$ is smooth in this coordinate.
We know that $|\nabla^{\epsilon} d_{\epsilon}(\cdot,y)|=1$, and 
\begin{equation} \label{Riccati-C3}
\|(\nabla^{\epsilon})^2 d_{\epsilon}(\cdot,y)\| \leq 4 r^{-1},\quad \|(\nabla^{\epsilon})^3 d_{\epsilon}(\cdot,y)\| \leq C(n,\|R(Y,g_{\epsilon})\|_{C^1})r^{-2}
\end{equation}
on $B^{\epsilon}(y,r)-B^{\epsilon}(y,r/2)$ for sufficiently small $r$ depending only on $n,\Lambda,i_1$ due to \eqref{curvature-smoothen}. 
The bound on the second covariant derivative is due to the Hessian comparison theorem.
Then the bound on the third covariant derivative follows by differentiating the Riccati equation in polar coordinates, see e.g. \cite{HV}.
Thus, in the $C^{1,\alpha}$-harmonic coordinate $(x^1,\cdots,x^n)$ of $(Y,g)$ chosen above, \eqref{Riccati-C3}, \eqref{metric-smooth-1alpha} and \eqref{curvature-higher-smoothen} give the bound for the $C^{2,\alpha'}$-norm of $d_{\epsilon}(\cdot,y)$ in this harmonic coordinate for any $\alpha'\in (0,\alpha)$:
\begin{equation} \label{norm-d-alpha}
\|d_{\epsilon}(\cdot,y)\|_{C^{2,\alpha'}} \leq C(n,\Lambda,r_h,\alpha,\alpha',\epsilon)r^{-2}\;\; \textrm{ in } B^{\epsilon}(y,r)-B^{\epsilon}(y,r/2), 
\end{equation}
for sufficiently small $r<r_h/4$ depending on $\Lambda,i_1$.
Note that the sectional curvature bound \eqref{curvature-smoothen} independent of $\epsilon$ is crucial, as this yields the uniform injectivity radius bound \eqref{inj-smoothen} and a uniform choice of $r$ above.

\smallskip
Now we can use the method in \cite[Theorem 3.1]{BILL} adapted for the simpler case of the closed manifold $(Y,g)$ in $C^{1,\alpha}$-harmonic coordinate charts, and then apply \cite[Theorem 1.2]{BKL2}. 
For closed manifolds, the proof is much less technical as the distance function $d_{\epsilon}$ is already smooth within half the injectivity radius.
Set
$$0<h<\min\big\{ \frac{1}{10},\frac{T}{8}, \frac{i_1}{10},\frac{r_h}{10},\frac{\pi}{12\Lambda} \big\}.$$
Let $\{z_j\}$ be an $h$-net in $\{x\in V: d_Y(x,\partial V)\geq 2h\}$, and $\rho_0:=\min\{i_1/2,r_h/2,\pi/6\Lambda\}$. If $T\leq \rho_0$,
we define the following $C^{2,\alpha'}$ function
\begin{equation}\label{psi}
\psi_j(x,t)=\big(T-d_{\epsilon} (x,z_{j})\big)^2-t^2,
\end{equation}
and the domain
$$\Omega_j^0:=\{(x,t)\in Y\times [-T,T]: \psi_j(x,t)>8T^2\epsilon\}-\{x: d_{\epsilon}(x,z_j)\leq \frac{h}{2}\}\times [-T,T].$$
For each $z_j$, we fix a $C^{1,\alpha}$-harmonic coordinate chart $(x^1,\cdots,x^n)$ of $(Y,g)$ in the ball $B(z_j,\rho_0)$. By the discussions above $d_{\epsilon}(\cdot,z_j)$ is smooth in this coordinate, and \eqref{norm-d-alpha} gives a bound on its $C^{2,\alpha'}$-norm (with $r=h$) independent of $z_j$.
Hence we see that the $C^{2,\alpha'}$-norm of $\psi_j$ is uniformly bounded in the harmonic coordinate chosen above in $\Omega_j^0$ (depending on $\epsilon,h$).
Next, we show that $\psi_j$ is a non-characteristic function in $\Omega_j^0$.
We take the gradient of $\psi_j$ with respect to the space variable in the harmonic coordinate $(x^1,\cdots,x^n)$ of $(Y,g)$ chosen above,
$$\nabla_x \psi_j = 2 \big( T-d_{\epsilon}(x,z_j) \big) \nabla d_{\epsilon}(x,z_j).$$
Using \eqref{metric-smoothen} in the harmonic coordinate satisfying \eqref{metric-1alpha}, for sufficiently small $\epsilon$, we have
\begin{eqnarray*}
|\nabla d_{\epsilon}(x,z_j)| &=& g^{kl} (\partial_{x^k} d_{\epsilon}) (\partial_{x^l} d_{\epsilon}) \\
&\geq& (g_{\epsilon})^{kl} (\partial_{x^k} d_{\epsilon}) (\partial_{x^l} d_{\epsilon})-\epsilon (\partial_{x^k} d_{\epsilon}) (\partial_{x^l} d_{\epsilon}) \\
&\geq& (1-3\epsilon) |\nabla^{\epsilon} d_{\epsilon}(x,z_j)|^2=1-3\epsilon,
\end{eqnarray*}
where $\nabla^{\epsilon}$ is the gradient with respect to the metric $g_{\epsilon}$.
Thus in $\Omega_j^0$, the principle symbol
\begin{eqnarray}\label{ppsi}
p \big((x,t),\nabla \psi_j \big) &=&|\nabla_x \psi_j|^2-|\partial_t \psi_j|^2 \nonumber\\
&\geq& 4\big(T-d_{\epsilon}\big)^2(1-3\epsilon)^2-4t^2 \nonumber \\
&\geq& 4\psi_j(x,t) - 24T^2 \epsilon >8T^2\epsilon,
\end{eqnarray}
where $p(\cdot,\cdot)$ stands for the principle symbol of the wave operator with respect to the original metric $g$.
Note that the only part changed from \cite{BILL} in the construction above is the distance function $d_{\epsilon}$, while the harmonic coordinate and the wave operator considered are still with respect to the original metric $g$.

For the case of $T>\rho_0$,
define a cut-off function $\xi:\mathbb{R}\to\mathbb{R}$ by
\begin{equation}\label{xidef}
\xi(x)=\frac{(h-x)^3}{h^3}, \textrm{ for } x\in [0,h],
\end{equation}
and $\xi(x)=0$ for $x>h$. The function $\xi(x)$ on negative numbers can be defined in any way so that $\xi(x)\geq 1$ for $x<0$, and $\xi(x)$ is smooth on $(-\infty,h)$. 
This function $\xi(x)$ is of $C^{2,1}$ on $\mathbb{R}$ and monotone decreasing on $[0,+\infty)$. 
Define the function
\begin{equation}\label{psi1}
\psi_{1,j}(x,t)=\bigg(\Big(1-\xi \big(\rho_0-d_{\epsilon}(x,z_{j})\big)\Big)T-d_{\epsilon}(x,z_{j})\bigg)^2-t^2.
\end{equation}
This is a function of class $C^{2,\alpha'}$ with uniformly bounded norm \eqref{norm-d-alpha} in the domain
$$\Omega_{1,j}^0:=\{(x,t)\in Y\times [-T,T]: \psi_{1,j}(x,t)>8T^2\epsilon\}-\{x: d_{\epsilon}(x,z_j)\leq \frac{h}{2}\}\times [-T,T].$$
Note that in general, the domain characterized by $\psi_{1,j}(x,t)>8T^2 \epsilon$ has two connected components. Here we define $\Omega_{1,j}^0$ to be the connected component characterized by $\big(1-\xi(\rho_0-d_{\epsilon}(x,z_j)) \big)T-d_{\epsilon}(x,z_j)>0$. Thus, it satisfies that $d_{\epsilon}(x,z_j)<\rho_0$ in $\Omega_{1,j}^0$ by the definition of the function $\xi$, and hence $\Omega_{1,j}^0$ is contained in one coordinate chart of $Y$ times $[-T,T]$.

A similar argument as above shows that $\psi_{1,j}$ is non-characteristic on $\Omega_{1,j}^0$. 
Then one can follow the method in \cite{BILL} to construct all subsequent $\psi_{i,j}$ and $\Omega_{i,j}^0,\, \Omega_{i,j}$. 
Note that Theorem 1.2 in \cite{BKL2} is applicable to the present case, since the weighted Laplacian $\Delta_Y$ has coefficients $g^{ij}$ and $\rho_Y$ in $C^1$ with uniformly bounded norms in harmonic coordinates.
In the end, the unique continuation propagates to following domain
$$\Omega_{\epsilon}(h) :=\big\{(x,t)\in (Y \setminus V)\times [-T,T]: T-|t|-d_{\epsilon}(x,V) > \sqrt{\epsilon}+h \big\}.$$
Since the distance functions $d_Y,d_{\epsilon}$ are $\epsilon$-close in the Lipschitz distance,
it follows that
$$\Omega_{\epsilon}(h) \supset \big\{(x,t)\in (Y \setminus V)\times [-T,T]: T-|t|-d_Y(x, V) > \sqrt{\epsilon}+h +D\epsilon \big\}.$$
Thus, fixing any choice of $\alpha,\alpha'$, the theorem follows by choosing $\epsilon=h$.
The constants $C_4,h_0$ explicitly depend on $n,T,\Lambda,D,i_1,r_h,\|\rho_Y\|_{C^1}$, the upper bound for $\textrm{Vol}_n (Y)$, and the curvature bounds in \eqref{curvature-higher-smoothen}; the constant $C_5$ explicitly depends on $n$, see Appendix A in \cite{BILL}. 
The volume $\textrm{Vol}_n(Y)$ is bounded above by $n,\Lambda,D$.
The curvature bounds in \eqref{curvature-higher-smoothen} depend on $\epsilon$ only polynomially, which can be absorbed into the exponential term $\exp(h^{-C_5})$.
\end{proof}

By the Sobolev embedding theorem and \cite[Proposition 3.10]{BILL}, Theorem \ref{uc-Y} yields the stable unique continuation on the whole domain of influence.

\begin{proposition} \label{uc-whole}
Let $(Y,g)$ be a connected, closed, smooth manifold with $C^2$-smooth Riemannian metric $g$, satisfying 
$${\rm dim}(Y)=n, \quad |R(Y)|\leq \Lambda^2, \quad  {\rm diam}(Y)\leq D, \quad {\rm Vol}(Y)\geq c_0.$$
Let $\rho_Y\in C^1(Y)$ be a density function on $Y$.
Suppose $u\in H^1(Y\times[-T,T])$ is a solution of the wave equation \eqref{wave-Y} with $f=0$ and $\partial_t u (\cdot,0)=0$. 
Let $V\subset Y$ be a connected open subset with smooth boundary.
If the norms satisfy
$$\|u(\cdot,0)\|_{H^1(Y)}\leq E_1,\quad \|u\|_{H^{1}(V\times [-T,T])}\leq \varepsilon_0,$$
then for $0<h<h_0$, we have
$$\|u(\cdot,0)\|_{L^2(Y(V,T))}\leq C_4 h^{-\frac{2}{9}}\exp(h^{-C_5})\frac{E_1}{\Big(\log \big(1+h\frac{E_1}{\varepsilon_0}\big)\Big) ^{\frac{1}{6}}} +C_6 E_1 h^{\frac{1}{3(n+1)}},$$
where the domain of influence $Y(V,\tau):= \{y\in Y: d_Y(y,V) <\tau\}$. The constant $C_6$ explicitly depends on $n,\Lambda,D,c_0,{\rm Vol}_{n-1}(\partial V), i_b(\overline{V})$, where $i_b(\overline{V})$ denotes the boundary injectivity radius of $\overline{V}$.
Furthermore, if we additionally assume $T\geq r_1$ for some $r_1>0$, then $C_6$ can be chosen explicitly depending only on $n,\Lambda,D,c_0,r_1$ (independent of $V$).
\end{proposition}

\begin{proof}
This is a straightforward corollary of Theorem \ref{uc-Y} and the Sobolev embedding, due to the uniform boundedness of the Hausdorff measure of $\partial Y(V,\tau)$ for any $\tau>0$.
The latter was essentially proved in \cite[Proposition 3.10]{BILL}.
The dependence on $\textrm{Vol}_{n-1}(\partial V),i_b(\overline{V})$ is brought in when $T$ is small. 
If $T\geq r_1$, then the Hausdorff measure of $\partial Y(V,\tau)$ does not depend on $V$.
Note that Proposition 3.10 in \cite{BILL} assumes $C^1$ curvature bound in the case of manifolds with boundary. For the present case of manifolds without boundary, the sectional curvature bound is enough.
\end{proof}

Later in the next section, we will apply this quantitative unique continuation result to Riemannian manifolds 
$Y \in \overline{{\frak F}{\frak M}{\frak M}}(n,\Lambda,\LK,D)$. 
From discussions in Section \ref{section-prelimi}, we know that $Y\in \overline{{\frak F}{\frak M}{\frak M}}(n,\Lambda,\LK,D)$ is a smooth manifold with $C^4$-smooth Riemannian metric tensor and the sectional curvatures are uniformly bounded \eqref{curvatureY}.
Note that the constants $C_4,C_6$ depend on the lower bound for the volume of $Y$, which is given by Lemma \ref{volumebound-Y} and \eqref{volume-lower-X}.

\begin{remark} \label{norm-equiv}
On a Riemannian manifold $(Y,\widetilde{\mu}) \in \overline{{\frak F}{\frak M}{\frak M}}(n,\Lambda,D)$, we have two measures: $\widetilde{\mu}$ and the Riemannian volume element $dV_Y$. 
Since $\|\rho_Y\|_{C^0}$ and $\|\rho_Y^{-1}\|_{C^0}$ are both uniformly bounded by Lemma \ref{bound-Y-Lip}, the norms with respect to these two measures are equivalent, considering the fact that $\textrm{Vol}(Y)$ is bounded from below by Lemma \ref{volumebound-Y} and clearly from above.
From now on, the norms on $Y$ are taken with respect to $\widetilde{\mu}$.
\end{remark}

\section{Determination of Fourier coefficients}
\label{sec-coefficients}

Let $(X,\mu) \in \overline{{\frak M}{\frak M}}(n,\Lambda,\LK,D)$ and $p\in X^{reg}$. Suppose ${\rm Vol}_{d}(X) \geq v_0$ where $d={\rm dim}(X)$.
Assume that the interior spectral data $\{\lambda_j,\phi_j|_{B(p,r_0)}\}_{j=0}^J$ for the weighted Laplacian $\Delta_X$ are given on a ball $B(p,r_0) \subset X^{reg}$.
In this subsection, let $0<\eta< \min\{i_0/2,r_0/4\}$ be a fixed small number such that $\eta<\inf_{x\in B(p,r_0/2)}$ {\rm inj}(x), where $i_0$ is the injectivity radius bound for $Y$ given by \eqref{inj-bound}, and ${\rm inj}(x)$ for $x\in X^{reg}$ is the injectivity radius at $x$ in the Riemannian sense. 
Here $(Y,\widetilde{\mu}) \in \overline{{\frak F}{\frak M}{\frak M}}(n,\Lambda,\LK,D)$ is the Riemannian manifold stated in Theorem \ref{thm-Y} and $\pi:Y\to X=Y/O(n)$ is the projection.
We choose open subsets of $B(p,r_0)$ in the following way. 
Let $\{p_i\}_{i=1}^N$ be a maximal $\eta/2$-separated set in $B(p,r_0/2)$.
Let $\{U_i\}_{i=1}^N$ be disjoint open subsets containing $p_i$ in $B(p,r_0)$ satisfying
\begin{equation} \label{partition}
B(p,\frac{r_0}{4}) \subset \bigcup_{i=1}^N \overline{U_i} \subset B(p,\frac{r_0}{2}), \quad \textrm{diam}(U_i)\leq \eta.
\end{equation}
Assume that every $U_i$ contains a ball of radius $\eta/4$. 
Without loss of generality, we assume that every $\partial U_i$ is smoothly embedded in $X$ and admits a boundary normal neighborhood of width $\eta/10$. 
This is because one always has the choice to propagate the unique continuation from the disjoint balls $B(p_i,\eta/4)$. An error of order $\eta$ does not affect the final result. 

Note that with the choice of $U_i$ above, the total number $N$ is bounded above. 
This is because $\{B(p_i,\eta/4)\}$, and hence $\{\pi^{-1}(B(p_i,\eta/4))\}$, do not intersect, and each $\{\pi^{-1}(B(p_i,\eta/4))\}$ contains a ball of radius $\eta/4$ in $Y$, where $\pi:Y\to X$ is the natural projection.
Thus the total number of balls is bounded by geometric parameters including the volume of $Y$, which in turn is bounded above by $n,\textrm{dim}(Y),|R(Y)|,D$, i.e.,
\begin{equation} \label{bound-N}
N\leq C(n,\textrm{dim}(Y),\Lambda,\LK,D) \eta^{-\textrm{dim}(Y)}.
\end{equation}
Recall that the bound for the sectional curvature $R(Y)$ of $Y$ is due to Lemma \ref{sec-bound-Y}.
Here we have used the condition that $\eta<i_0/2$, so that the volume of balls of radius $\eta/4$ in $Y$ is bounded below by $C\eta^{\textrm{dim}(Y)}$.

\smallskip
Let $\alpha=(\alpha_1,\cdots,\alpha_N)$ with $\alpha_k\in [\eta, D]$ be a multi-index. We define the domain of influence associated with $\alpha$ by
\begin{equation}\label{Xalpha}
X_{\alpha}:=\bigcup_{k=1}^N X(U_k, \alpha_k)=\bigcup_{k=1}^N \big\{ x \in X: d_X(x,U_k) < \alpha_k \big\}.
\end{equation}
Denote by $\mathcal{V}_J \subset C^{0,1}(X)$ the function space spanned by the first $J$ orthonormalized eigenfunctions $\{\phi_j\}_{j=0}^J$.
Given $E_1,\varepsilon_1>0$, we define
\begin{equation}\label{Udef}
\mathcal{U}(E_1,\varepsilon_1):=\bigcap_{k=1}^ N \big\{v\in \mathcal{V}_{J}: \|v\circ  \pi\|_{H^2(Y)}\leq E_1,\; \|\widetilde{W}(v \circ \pi)\|_{L^2(\pi^{-1}(U_{k})\times [-\alpha_k,\alpha_k])}\leq \varepsilon_1 \big\},
\end{equation}
where $\widetilde{W}(\widetilde{v})$ is the solution of the wave equation
\begin{equation}
(\partial_t^2 -\Delta_Y) \widetilde{W}=0 \;\, \textrm{ on }Y,
\end{equation}
with the initial conditions $\widetilde{W}|_{t=0}=\widetilde{v}$ and $\partial_t \widetilde{W}|_{t=0}=0$.

\begin{remark}
The solution operator of the wave equation commutes with $\pi$. Namely, define $W(v)$ to be the solution of the wave equation $(\partial_t^2 -\Delta_X) W=0$ on $X$ with the initial conditions $W|_{t=0}=v$ and $\partial_t W|_{t=0}=0$. Then 
$$W(v)\circ \pi=\widetilde{W}(v\circ \pi).$$ 
This is because $W(v)\circ \pi$ satisfies the initial conditions and the wave equation on $Y$, see \cite[Appendix A]{KLLY}, and the claim follows from the uniqueness of solution.
\end{remark}

\begin{remark} \label{remark-computable}
The conditions in \eqref{Udef} can be computed using the given interior spectral data on $B(p,r_0)$. Namely, if $v=\sum_{j=0}^J v_j \phi_j$, then $v\circ \pi=\sum_{j=0}^J v_j \phi_j^O$. 
Note that $\phi_j^O=\phi_j \circ \pi$ are orthonormal with respect to the measure $\widetilde{\mu}$ on $Y$.
Since $\lambda_j=\lambda_j^O$, then $\|v \circ \pi\|_{H^2(Y)}=\sum_{j=0}^{J}(1+\lambda_j^2)v_j^2.$
The wave $\widetilde{W}(v\circ \pi)$ on $Y$ is given by
\begin{equation}\label{waveboundary}
\widetilde{W}(v\circ \pi)(y,t)= \sum_{j=0}^{J} v_j \cos(\sqrt{\lambda_j} t) \phi_j^O (y), \quad y\in Y, \; t\geq 0.
\end{equation}
Hence, due to \eqref{norm-iso},
$$\|\widetilde{W}(v \circ \pi)\|_{L^2(\pi^{-1}(U_{k})\times [-\alpha_k,\alpha_k])}=\big\|\sum_{j=0}^{J} v_j \cos(\sqrt{\lambda_j} t) \phi_j \big\|_{L^2(U_{k}\times [-\alpha_k,\alpha_k])},$$
where the right-hand side can be computed using the given interior spectra data.
\end{remark}

\begin{lemma} \label{uu0}
Let $\widetilde{u}\in H_O^2(Y)$ satisfying
\begin{equation*}\label{aprior}
\|\widetilde{u}\|_{L^2(Y)}=1,\quad \|\widetilde{u}\|_{H^2(Y)}\leq  E_0.
\end{equation*}
Then for sufficiently small $\gamma \in (0,N^{-2})$, we can construct a function $u_0\in H^2(X^{reg})$ such that
$$u_0 |_{X_{\alpha}}=0, \;\;\; u_0 |_{X^c_{\alpha+\gamma}}=\widetilde{u} \circ \pi^{-1}, \;\;\; \|u_0\|_{L^2(X)}\leq  1,$$
\begin{equation}\label{norm-u0}
\|u_0\circ \pi\|_{H^2(Y)}\leq C_0 E_0 \gamma^{-{\rm dim}(Y)-2},
\end{equation}
where $\alpha+\gamma=(\alpha_1+\gamma,\cdots,\alpha_N+\gamma)$. 
\end{lemma}

\begin{proof}
The proof is similar to Lemma 4.1 in \cite{BILL}. To construct the required partition of unity, one can use the projections of partition of unity on $Y$. Let $\{x_j\}$ be a maximal $\gamma/2$-separated set in $X$ for $\gamma<i_0/2$, where $i_0$ is the constant given in \eqref{inj-bound}. Then $\{\pi^{-1}(B(x_j,\gamma/2))\}$ is an open cover of $Y$ due to $d_X(x,y)=d_Y(\pi^{-1}(x),\pi^{-1}(y))$.
We take any $C^2$ partition of unity $\{\widetilde{\chi}_l\}$ subordinate to this open cover of $Y$, and consider the projection to $O(n)$-invariant functions,
$$\chi_l^O:=\mathbb{P}_O \widetilde{\chi}_l.$$
The projection $\mathbb{P}_O$ is defined as
$$(\mathbb{P}_O v)(y)=\frac{1}{\textrm{Vol}(O(n))} \int_{O(n)} v(o(y))d\mathcal{H}_{O(n)}(o), \quad y\in Y,$$
where $d\mathcal{H}_{O(n)}$ denotes the Riemannian measure (or the Haar measure) of $O(n)$.
By \cite[Lemma 4.1]{KLLY}, $\chi_l^O\in C^2_O(Y)$ (i.e., the set of $O(n)$-invariant $C^2$ functions on $Y$), and
$$\sum_l \chi_l^O (y) =\frac{1}{\textrm{Vol}(O(n))} \int_{O(n)} \sum_l \widetilde{\chi}_l(o(y)) d\mathcal{H}_{O(n)}(o)=1.$$
Note that the summation above is finite for every $y\in Y$.
Thus the partition of unity on $X$ subordinate to $\{B(x_j,\gamma/2)\}$ can be taken as 
$$\chi_l:=\chi_l^O\circ \pi^{-1} \in C^{0,1}(X).$$
Then $u_0$ can be defined as
\begin{equation}\label{def-u0-partition}
u_0(x):=\sum_{\textrm{supp}(\chi_l)\cap X_{\alpha}=\emptyset} \chi_l(x) \cdot (\widetilde{u}\circ \pi^{-1})(x)\, ,\quad x\in X.
\end{equation}
The first three conditions claimed are clearly satisfied by construction, considering that $\|\widetilde{u} \circ \pi^{-1} \|_{L^2(X)}=1$ by \eqref{norm-iso}. 
The last condition follows from the given $H^2$-bound and
$\|\chi_l^O\|_{C^2}\leq C\gamma^{-2}$, if we can estimate the number of nonzero $\chi_l$ in \eqref{def-u0-partition}.

It remains to show that the number of nonzero $\chi_l$ in \eqref{def-u0-partition} for any given $x$ is bounded.
In fact, knowing a rough bound on the total number of $\chi_l$ is already enough for our purpose.
Since $\{\pi^{-1}(B(x_j,\gamma/4))\}$ do not intersect and each $\pi^{-1}(B(x_j,\gamma/4))$ contains a ball of radius $\gamma/4$ in $Y$, the total number of balls is bounded above by $C\,\textrm{Vol}(Y) \gamma^{-\textrm{dim}(Y)} $, which is bounded by $n,\textrm{dim}(Y),|R(Y)|,D,\gamma$. Here we have used the condition that $\gamma<i_0/2$ in the same way as \eqref{bound-N}. Compared with \cite[Lemma 4.1]{BILL}, the bound we obtained here for the number of nonzero $\chi_l$ in the sum \eqref{def-u0-partition} (for any given point $x$) depends on $\gamma$, which results in a much rougher $H^2$-bound. 
\end{proof}

\smallskip
Let $\widetilde{u}\in H_O^2(Y)$, and $u_0$ be constructed as in Lemma \ref{uu0}. Suppose
\begin{equation}
u_0=\sum_{j=0}^{\infty} d_j \phi_j,
\end{equation}
where $\phi_j$ are the orthonormalized eigenfunctions of $\Delta_X$.
We consider
\begin{equation}
\widetilde{u}_0:= u_0\circ \pi=\sum_{j=0}^{\infty} d_j \phi_j^O,
\end{equation}
where $\phi_j^O$ are the eigenfunctions of the $O(n)$-invariant component $\Delta_O$ of $\Delta_Y$. 
Then we define 
\begin{equation}
u_J:=\sum_{j=0}^{J} d_j \phi_j,\quad \widetilde{u}_J:=u_J\circ \pi=\sum_{j=0}^{J} d_j \phi_j^O.
\end{equation}

\begin{lemma} \label{uJ}
Let $E_1=C_0 E_0 \gamma^{-{\rm dim}(Y)-2}$. For any $\varepsilon_1>0$, there exists sufficiently large $J$ such that $u_J\in \mathcal{U}(E_1,\varepsilon_1)$, where the set $\mathcal{U}(E_1,\varepsilon_1)$ is defined in \eqref{Udef}.
\end{lemma}

\begin{proof}
The first norm condition in \eqref{Udef} is satisfied due to \eqref{norm-u0}.
The second norm condition is a consequence of Lemma \ref{uu0}, finite speed propagation of the wave equation on $Y$, and the energy estimates, see \cite[Lemma 4.2]{BILL}. 
Due to Theorem \ref{thm-Y} and $d_X(x,y)=d_Y(\pi^{-1}(x),\pi^{-1}(y))$, the domains of influence on $X$ and $Y$ match. More precisely, define
\begin{equation}
Y(\widetilde{U},\tau):= \{y\in Y: d_Y(y,\widetilde{U}) <\tau\}.
\end{equation}
Then
\begin{equation}\label{domain-influence-XY}
Y(\pi^{-1}(U),\tau)=\pi^{-1}\big( X(U,\tau) \big).
\end{equation}

We prove \eqref{domain-influence-XY} as follows. Let $y\in Y$ such that $d_Y(y, \pi^{-1}(U))<\tau$. Let $x=\pi(y)$. Then $d_Y(\pi^{-1}(x), \pi^{-1}(U))<\tau$ by definition, which yields that $d_X(x,U)<\tau$. This shows that $Y(\pi^{-1}(U),\tau) \subset \pi^{-1}\big( X(U,\tau) \big)$. 
On the other hand, let $z=\pi^{-1}(x)$ such that $d_X(x,U)<\tau$, which gives $d_Y(\pi^{-1}(x), \pi^{-1}(U))<\tau$. Thus there exists $z_0\in \pi^{-1}(x)$ such that $d_Y(z_0, \pi^{-1}(U))<\tau$. 
Since $X=Y/O(n)$, there exists $o\in O(n)$ such that $o(z)=z_0$.
Since $O(n)$ acts isometrically on $Y$, we have
$d_Y(z,\pi^{-1}(U))=d_Y(z,o^{-1}(\pi^{-1}(U)))<\tau$. 
This proves the other direction $\pi^{-1}\big( X(U,\tau) \big) \subset Y(\pi^{-1}(U),\tau)$.
\end{proof}

\begin{proposition}\label{projection-Y}
Let $\widetilde{u}\in H_O^2(Y)$ satisfying
\begin{equation*}\label{aprior}
\|\widetilde{u}\|_{L^2(Y)}=1,\quad \|\widetilde{u}\|_{H^2(Y)}\leq  E_0.
\end{equation*}
Set $u=\widetilde{u}\circ \pi^{-1} \in L^2(X)$.
Let $\alpha=(\alpha_1,\cdots,\alpha_N)$, $\alpha_k\in [\eta, D]$ be given, and $X_{\alpha}$ be defined in \eqref{Xalpha}. 
Then for any $\varepsilon>0$, there exists sufficiently large $J$,
such that by only knowing the first $J$ interior spectral data $\{\lambda_j,\phi_j|_{B(p,r_0)}\}_{j=0}^J$ of $X$ and the first $J$ Fourier coefficients $\{a_j\}_{j=0}^J$ of $u$, we can find $\{b_j\}_{j=0}^{J}$ and $u^a=\sum_{j=0}^{J}b_j \phi_j$, such that
$$\|u^a-\chi_{_{X_{\alpha}}}u\|_{L^2(X,\mu)} <\varepsilon,$$
where $\chi$ denotes the characteristic function.
\end{proposition}

\begin{proof}

We consider the following minimization problem in $\mathcal{U}$ defined in (\ref{Udef}).
Let $u_{min}\in\mathcal{U}$ be the solution of the minimization problem
\begin{equation}\label{minimization}
\|u_{min}-u\|_{L^2(X)}=\min_{w\in \mathcal{U}(E_1,\varepsilon_1)} \|w-u\|_{L^2(X)},
\end{equation}
with parameters $E_1,\varepsilon_1$ to be determined.
Since $w\in\mathcal{U}$, the norm conditions in \eqref{Udef} yield the following by interpolation and energy estimate:
\begin{eqnarray*}
\|\widetilde{W}(w \circ \pi)\|^2_{H^1(\pi^{-1}(U_{k})\times [-\alpha_k,\alpha_k])} &\leq& \|\widetilde{W}(w \circ \pi)\|_{L^2(\pi^{-1}(U_{k})\times [-\alpha_k,\alpha_k])} \|\widetilde{W}(w \circ \pi)\|_{H^2(Y\times [-\alpha_k,\alpha_k])} \\
&\leq& \varepsilon_1 \cdot C(D) \|w\circ \pi\|_{H^2(Y)} \leq C(D) E_1 \varepsilon_1.
\end{eqnarray*}
Denote $\widetilde{U}_k=\pi^{-1}(U_k)$.
Then we apply our quantitative unique continuation result (Proposition \ref{uc-whole}), in view of the bounded geometric parameters \eqref{diam-Y}-\eqref{inj-bound} for $(Y,\widetilde{\mu}) \in \overline{{\frak F}{\frak M}{\frak M}}(n,\Lambda,\LK,D)$ and Lemma \ref{volumebound-Y}.
By Lemma \ref{bound-Y-Lip}, Remark \ref{norm-equiv}, \eqref{domain-influence-XY} and \eqref{norm-iso}, we have
$$\|w\|_{L^2(X(U_k,\alpha_k))} = \|w\circ \pi\|_{L^2(Y(\widetilde{U}_k,\alpha_k))}< \varepsilon_2=\varepsilon_2(h,E_1,\eta,\gamma,\varepsilon_1),$$ 
for all $k=1,\cdots,N$.
Hence,
$$\|w\|_{L^2(X_{\alpha})} < N\varepsilon_2.$$
Then for any $w\in \mathcal{U}$ and in particular for $w=u_{min}$,
\begin{eqnarray} \label{lower}
\|w-u\|_{L^2(X)}^2 &=&  \|w-u\|_{L^2(X_{\alpha})}^2+  \|w-u\|_{L^2(X_{\alpha}^c)}^2 \nonumber \\
&>& \|u\|^2_{L^2(X_{\alpha})}-2N\varepsilon_2 +  \|w-u\|^2_{L^2(X_{\alpha}^c)}.
\end{eqnarray}

On the other hand the following estimate holds for $u_{J}$:
\begin{eqnarray*}
\|u_{J}-u\|_{L^2(X)}^2 &\leq& (\|u_{J}-u_0\|_{L^2(X)} + \|u_0-u\|_{L^2(X)})^2 \\
&\leq& \|u_{J}-u_0\|^2_{L^2(X)}+4\|u_{J}-u_0\|_{L^2(X)}  + \|u_0-u\|^2_{L^2(X)} \\
&\leq & C \lambda_{J}^{-\frac{1}{2}}\gamma^{-\textrm{dim}(Y)-2}+ \|u\|^2_{L^2(X_{\alpha})}+ \|u_0-u\|^2_{L^2(X_{\alpha+\gamma}-X_{\alpha})} ,
\end{eqnarray*}
where
$$\|u_0-u\|_{L^2(X_{\alpha+\gamma}-X_{\alpha})}\leq \|u\|_{L^2(X_{\alpha+\gamma}-X_{\alpha})} < C \gamma^{\frac{1}{2(n+1)}}.$$
The first inequality above is due to the construction of $u_0$, cf. \eqref{def-u0-partition}.
The last inequality above is due to \eqref{norm-iso}, Sobolev embedding (on $Y$) and
\begin{eqnarray*}
\mu(X_{\alpha+\gamma}-X_{\alpha}) &=& \widetilde{\mu}(\pi^{-1}(X_{\alpha+\gamma})-\pi^{-1}(X_{\alpha})) \\
&=& \frac{\rho_Y}{\textrm{Vol(Y)}} \textrm{Vol}(\pi^{-1}(X_{\alpha+\gamma})-\pi^{-1}(X_{\alpha})) \\
&=& \frac{\rho_Y}{\textrm{Vol(Y)}} \textrm{Vol}(Y_{\alpha+\gamma}-Y_{\alpha}) < C N \gamma <C \gamma^{1/2},
\end{eqnarray*}
where we have used \cite[Proposition 3.10]{BILL}, \eqref{norm-iso}, \eqref{domain-influence-XY}, the upper bound for $\rho_Y$, the lower bound for $\textrm{Vol}(Y)$ due to Lemma \ref{volumebound-Y} and $\gamma<N^{-2}$. 
In the above we denote the difference of sets by $A-B:=A\cap B^c$ for two sets $A,B$.
The domain $Y_{\alpha}$ is defined by
\begin{equation}\label{Yalpha}
Y_{\alpha}:=\bigcup_{k=1}^N Y(\pi^{-1}(U_k), \alpha_k)=\bigcup_{k=1}^N \big\{ y \in Y: d_Y(y,\pi^{-1}(U_k)) < \alpha_k \big\}.
\end{equation}
The fact that 
\begin{equation}
\textrm{Vol} \Big( Y(\pi^{-1}(U_k),\alpha_k+\gamma)-Y(\pi^{-1}(U_k),\alpha_k) \Big) <C\gamma
\end{equation}
follows from \cite[Proposition 3.10]{BILL}. Note that here the constant $C$ depends on $\eta$ due to the condition $\alpha_k\geq \eta$, the same as the special case of Proposition \ref{uc-whole}.

Hence,
$$\|u_{J}-u\|_{L^2(M)}^2 < C \lambda_{J}^{-\frac{1}{2}}\gamma^{-\textrm{dim}(Y)-2}+ \|u\|^2_{L^2(X_{\alpha})}+ C \gamma^{\frac{1}{n+1}}.$$
For sufficiently large $J$, we have $u_{J}\in \mathcal{U}(E_1,\varepsilon_1)$ by Lemma \ref{uJ} with $E_1=C_0 E_0 \gamma^{-\textrm{dim}(Y)-2}$. This indicates that the minimizer $u_{min}$ also satisfies
\begin{equation} \label{upper}
\|u_{min}-u\|_{L^2(X)}^2 < C \lambda_{J}^{-\frac{1}{2}}\gamma^{-\textrm{dim}(Y)-2}+ \|u\|^2_{L^2(X_{\alpha})}+ C \gamma^{\frac{1}{n+1}}.
\end{equation}

Combining the two inequalities (\ref{lower}) and (\ref{upper}), we have
$$\|u_{min}-u\|^2_{L^2(X_{\alpha}^c)} < 2N\varepsilon_2 +C(\Lambda) \lambda_{J}^{-\frac{1}{2}}\gamma^{-\textrm{dim}(Y)-2}+ C\gamma^{\frac{1}{n+1}}.$$
The fact that $\|u_{min}\|_{L^2(X_{\alpha})}< N\varepsilon_2$ implies that
\begin{eqnarray*}
\|\chi_{X_{\alpha}}u-(u-u_{min})\|^2_{L^2(X)}&=&\|u_{min}-\chi_{X_{\alpha}^c}u\|^2_{L^2(X)} \\
&=& \|u_{min}-\chi_{X_{\alpha}^c}u\|^2_{L^2(X_{\alpha}^c)}+\|u_{min}\|^2_{L^2(X_{\alpha})} \\
&<& 2N\varepsilon_2 +C(\Lambda) \lambda_{J}^{-\frac{1}{2}}\gamma^{-\textrm{dim}(Y)-2}+C\gamma^{\frac{1}{n+1}} +N^2\varepsilon_2^2 .
\end{eqnarray*}

Now observe that the Fourier coefficients of $u_{min}$ is solvable given interior spectral data because it is equivalent to a polynomial minimization problem in a bounded domain in $\mathbb{R}^J$, see Remark \ref{remark-computable}. Suppose we have found a minimizer $u_{min}=\sum_{j=0}^{J} c_j \phi_j$. Since the first $J$ Fourier coefficients of $u$ are given as $a_j$, we can replace the function $u-u_{min}$ in the last inequality by $\sum_{j=0}^J a_j \phi_j-u_{min}$ and the error in $L^2$-norm is controlled by $\lambda_{J}^{-1/2}$. Hence by the Cauchy-Schwarz inequality, we obtain
\begin{equation}\label{ualast}
\big\|\chi_{X_{\alpha}}u-\sum_{j=0}^{J}(a_j-c_j)\phi_j \big\|^2_{L^2(X)} < 4N\varepsilon_2+2N^2\varepsilon_2^2 +C(\Lambda) \lambda_{J}^{-\frac{1}{2}}\gamma^{-\textrm{dim}(Y)-2}+ C\gamma^{\frac{1}{n+1}},
\end{equation}
which makes $u^a :=\sum_{j=0}^{J} b_j \phi_j$ with $b_j=a_j-c_j$ our desired function. 

For choices of parameters, making each term of the right-hand side of \eqref{ualast} small gives choices of $\varepsilon_2$, $\gamma$, $\lambda_J$ in this order, where $\eta$ and hence $N$ are fixed. The choice of $\gamma$ gives $E_1$, and together with $\varepsilon_2$ gives choices of $h$, then $\varepsilon_1$ by Proposition \ref{uc-whole}, which at last determines $J$ by Lemma \ref{uJ}.
\end{proof}

We apply Proposition \ref{projection-Y} to $\widetilde{u}$ being constant function $1$ (i.e., the first eigenfunction $\phi_0=1$ since $\mu(X)=1$), 
and thus we can approximate any $\mu (X_{\alpha})$ with explicit estimates.

\begin{lemma} \label{measure-appro}
Let $\alpha=(\alpha_1,\cdots,\alpha_N)$, $\alpha_k\in [\eta, D]$ be given, and $X_{\alpha}$ be defined in \eqref{Xalpha}. 
Then for any $\varepsilon>0$, there exists sufficiently large $J$, such that by only knowing the first $J$ interior spectral data $\{\lambda_j,\phi_j|_{B(p,r_0)}\}_{j=0}^J$ of $X$, we can compute a number $\mu^a (X_{\alpha})$ satisfying 
$$\big|\mu^a (X_{\alpha})-\mu (X_{\alpha}) \big|<\varepsilon.$$
\end{lemma}

\section{Reconstruction of interior distance functions}
\label{sec-recon}

Let $(X,\mu) \in \overline{{\frak M}{\frak M}}(n,\Lambda,\LK,D)$ and $p\in X^{reg}$. Suppose ${\rm Vol}_{d}(X) \geq v_0$ where $d={\rm dim}(X)$.
Given an open subset $U\subset X$, the interior distance functions $r_x: U\to \mathbb{R}$ corresponding to $x\in X$ are defined by $r_x (z)=d(x,z),\, z\in U$. Denote $\mathcal{R}_U(X)=\{r_x: x\in X\}$, see e.g. \cite{KKL} for basic properties. 
Assume that the finite interior spectral data for the weighted Laplacian $\Delta_X$ are given on a ball $B(p,r_0) \subset X^{reg}$.
In this subsection, we consider interior distance functions on $B(p,r_0/2)$, denoted by $B$ in short.
Let $0<\eta< \min\{i_0/2,r_0/4\}$ be such that $\eta<\inf_{x\in B(p,r_0/2)}$ {\rm inj}(x), and $\{p_i\}_{i=1}^N$ be an $\eta/2$-net in $B=B(p,r_0/2)$, where the total number $N$ is bounded by \eqref{bound-N}.
Let $\{U_i\}_{i=1}^N$ be disjoint open neighborhoods of $p_i$ satisfying \eqref{partition}.

Let $\beta=(\beta_1,\cdots,\beta_N),\, \beta_i\in \mathbb{N}$ be a multi-index. Define
\begin{equation} \label{def-slice}
X_{\beta}^{\ast}:= \bigcap_{i: \beta_i>0} \Big\{ x\in X: d(x,U_i)\in [\beta_i \eta-2\eta, \beta_i\eta) \Big\}.
\end{equation}

\begin{lemma} \label{volume-slice}
Under the assumptions of Lemma \ref{measure-appro}, we have
$$\big|\mu^a (X_{\beta}^{\ast})-\mu (X_{\beta}^{\ast}) \big|<2^{N} \varepsilon.$$
\end{lemma}

\begin{definition} \label{def-Rstar}
Let $\eta\in (0,i_0/2)$. We choose all multi-indices $\beta$ with $\beta_i\in \mathbb{N}$, $\beta_i \in (0, D/\eta]$ for all $i$ such that 
\begin{equation}
\mu^a (X_{\beta}^{\ast}) > c_{\ast} \eta^{\dim(Y)}, 
\end{equation}
where $c_{\ast}$ is a uniform constant  depending on $n,\Lambda,D, \textrm{dim}(X),v_0$. 
For all such multi-indices $\beta$, define piecewise constant functions
\begin{equation}
r_{\beta}(z) := \beta_i \eta, \;\textrm{ if }z\in U_i. 
\end{equation}
Denote by $\mathcal{R}^{\ast}$ the set of all such functions $r_{\beta}$.
\end{definition}

\begin{lemma} \label{distance-recon}
For any $\eta>0$, there exists a constant $J_0$ such that for any $J>J_0$, one can construct a finite set $\mathcal{R}^{\ast}$ from the first $J$ interior spectral data $\{\lambda_j,\phi_j|_{B(p,r_0)}\}_{j=0}^J$ of $X$ such that
$$d_H\big(\mathcal{R}^{\ast},\mathcal{R}_B(X) \big) <4\eta,$$
where $B=B(p,r_0/2)$ and $d_H$ is the Hausdorff distance in $L^{\infty}(B)$.
\end{lemma}

\begin{proof}
For any $x\in X$, it is clear that there exists $\beta_i\in \mathbb{N}, \,\beta_i>0$ such that 
$$B(x,\eta/2) \subset X_{\beta}^{\ast}.$$
It is clear that such $r_\beta$ is an approximation of $r_x$.
Indeed, by definition \eqref{def-slice}, the fact that $x\in X_{\beta}^{\ast}$ gives $|d(x,U_i)-\beta_i\eta|<2\eta$ for all $i$.
Since $\textrm{diam}(U_i)\leq \eta$, this shows for all $i$,
$$|d(x,z)-\beta_i\eta|<3\eta,\;\; \forall\, z\in U_i.$$

Furthermore, since $\pi^{-1}(B(x,\eta/2))$ contains a ball of radius $\eta/2$ in $Y$, 
\begin{eqnarray} \label{volume-smallball}
\mu (B(x,\eta/2)) = \widetilde{\mu} (\pi^{-1}(B(x,\eta/2))) &\geq& \widetilde{\mu} (B_Y(\eta/2)) \nonumber \\
&=& \frac{\rho_Y}{{\rm Vol}(Y)} {\rm Vol}(B_Y(\eta/2)) \nonumber \\
&\geq& c \eta^{\dim(Y)},
\end{eqnarray}
where the last inequality used Lemma \ref{bound-Y-Lip}.
We choose 
$\varepsilon_{\ast}=c \eta^{\dim(Y)}/2$ and $\varepsilon=2^{-N} \varepsilon_{\ast}$ in Lemma \ref{volume-slice}. Hence,
$$\mu^a (X_{\beta}^{\ast}) > \mu(X_{\beta}^{\ast})-\varepsilon_{\ast} \geq \mu(B(x,\eta/2)) -\varepsilon_{\ast} \geq \varepsilon_{\ast}.$$

On the other hand, let $\beta$ such that $\mu^a\big( X_{\beta}^{\ast} \big) >\varepsilon_{\ast}$, where $\varepsilon_{\ast}$ and $\varepsilon$ are chosen as above. By Lemma \ref{volume-slice}, we see that 
$\mu \big( X_{\beta}^{\ast} \big)>0$, which means that $X_{\beta}^{\ast}$ is nonempty. Choosing any $x'\in X_{\beta}^{\ast}$, the same argument as before shows that $r_{x'}$ is an approximation of $r_{\beta}$.
\end{proof}

We say that an element $r_{\beta}\in \mathcal{R}^{\ast}$ and a point $x\in X$ correspond to each other if $\|r_{\beta}-r_x\|_{L^{\infty}(U)}<4\eta$.

Observe that $\mathcal{R}^{\ast}$ determines the distances on the $\eta/2$-net $\{p_i\}_{i=1}^N$ within error $8\eta$. Namely, $\mathcal{R}^{\ast}$ determines
\begin{equation} \label{def-Da}
D^a(p_i,p_k):= \inf_{r_{\beta} \in \mathcal{R}^{\ast}} \Big(r_{\beta}(p_i) + r_{\beta}(p_k) \Big),
\end{equation}
and it satisfies (see \cite[Section 3.1]{FILLN})
\begin{equation} \label{d-net}
|D^a(p_i,p_k)-d(p_i,p_k)|<8\eta,\;\; \textrm{ for all }\, i,k=1,\cdots,N.
\end{equation}

\medskip
While a space $X\in \overline{{\frak M}{\frak M}}(n,\Lambda,D)$ has curvature bounded from below in the sense of Alexandrov, it does not have curvature bounded from above in general due to Fukaya's counterexample \cite[Example 1.13]{F86}.
In the case of orbifolds, i.e., $\dim(X)=n-1$, with a lower volume bound, it is possible to have an upper bound for the sectional curvature of the regular part, if the metric on the regular part is of class $C^4$.
For this reason, we impose a bound on the covariant derivatives of the curvature tensor \eqref{higher-curvature}, and consider the  class $\overline{{\frak M}{\frak M}}(n,\Lambda,\LK,D)$.


\begin{lemma} \label{curvature-orbifold}
Let $X \in \overline{{\frak M}{\frak M}}(n,\Lambda,\LK,D)$. Suppose $\dim(X)=n-1$ and ${\rm Vol}_{n-1}(X)\geq v_0$. Then the sectional curvature of $X^{reg}$ is bounded above by $C(n,\Lambda,\LK,D,v_0)$.
\end{lemma}

\begin{proof}
This can be argued similar to Lemma \ref{sec-bound-Y}.
Suppose the sectional curvature is unbounded at $p_k\in X_k$, then \cite[Lemma 7.8]{F88} implies that the isotropy group $G_{p_0}$ at $\lim\limits_{k\to\infty} p_k=p_0 \in X_0=\lim\limits_{k\to\infty} X_k$ has positive dimension.
However, since a uniform lower volume bound is assumed, the sequence of orbifolds $X_k$ does not collapse (e.g. \cite[Theorem 10.10.10]{BBI}) and hence $X_0$ is an orbifold. So the isotropy group of orbifold $X_0$ should have zero dimension, which is a contradiction.
\end{proof}

In the orbifold case, we apply the method in \cite{FILLN} to reconstruct the Riemannian metric of $X^{reg}$ locally and further a global approximation. 
Away from the singular set $S$, Cheeger's injectivity radius estimate (\cite{C2}) is still valid, depending on the lower volume bound and sectional curvature bound of
$X^{reg}$ (Lemma \ref{curvature-orbifold}):
\begin{eqnarray} \label{inj-orbifold}
i(x) &\geq& \min\Big\{  \pi/\sqrt{|R(X^{reg})|)},\,  c(n,D,|R(X^{reg})|,v_0), \, d(x,S) \Big\} \nonumber \\
&=& \min\Big\{ c(n,\Lambda,\LK,D,v_0), \, d(x,S) \Big\} ,\quad x\in X^{reg}.
\end{eqnarray}

With sectional curvature bound (Lemma \ref{curvature-orbifold}) and injectivity radius bound \eqref{inj-orbifold}, an approximation of interior distance functions as given in Lemma \ref{distance-recon} gives the approximate Riemannian structure locally at points bounded away from the singular set. 
Namely, if we \emph{a priori} know a point $x_0\in X^{reg}$ bounded away from the singular set $S$, we can approximately reconstruct the Riemannian metric at $x_0$.

\begin{proposition} \label{recon-1}
Let $(X,\mu) \in \overline{{\frak M}{\frak M}}(n,\Lambda,\LK,D)$ and $p\in X^{reg}$. Suppose $\dim(X)=n-1$ and ${\rm Vol}_{n-1}(X) \geq v_0$. Let $\sigma\in (0,1)$, and $r_0>0$ such that $B(p,r_0) \subset X^{reg}$.
Then for any sufficiently small $\eta$,
there exists a constant $J_0=J_0(n,\Lambda,\LK,D,v_0,\eta)$ such that the following holds.

Let $x_0\in X^{reg}$ satisfying $d(x_0,S)>\sigma$, and
let $\widehat{r}_0\in \mathcal{R}^{\ast}$ be a corresponding element to $x_0$ in the sense of Lemma \ref{distance-recon}. 
Then there exists a basis $\{v_k\}_{k=1}^{n-1}$ in the tangent space $T_{x_0} X$ such that
one can calculate numbers $\widehat{g}_{kl}$ directly from the finite interior spectral data $\{\lambda_j,\phi_j|_{B(p,r_0)}\}_{j=0}^J$ for $J > J_0$ such that
$$\big|\widehat{g}_{kl} - g_{kl}(x_0) \big| < C(\sigma) \eta^{1/8},$$
where $g_{kl}(x_0)$ is the metric components in the Riemannian normal coordinates associated to the basis $\{v_k\}_{k=1}^{n-1}$. The constant $C(\sigma)$ depends on $n,\Lambda,\LK,D,v_0,r_0,\sigma$, and $C(\sigma)\to \infty$ as $\sigma\to 0$.
\end{proposition}

\begin{proof}
The idea is to apply Theorem 2.4 in \cite{FILLN}, but the issue here is that we need to stay away from the singular set. 
Lemma \ref{distance-recon} shows that $\mathcal{R}^{\ast}$ is a $4\eta$-approximation of the interior distance functions on $B(p,r_0/2)$.
To construct the Riemannian metric at $x_0$, we need to use the data in the $\eta^{1/4}$-neighborhood (w.r.t. the $\ell^{\infty}$-norm) of $\widehat{r}_{0}$ in $\mathcal{R}^{\ast}$, as formulated in \cite[Theorem 2.4]{FILLN}.
For any $\widehat{r}_{\ell}\in \mathcal{R}^{\ast}$ satisfying 
$$\|\widehat{r}_{\ell}-\widehat{r}_{0}\|_{\ell^{\infty}}<\eta^{1/4},$$
applying \cite[Proposition 5.2]{FILLN} gives
$$d(x_{\ell},x_0)< C\eta^{1/8},$$
where $x_{\ell}$ is any point corresponding to $\widehat{r}_{\ell}$.
Note that \cite[Proposition 5.2]{FILLN} is valid for Alexandrov spaces with curvature bounded from below.
This indicates
$$d(x_{\ell},S)>\sigma-C\eta^{1/8}.$$
In particular, $x_{\ell}\in X^{reg}$ if $\eta<C\sigma^8$. Due to the convexity of the regular part $X^{reg}$, any minimizing geodesic connecting $x_0,x_{\ell}$ or connecting points in $B(p,r_0)$ does not intersect with the singular set.
Hence, together with Lemma \ref{curvature-orbifold} and \eqref{inj-orbifold}, Theorem 2.4 in \cite{FILLN} is applicable.
The (explicit) dependence of $C(\sigma)$ on $\sigma$ traces back to \cite[Proposition 4.3]{FILLN}, which blows up as $\sigma\to 0$.
\end{proof}

\section{Finding singular set and reconstruction of metric}
\label{sec-singular}

To reconstruct the metric structure from the given finite interior spectral data, one needs to determine if a corresponding point of an element in $\mathcal{R}^{\ast}$ constructed in Definition \ref{def-Rstar} is bounded away from the singular set $S$ and then applies Proposition \ref{recon-1}.
We propose the following procedure.

\medskip
{\bf Step 1.} The first step is to find out the approximate location of the singular set $S$.
We follow the notations is \cite[Section 6.3]{KLLY}.
Let $x\in X^{reg}$, denote by $i(x,\xi)$ the distance from $x$ to the cut point along the geodesic $\gamma_{x,\xi}$ from $x$ with the unit initial vector $\xi$. 
Let $\rho<i(x,\xi)$ and $y=\gamma_{x,\xi}(\rho)\in X^{reg}$, and define
\begin{equation} \label{setN}
N(x,\xi;\rho,s,\varepsilon):=B(y,s+\varepsilon) \setminus B(x,\rho+s).
\end{equation}
Since $d(x,y)=\rho<i(x,\xi)$ in our setting, the minimizing geodesic connecting $x,y$ is unique, so we can also use the notation 
$N(x,y;\rho,s,\varepsilon)$
 to denote the set \eqref{setN}.
 
Recall \cite[Lemma 6.13]{KLLY} that one can find the cut locus by examining if the measure of $N(x,\xi;\rho,s,\varepsilon)$ is zero, when the complete spectral data are known. To study stability, we need a quantitative version of the lemma.
Before reaching the cut locus, i.e., if $\rho+s+\varepsilon\leq i(x,\xi)$, it follows from the triangle inequality that the set $N(x,\xi;\rho,s,\varepsilon)$ contains a ball of radius $\varepsilon/2$ centered at $\gamma_{x,\xi}(\rho+s+\varepsilon/2)$. Hence the same volume estimate as \eqref{volume-smallball} applies:
\begin{equation} \label{volume-beforecut}
\mu\big(N(x,\xi;\rho,s,\varepsilon)\big)\geq c_{\ast}\varepsilon^{\dim(Y)}, \; \textrm{ if }\rho+s+\varepsilon\leq i(x,\xi).
\end{equation}

If $\rho+s>i(x,\xi)$, \cite[Lemma 6.13]{KLLY} says that the set $N(x,\xi,\rho,s,\varepsilon)$ is empty for sufficiently small $\varepsilon$. The following lemma gives a uniform choice of $\varepsilon$, depending on the space $X$.

\begin{lemma} \label{compactness1}
Let $\rho\geq\rho_0>0$, and let $\tau>0$ be fixed.
Then there exists a constant $\widehat{\varepsilon}=\widehat{\varepsilon}(X,\rho_0,\tau)>0$ such that
if a point $x\in X^{reg}$ and a unit vector $\xi\in T_x X$ satisfy $\rho<i(x,\xi)$ and
$\rho+s\geq i(x,\xi)+\tau$, then $N(x,\xi;\rho,s,\varepsilon) =\emptyset$ for all $\varepsilon\leq \widehat{\varepsilon}$.
\end{lemma}

\begin{proof}
The proof is done by contradiction. Suppose that there exist $\varepsilon_j\to 0$, and points and numbers $x_j,y_j,\rho_j,s_j$ satisfying the conditions above, such that $N(x_j,y_j;\rho_j,s_j,\varepsilon_j)\neq \emptyset$. This means that there exist points $q_j$ such that $q_j\in B(y_j,s_j+\varepsilon_j)$ and $q_j\notin B(x_j,\rho_j+s_j)$, i.e.,
$$d(y_j,q_j)<s_j+\varepsilon_j, \quad d(x_j,q_j)\geq \rho_j+s_j.$$
Picking converging subsequences, denote the limit of subsequences of $x_j,y_j,\rho_j,s_j$ by $x,y,\rho,s$ as $j\to \infty$. Then,
$$d(y,q)\leq s, \quad d(x,q)\geq \rho+s.$$
Hence $d(x,q)\leq d(x,y)+d(y,q)\leq \rho+s$, which shows
$d(x,q) = \rho+s$ and $d(y,q)=s$.
In particular, $d(x,q)=d(x,y)+d(y,q)$.
If $q\in X^{reg}$, then the convexity of the regular part means that minimizing geodesics between these points are in the regular part. Then the equality above implies that $q$ is on the extension of the geodesic $\gamma$ passing $x,y$, and the geodesic $\gamma(t)$ starting from $x$ is minimizing at least until $t=\rho+s$. 
On the other hand, if $q\in S$ is a singular point, then \cite{Pet} implies that any minimizing geodesic from $y$ to $q$ only intersects with $S$ at the end point $q$, see \cite[Theorem 10.8.4]{BBI}. 
Then the same argument shows again that 
$q$ is on the extension of the geodesic $\gamma$ passing $x,y$, and the geodesic $\gamma(t)$ starting from $x$ is defined and minimizing at least until $t=\rho+s$.

However, due to the fact that the limit of minimizing geodesics is minimizing in complete length spaces and the limit of cut points is cut point, the cut distance is continuous in $X$. Note that one needs to take the closedness of the singular set into account. Hence $\rho_j+s_j\geq i(x_j,\xi_j)+\tau$ gives $\rho+s\geq i(x,\xi)+\tau$ for some subsequences, where $\tau>0$ is fixed. 
Note that the condition $\rho\leq i(x,\xi)$ combined with the above indicates that $s\geq \tau>0$.
Since geodesics do not minimize past the first cut point, $\gamma(t)$ cannot minimize at $t=\rho+s$, which is a contradiction.
\end{proof}

Note that the choice of $\widehat{\varepsilon}$ above depends on the space $X$ and $\sigma$.
To determine singular points, one can replace $x$ by $x'=\gamma_{x,\xi}(\rho/2)$, and repeat the procedure above. If the cut locus from $x'$ does not move further, then it is a singular point; otherwise, it is not a singular point. 
For $y< i(x)$, denote by $\textrm{cut}_{x,y}$ the cut point with respect to $x$ along the extension of the minimizing geodesic from $x$ to $y$.

\begin{lemma} \label{compactness2}
Let $\rho\geq\rho_0>0$, $d(x,S)>2\rho$, $x'=\gamma_{x,\xi}(\rho/2)$ and $y=\gamma_{x,\xi}(\rho)$.
Let $\sigma>0$ be fixed.
Then there exist constants $\tau,\widehat{\varepsilon}$, depending on $X,\rho_0,\sigma$, such that the following holds.

Let $s$ be the smallest number such that $N(x,y;\rho,s,\varepsilon)=\emptyset$ for some $0<\varepsilon \leq \widehat{\varepsilon}$.
If $N(x',y;\frac{\rho}{2},s+\frac{\tau}{2},\varepsilon)=\emptyset$ for some $0<\varepsilon \leq \widehat{\varepsilon}$, then $d(\textrm{cut}_{x,y},S)<\sigma$.
\end{lemma}

\begin{proof}
First, we claim the following observation:
there exists $\tau=\tau(X,\rho_0,\sigma)>0$ such that if $d(\textrm{cut}_{x,y},S)\geq \sigma$, then $d(\textrm{cut}_{x',y}, \textrm{cut}_{x,y})\geq \tau$.
This can be seen from a straightforward argument by contradiction, using the convexity of the regular part and the fact that the limit of cut points is cut point.

Let $\widehat{\varepsilon}$ be the constant determined in Lemma \ref{compactness1} with the fixed parameter $\tau/10$.
One can assume that $\widehat{\varepsilon}<\tau/10$.
By \eqref{volume-beforecut} and Lemma \ref{compactness1}, the condition that $s$ is the smallest number satisfying $N(x,y;\rho,s,\varepsilon)=\emptyset$ for some $\varepsilon \leq \widehat{\varepsilon}$ implies 
\begin{equation} \label{s-smallest}
i(x,\xi)-\varepsilon < \rho+s \leq i(x,\xi)+\frac{\tau}{10},
\end{equation} 
which shows
$$\rho+s+\frac{\tau}{2}+\varepsilon <  i(x,\xi)+\tau.$$
This implies that $d(\textrm{cut}_{x',y}, \textrm{cut}_{x,y})<\tau$ and thus the lemma. This is because if the opposite is true: $d(\textrm{cut}_{x',y}, \textrm{cut}_{x,y})\geq \tau$, then \eqref{volume-beforecut} and
$$\frac{\rho}{2}+s+\frac{\tau}{2}+\varepsilon< i(x,\xi)-\frac{\rho}{2}+\tau = d(x', \textrm{cut}_{x,y})+\tau \leq d(x', \textrm{cut}_{x',y}),$$
shows that the measure of $N(x',y;\frac{\rho}{2},s+\frac{\tau}{2},\varepsilon)\neq \emptyset$ for every $0<\varepsilon \leq \widehat{\varepsilon}$ is bounded below, which is a contradiction.
\end{proof}

One can easily check that Lemma \ref{compactness2} is still valid if we can only test for the approximate measure $\mu^a$, thanks to \eqref{volume-beforecut} and Lemma \ref{volume-slice} (with $N=2$).
This requires us to choose sufficiently small $\varepsilon$ depending on $X,\sigma$, and hence sufficiently small $\eta$ for finer slicing procedures.

\smallskip
\noindent {\bf Finding singular points.} Let $\sigma>0$ be fixed.
Let $\tau,\widehat{\varepsilon}>0$ be the constants determined in Lemma \ref{compactness2}, and choose an arbitrary $\varepsilon\leq \widehat{\varepsilon}$.
Then we choose sufficiently large $J$ such that the interior spectral data $\{\lambda_j,\phi_j|_{B(p,r_0)}\}_{j=0}^J$ determine $\mu^a$ such that
$\big|\mu^a (X_{\beta}^{\ast})-\mu (X_{\beta}^{\ast}) \big|<c_{\ast}\varepsilon^{\textrm{dim}(Y)}/4$ in Lemma \ref{volume-slice}, so that $\mu^a$ has enough accuracy to distinguish $N(x,y;\rho,s,\varepsilon)$ being empty or not, thanks to \eqref{volume-beforecut}.
We range $x,y$ over all pair of points from the $\eta/2$-net $\{p_i\}_{i=1}^N$ in $B(p,r_0/2) \subset X^{reg}$. 
For each pair $x,y$, choose a point $x'$ in the net such that $d(x',x)\leq d(x,y)/2 +\eta$ and $d(x',y)\leq d(x,y)/2 +\eta$.
This finds the midpoint between $x$ and $y$ in the net with error $\eta$.
Recall that the distances between points in the net can be determined within error $8\eta$, see \eqref{d-net}.
We search for $x,y,\rho,s$ with $d(x,y)=\rho<i(x)$ according to the following criteria:

\begin{itemize}
\item[(1)] $s$ is the smallest number such that $\mu^a\big(N(x,y;\rho,s,\varepsilon) \big)<c_{\ast}\varepsilon^{\textrm{dim}(Y)}/2$;

\item[(2)] choose from $s$ chosen in (1) such that $\mu^a\big(N(x',y;\frac{\rho}{2},s+\frac{\tau}{2},\varepsilon)\big)<c_{\ast}\varepsilon^{\textrm{dim}(Y)}/2$.
\end{itemize}

\noindent We denote the data $x,y,\rho,s$ obtained above by
\begin{equation} \label{data-1}
d_S^{\sigma}(x,y):=\big(\rho+s,s \big).
\end{equation}

Note that the criterion (2) above uses $x'$ as a point in the net $\eta$-close to the midpoint of $x$ and $y$, while in Lemma \ref{compactness2}, $x'$ is the actual midpoint. This can be handled by choosing sufficiently small $\eta$, considering that the measure of the set $N(x,y;\rho,s,\varepsilon)$ is continuous with respect to $x$, and one can choose sufficiently small $\eta$ so that the difference of measure by perturbing $x$ by $\eta$ is less than $c_{\ast}\varepsilon^{\textrm{dim}(Y)}/8$.

\medskip
{\bf Step 2.} The next step is to approximately identify the singular set in the data $\mathcal{R}^{\ast}$.
In the case of finite spectral data, the points $x,x',y$ in Lemma \ref{compactness2} need to be chosen from the finite number of points $\{p_i\}_{i=1}^N$, the $\eta/2$-net in $B(p,r_0/2)\subset X^{reg}$. 
Without loss of generality, assume that $r_0$ is less than the injectivity radius bound given in \eqref{inj-orbifold}.
Note that the procedure described in Lemma \ref{compactness2} stops at cut points.
In order to find the singular set $S$, the points have to satisfy that minimizing geodesics passing pair of points can reach the singular set. 


\begin{lemma}\label{compactness3}
For any $\delta>0$, there exists $\eta=\eta(X,\delta)>0$ such that for any $\eta$-net $\Gamma$ in $B(p,r_0/2)$, the $\delta$-neighborhood of $\cup_{x,y\in \Gamma} \textrm{cut}_{x,y}$ contains the singular set $S$.
\end{lemma}

\begin{proof}
The proof is a straightforward argument by contradiction. If not, assume that there exist $\delta>0$, and $\eta_j\to 0$, $\eta_j$-nets $\Gamma_j$ in $B(p,r_0/2)$ such that there exist $z_j\in S$ such that for any $q_j\in \cup_{x,y\in \Gamma_j} \textrm{cut}_{x,y}$, one has $d(z_j, q_j)>\delta$. 
Pick any point $x_j\in \Gamma_j \cap B(p,r_0/4)$, and consider the minimizing geodesics $[x_j z_j]$ from $x_j$ to $z_j$.
Then one can find converging subsequences $x_j\to x$, $z_j\to z\in S$, and converging minimizing geodesics $[x_j z_j] \to [xz]$. 
Pick any $y\in B(p,r_0/2)-B(p,3r_0/8)$ so that $y\in [xz]$, and thus $\textrm{cut}_{x,y}=z$.
Since $\Gamma_j$ are $\eta_j$-nets with $\eta_j\to 0$, there exist $y_j\in \Gamma_j$ such that $y_j\to y$. Since the limit of cut points is cut point, then for sufficient large $j$, $d(z_j,\textrm{cut}_{x_j,y_j})$ is small than any given $\delta$, which is a contradiction.
\end{proof}

\begin{lemma} \label{compactness4}
Let $\tau>0$ be fixed. Then there exist constants $\eta,\widehat{\varepsilon}$, depending on $X,\tau$, such that the following holds.

Let $\Gamma$ be any $\eta$-net in $B(p,r_0/2)\subset X^{reg}$.
For any $z\in S$, there exist $x,y\in \Gamma$ such that $d(\textrm{cut}_{x,y},z)<\tau$.
Moreover, let $s$ be the smallest number such that $N(x,y;\rho,s,\varepsilon)=\emptyset$ for some $0<\varepsilon \leq \widehat{\varepsilon}$.
Then $N(x',y;\frac{\rho}{2},s+\frac{\tau}{2},\varepsilon)=\emptyset$ for all $\varepsilon \leq \widehat{\varepsilon}$.
\end{lemma}

\begin{proof}
Let $\tau>0$ be fixed. There exists $\delta=\delta(X,\tau)>0$ such that for any $x,y\in X$ satisfying $d(\textrm{cut}_{x,y},S)<\delta$, then $d(\textrm{cut}_{x',y}, \textrm{cut}_{x,y})<\tau/10$.
One can assume that $\delta<\tau$.
Then by Lemma \ref{compactness3}, there exists $\eta=\eta(X,\delta)>0$ such that the $\delta$-neighborhood of $\cup_{x,y\in \Gamma} \textrm{cut}_{x,y}$ contains $S$, where $\Gamma$ is any $\eta$-net in $B(p,r_0/2)$.
This means that for any $z\in S$, there exists $x,y\in \Gamma$ such that $d(\textrm{cut}_{x,y},z)<\delta<\tau$.
Hence we have $d(\textrm{cut}_{x',y}, \textrm{cut}_{x,y})<\tau/10$.
Recall from \eqref{s-smallest} that $\rho+s>i(x,\xi)-\varepsilon$. Hence, for all $\varepsilon \leq \widehat{\varepsilon}<\tau/10$,
$$\frac{\rho}{2}+s+\frac{\tau}{2}>i(x,\xi)-\varepsilon-\frac{\rho}{2}+\frac{\tau}{2} =d(x', \textrm{cut}_{x,y})-\varepsilon+\frac{\tau}{2}>d(x', \textrm{cut}_{x',y})+\frac{\tau}{10},$$
where $\widehat{\varepsilon}$ is the constant determined in Lemma \ref{compactness1} with the fixed parameter $\tau/10$.
Thus, Lemma \ref{compactness1} shows that $N(x',y;\frac{\rho}{2},s+\frac{\tau}{2},\varepsilon)=\emptyset$ for all $\varepsilon \leq \widehat{\varepsilon}$.
\end{proof}

Let $\sigma\in (0,1)$ be fixed. Setting the fixed parameter $\tau$ in Lemma \ref{compactness4} to be the constant determined in Lemma \ref{compactness2}, one determines a constant $\eta$.
One can assume $\tau,10\eta\leq \sigma$ without loss of generality.
Let $\Gamma=\{p_i\}_{i=1}^N$ be any $\eta/2$-net in $B(p,r_0/2)\subset X^{reg}$.
With the data $d_S^{\sigma}$ obtained in Step 1 \eqref{data-1}, we search for all multi-indices $\beta \in \mathcal{R}^{\ast}$ evaluated on all pair of points $x,y\in \Gamma$ in the data $d_S^{\sigma}$ such that 
\begin{equation} \label{criteria-S}
\Big\|\big(r_{\beta}(x),r_{\beta}(y) \big)-d_S^{\sigma}(x,y) \Big\|_{\ell^{\infty}} <3\sigma.
\end{equation}
Denote by $S^{\ast}$ all such multi-indices in $\mathcal{R}^{\ast}$ according to \eqref{criteria-S}.
Here we abuse the notation in Definition \ref{def-Rstar} also denoting the set of multi-indices by $\mathcal{R}^{\ast}$.
We shall show that $S^{\ast}$ has the following property.

\smallskip
$\bullet$ \emph{Claim that Step 1 and \eqref{criteria-S} find $S^{\ast}\subset \mathcal{R}^{\ast}$ and $\widehat{S}\subset X$ such that 
\begin{equation} \label{eq-findingS}
d_H(S, \widehat{S})<C_s\sigma^{1/2},\quad \widehat{S}:= \bigcup_{\beta\in S^{\ast}}X_{\beta}^{\ast},
\end{equation}
for some constant $C_s>1$, where $X_{\beta}^{\ast}$ is defined in \eqref{def-slice}.}

\smallskip
We prove this claim as follows.
For any $z\in S$, there exists $x,y\in\Gamma$ such that they generate admissible data according to criteria \eqref{data-1} due to Lemma \ref{compactness4}. Moreover, we have $d(\textrm{cut}_{x,y},z)<\tau\leq\sigma$.
Recall \eqref{s-smallest} that $|\rho+s-i(x,\xi)|<\tau$ where $(\rho+s,s)=d_S^{\sigma}(x,y)$, which implies 
\begin{equation} \label{data-to-cut}
|\rho+s-d(x,\textrm{cut}_{x,y})|<\tau, \quad |s-d(y,\textrm{cut}_{x,y})|<\tau.
\end{equation}
Hence by the triangle inequality,
$$|\rho+s-d(x,z)|<2\sigma, \quad |s-d(y,z)|<2\sigma.$$
By definition \eqref{def-slice}, one can find a multi-index $\beta\in \mathcal{R}^{\ast}$ such that $z\in X_{\beta}^{\ast}$, i.e., $\|r_z-r_{\beta}\|_{L^{\infty}(\Gamma)}<4\eta$, and thus this multi-index $\beta$ satisfies \eqref{criteria-S}. This proves one direction of the claim \eqref{eq-findingS}.

For the other direction of the claim, take any point $z'\in X_{\beta}^{\ast}$ for some $\beta\in S^{\ast}$, i.e., satisfying $\|r_{z'}-r_{\beta}\|_{L^{\infty}(\Gamma)}<4\eta$.
By definition of $S^{\ast}$, there exist $p_1,p_2\in \Gamma$ in the data $d_S^{\sigma}$ such that \eqref{criteria-S} is satisfied.
Then by triangle inequality,
$$\Big\|\big(r_{z'}(p_1),r_{z'}(p_2) \big)-d_S^{\sigma}(p_1,p_2) \Big\|_{\ell^{\infty}} <3\sigma+4\eta.$$
Since $d(p_1,\textrm{cut}_{p_1,p_2})=d(p_1,p_2)+d(p_2,\textrm{cut}_{p_1,p_2})$, combining with \eqref{data-to-cut}, we have
\begin{equation}
\big|d(z',p_1)-d(z',p_2)-d(p_1,p_2) \big|<9\sigma,
\end{equation}
where we have used $\eta\leq \sigma/10$.
Thus, applying \cite[Lemma 5.1]{FILLN} gives
\begin{eqnarray*}
d(z',\textrm{cut}_{p_1,p_2}) \leq |d(\textrm{cut}_{p_1,p_2},p_2)-d(z',p_2)|+C \sigma^{1/2} < C\sigma^{1/2}+4\sigma+4\eta.
\end{eqnarray*}
On the other hand, by Step 1 criteria \eqref{data-1} and Lemma \ref{compactness2}, there exists $z\in S$ such that $d(z, \textrm{cut}_{p_1,p_2})<\sigma$. Thus, we have
$$d(z',z)\leq d(z',\textrm{cut}_{p_1,p_2})+d(z,\textrm{cut}_{p_1,p_2}) < C\sigma^{1/2}+5\sigma+4\eta,$$
which concludes the proof of the claim \eqref{eq-findingS}.

\medskip
{\bf Step 3.} We search for all $\beta\in \mathcal{R}^{\ast}$ such that
\begin{equation} \label{criteria-3}
\eta\cdot \|\beta-\beta_S\|_{\ell^{\infty}}> 3C_s \sigma^{1/2}, \;\textrm{ for all }\beta_S\in S^{\ast},
\end{equation}
where $S^{\ast}$ is chosen in Step 2 according to \eqref{criteria-S}.
Denote by $\Gamma=\{p_i\}_{i=1}^N$ the $\eta/2$-net in $B(p,r_0/2)$.
For any point $x_0\in X_{\beta}^{\ast}$ corresponding to such a multi-index $\beta$, and for any $z'\in X_{\beta_S}^{\ast}$ for an arbitrary $\beta_S\in S^{\ast}$, we know
\begin{eqnarray*}
\|r_{x_0}-r_{z'}\|_{{\ell}^{\infty}(\Gamma)} &\geq& \|r_{\beta}-r_{\beta_S}\|_{\ell^{\infty}}-\|r_{x_0}-r_{\beta}\|_{\ell^{\infty}}-\|r_{z'}-r_{\beta_S}\|_{\ell^{\infty}} \\
&>& \eta\cdot \|\beta-\beta_S\|_{\ell^{\infty}}-8\eta > 3C_s \sigma^{1/2}-8\eta.
\end{eqnarray*}
This implies that there exists a point in $\Gamma$, say $p_1\in \Gamma$, such that 
$$|r_{x_0}(p_1)-r_{z'}(p_1)|>3C_s \sigma^{1/2}-8\eta,$$
which yields $d(x_0,z')>3C_s \sigma^{1/2}-8\eta$ by the triangle inequality.
Since $z'$ is chosen arbitrarily in $\widehat{S}$, it follows from \eqref{eq-findingS} that 
$$d(x_0,S)\geq d(x_0,\widehat{S})-C_s\sigma^{1/2} > 3C_s \sigma^{1/2}-8\eta-C_s\sigma^{1/2} >\sigma^{1/2},$$
which gives a lower bound away from the singular set $S$. Thus, we can apply Proposition \ref{recon-1} to recover the Riemannian metric locally near $x_0$. Note that the constant in Proposition \ref{recon-1} depends on $\sigma$, and goes to infinity as $\sigma\to 0$.


\medskip
{\bf Step 4.} The last step is to show that the set of points not corresponding to multi-indices picked up in Step 3 is small in $\sigma$.
Namely, for all other multi-indices $\beta$ not picked up in Step 3, there exists $\beta_S\in S^{\ast}$ such that 
\begin{equation}
\eta\cdot \|\beta-\beta_S\|_{\ell^{\infty}}\leq 3C_s\sigma^{1/2}.
\end{equation}
This means that for any $x\in X_{\beta}^{\ast}$ and any $z\in X_{\beta_S}^{\ast}\setminus S \subset \widehat{S}$, we have for all $i=1,\cdots,N$,
$$|d(x,p_i)-d(z,p_i)|<3C_s\sigma^{1/2}+8\eta,$$
where $\{p_i\}_{i=1}^N$ is the $\eta/2$-net in $B(p,r_0/2)$.
Take any point in $\{p_i\}_{i=1}^N$, say $q_1$, and consider the minimizing geodesic $[q_1 z]$. Since $\{p_i\}$ is an $\eta/2$-net, one can find another point in $\{p_i\}$, say $q_2$, such that 
$$|d(z,q_1)-d(z,q_2)-d(q_1,q_2)|<\eta.$$ Then the triangle inequality shows
$$|d(x,q_1)-d(x,q_2)-d(q_1,q_2)|<6C_s\sigma^{1/2}+17\eta.$$
Thus, applying \cite[Lemma 5.1]{FILLN} gives $d(x,z)<C\sigma^{1/4}$.
Then by \eqref{eq-findingS} we have
$$d(x,S)\leq d(x,z)+d(z,S) < C\sigma^{1/4}+C_s\sigma^{1/2}.$$


\medskip
From the constructions in Step 1-4 above, we arrive at the following result.

\begin{theorem} \label{thm-local}
Let $(X,\mu) \in \overline{{\frak M}{\frak M}}(n,\Lambda,\LK,D)$ and $p\in X^{reg}$. Suppose $\dim(X)=n-1$ and ${\rm Vol}_{n-1}(X) \geq v_0$.
Let $\sigma\in (0,1)$, and $U\subset X^{reg}$ be an open subset containing a ball $B(p,r_0)$.
Then for sufficiently small $\eta$, there exists a constant $J_0$ such that the following holds.

\smallskip
(1) The finite interior spectral data $\{\lambda_j,\phi_j|_U\}_{j=0}^J$ for $J > J_0$ determines a $C_0\eta^{1/2}$-net $\{x_i\}_{i=1}^I$ in $X$, where the constant $C_0$ depends on $n,\Lambda,D$.

\smallskip
(2) The finite interior spectral data also determines a subset of the net, say $\{x_i\}_{i=1}^{I_0}$ for some $I_0<I$, with the following property:
there exists a basis $\{v_k\}_{k=1}^{n-1}$ in the tangent space $T_{x_i} X$ such that
one can calculate numbers $\widehat{g}_{kl}$ directly from the finite interior spectral data $\{\lambda_j,\phi_j|_U\}_{j=0}^J$ for $J > J_0$ such that
$$\big|\widehat{g}_{kl} - g_{kl}(x_i) \big| < C'_1(\sigma) \eta^{1/8}, \quad \textrm{for }\, i=1,\cdots,I_0,$$
where $g_{kl}(x_i)$ is the metric components at $x_i$ in the Riemannian normal coordinates associated to the basis $\{v_k\}_{k=1}^{n-1}$.
The index $I_0$ depends on $\sigma,\eta$.
The constant $C'_1(\sigma)$ depends on $n,\Lambda,\LK,D,v_0,r_0,\sigma$, and $C'_1(\sigma)\to \infty$ as $\sigma\to 0$.

\smallskip
(3) The remaining part $\{x_i\}_{i=I_0+1}^{I}$ of the net satisfies $d(x_i,S)<C_3\sigma^{1/4}$ for $i=I_0+1,\cdots,I$, where the constant $C_3$ depends on $n,\Lambda,D$.
\end{theorem}

\begin{proof}
By Proposition \ref{distance-recon}, the finite interior spectral data for sufficiently large $J$ determine a $4\eta$-approximation $\mathcal{R}^{\ast}$ of the interior distance functions. 
We consider the set of points $\{x_i\}_{i=1}^I$ in $X$ corresponding to the elements in $\mathcal{R}^{\ast}$.
Then applying \cite[Proposition 5.2(2)]{FILLN} yields that $\{x_i\}_{i=1}^I$ is a $C_0\eta^{1/2}$-net in $X$. Note that the Proposition 5.2(2) in \cite{FILLN} applies to Alexandrov spaces with curvature bounded from below. 
This proves the first statement.
The second statement follows from our constructions in Step 1-4 above and Proposition \ref{recon-1}, where the subset $\{x_i\}_{i=1}^{I_0}$ of the net corresponds to elements in $\mathcal{R}^{\ast}$ picked up in Step 3 \eqref{criteria-3}.
The third statement is proved in Step 4.
\end{proof}

\begin{theorem} \label{thm-global}
Let $(X,\mu) \in \overline{{\frak M}{\frak M}}(n,\Lambda,\LK,D)$ and $p\in X^{reg}$. Suppose $\dim(X)=n-1$ and ${\rm Vol}_{n-1}(X) \geq v_0$.
Let $\sigma\in (0,1)$, and $U\subset X^{reg}$ be an open subset containing a ball $B(p,r_0)$.
Then for sufficiently small $\eta$, there exists a constant $J_0$ such that the following holds.

The finite interior spectral data $\{\lambda_j,\phi_j|_U\}_{j=0}^J$ for $J > J_0$ determines a subset $\{x_i\}_{i=1}^{I_0}$ of a $C_0\eta^{1/2}$-net $\{x_i\}_{i=1}^{I}$ in $X$ and numbers $\widehat{d}_{i,i'}$ for $i,i'\in \{1,\cdots, I_0\}$, such that
$$\big|\widehat{d}_{i,i'}-d(x_i,x_{i'})\big|< C_1(X,\sigma) \eta^{1/8}, \quad \textrm{for }\, i,i'=1,\cdots,I_0.$$
Moreover, the remaining part $\{x_i\}_{i=I_0+1}^{I}$ of the net is in the $C_3\sigma^{1/4}$-neighborhood of the singular set of $X$.
\end{theorem}

\begin{proof}
Theorem \ref{thm-local} gives an approximation of the Riemannian metric in normal coordinates near $\{x_i\}_{i=1}^{I_0}$, where $\{x_i\}_{i=1}^{I_0}$ is the subset of the net in $X$ satisfying $d(x_i,S)>\sigma^{1/2}$ determined in Step 3. Furthermore, since $x_i$ $(i=1,\cdots,I_0)$ is bounded away from the singular set, given an element $\widehat{r}\in \mathcal{R}^{\ast}$ close to the element corresponding to $x_i$ (w.r.t. the $\ell^{\infty}$-norm), following the method in \cite{FILLN} also gives an approximation of the coordinate of $\widehat{r}$. Hence one obtains an approximation of $d(x_{i_1},x_{i_2})$ when $x_{i_1},x_{i_2}$ are close to each other, see \cite[Corollary 2.5]{FILLN}.

For global approximations, the issue is that a shortest path between two points $x,y\in X^{reg}$ may come close to the singular set $S$.
Suppose that $x,y\in X^{reg}$ satisfy $d(x,S)>\sigma^{1/2}$ and $d(y,S)>\sigma^{1/2}$.
Observe that there exists a constant $\widetilde{\sigma}=\widetilde{\sigma}(X,\sigma)$ such that any shortest path between $x,y$ is bounded away from $S$ by $\widetilde{\sigma}$. This follows from a straightforward compactness argument in the space $X$, using the convexity of the regular part.
Now applying Theorem \ref{thm-local} while replacing the fixed parameter $\sigma$ there by $(\widetilde{\sigma}/2C_3)^4$, one sees from Theorem \ref{thm-local}(3) that near any shortest path between $x,y$, there exist points in the net (depending on $\widetilde{\sigma}$) at which one can construct the Riemannian metric locally as in Theorem \ref{thm-local}(2).
The constant $C_1$ now depends on $\widetilde{\sigma}$ which depends on $X,\sigma$.
Then following the global constructions in \cite[Section 8]{FILLN} gives the estimate for the distances between points in the net bounded away from $S$ by $\sigma^{1/2}$.

The points $x_i$ in the net satisfying $d(x_i,S)>\sigma^{1/2}$ can be determined by the given interior spectral data,
using Step 3 \eqref{criteria-3}. 
However, note that this procedure does not guarantee to find all points $x_i$ satisfying $d(x_i,S)>\sigma^{1/2}$.
The remaining part of the net is in the $C_3\sigma^{1/4}$-neighborhood of the singular set, as argued in Step 4.
\end{proof}

At last, we prove Theorem \ref{thm-GH}.

\begin{proof}[Proof of Theorem \ref{thm-GH}]
This is a direct consequence of Theorem \ref{thm-global}.
The space $X\setminus S_{\sigma,\delta}$ is the $C_0\eta^{1/2}$-neighborhood of $\{x_i\}_{i=1}^{I_0}$, and the finite space is $\widehat{X}=\{x_i\}_{i=1}^{I_0}$.
Using the numbers $\widehat{d}_{i,i'}$ determined in Theorem \ref{thm-global} from the given interior spectral data, one can construct a metric $d_{\widehat{X}}$ on $\widehat{X}$ such that 
$$|d_{\widehat{X}}(x_i,x_{i'})-\widehat{d}_{i,i'}|<2C_1 \eta^{1/8}, \quad \textrm{for }\, i,i'=1,\cdots,I_0.$$
The existence of such a metric $d_{\widehat{X}}$ is guaranteed by Theorem \ref{thm-global}.
As the space $\widehat{X}$ is finite, this is equivalent to finding a solution of a system of linear inequalities with $I_0(I_0-1)/2$ variables.
Then $(\widehat{X},d_{\widehat{X}})$ is a $4C_1\eta^{1/8}$-Gromov-Hausdorff approximation of $X\setminus S_{\sigma,\delta}$, equipped with the restriction of the metric $d$ of $X$ on $X\setminus S_{\sigma,\delta}$.
Moreover, $S_{\sigma,\delta}$ is contained in the $C_0\eta^{1/2}$-neighborhood of the remaining part $\{x_i\}_{i=I_0+1}^{I}$ of the net in $X$, which is at most $C_3\sigma^{1/4}+C_0\eta^{1/2}$ distance away from the singular set.
For the relation between parameters $\eta$ and $\delta$, one needs an explicit computation from the very beginning of the reconstruction.
The three logarithms trace back to the logarithmic and exponential dependence of constant in Theorem \ref{uc-Y}, and to Lemma \ref{distance-recon} where all $N$ points $\{p_k\}_{k=1}^N$ are used in the reconstruction, see Lemma \ref{volume-slice} and \eqref{bound-N}.
One can refer to similar computations in the Appendix of \cite{BILL}.

When we are given the complete interior spectral data $\{\lambda_j,\phi_j|_U\}_{j=0}^\infty$,
where $U\subset X$ is an open and non-empty set, the uniqueness result follows from the fact that the interior spectral data
determines the heat kernel $H(x,y,t)$ for $t>0$, $x,y\in U$, and applying the result of \cite{KLLY} that the  
inverse problem for the heat kernel is uniquely solvable.
\end{proof}


\section{Applications of Gel'fand's problem for collapsing manifolds} \label{sec-applications}

\subsection{Manifold learning}
\label{App: Manifold learning}

Gel'fand's  problem \ref{problem-collapsing} is encountered in manifold learning, also called dimensionality reduction in data science, see e.g. \cite{COLT, FILH0,ISOMAP}.
In manifold learning one is given a point cloud, that is, a finite subset $\{y_j\}_{j=1}^N\subset \R^n$ of 
 $N$ points that are close to a $d$-dimensional submanifold $M_0$ in an $n$-dimensional Euclidean  space, where $n$ is
 much larger than $d$ \cite{FMN}. Then, the goal in learning is to find a manifold $M\subset \R^m$ 
 that is diffeomorphic to $M_0$ so that $d<m<n$. In a more general problem, called the geometric Whitney
 problem \cite{FIKLN}, one is given a possibly discrete metric space $(Y,d_Y)$ (that is, not necessarily a subset of an Euclidian space), and the task is to construct a smooth Riemannian manifold $M$
that is close to $Y$ in the Gromov-Hausdorff sense.
 The manifold learning problem is encountered in situations where a given data set 
 is assumed to have an intrinsic structure close to a smooth manifold.
 An example of the above problems is encountered in
the cryogenic electron-microscopy (Cryo-EM), see \cite{Cryo4,Singer}, that is, a method to find
the structure of large 
molecules or viruses.
This method is applied to a large number of identical  
samples of the  molecules 
frozen in vitreous water.
By applying electron-microscopy to  a thin slice of the ice containing the molecules, 
 one obtains 
a large number of images of the molecule, all taken from unknown directions.
  For example, the images of the COVID-19 viruses (including their  spike proteins)
 were obtained using this technique. The 2017 Nobel Prize in Chemistry was awarded for developing methods for this problem.
 Mathematically, 
each measured image of individual viruses consists of $m\times m$ pixels that correspond to points in $\R^n$ where $n=m^2$.
The number of images of the virus particles, $N$, is quite large and all images are taken of identical objects (e.g. a virus) but
from unknown directions. Mathematically speaking, the unknown directions correspond to
an element in the 3-dimensional rotation group  $SO(3)$. Thus the set of images
can be regarded as a set of $N$ data points $y_1,\dots,y_N$ in a high dimensional Euclidean
space $\R^n$ that lie approximately on an orbit
of the 3-dimensional rotation group  $SO(3)$. The individual images are 
very noisy
and thus the imaging problem corresponds to finding an approximation of a  submanifold $M_0$ in $\R^n$ that
is diffeomorphic to $SO(3)$, when we are given $y_j=x_j+\xi_j$, $j=1,2,\dots,N$,  where $x_j$ are randomly sampled
points on $M_0$ and $\xi_j\sim N(0,\sigma^2I)$ are Gaussian measurement errors.
In addition to Cryo-EM, manifold learning is widely applied in statistics, visualization of data, astronomy  and chemistry, see
\cite{Meila}.

An extensively studied method in manifold learning is the diffusion maps (or the spectral embedding) introduced
 by Coifman and Lafon, see \cite{Coifman2,Coifman1}.
In this method one considers the data points $X=\{x_j\}_{j=1}^N\subset \R^n$, which are sampled from a manifold $M_0$ as the nodes of a graph, and uses Euclidean
distances between these points to
compute a kernel function $k_t:X\times X\to \R$ that approximates the heat kernel of the graph Laplacian
of $X$ at a small time $t$. The values $k_t(x_j,x_{j'})$ can be considered as an
approximation the values of the heat kernel $H(x_j,x_{j'},t)$ of the manifold $M_0$ at the points 
$x_j,x_{j'}\in M_0$ at the time $t$. The values of the kernel $h_t(x_j,x_{j'})=H(x_j,x_{j'},t)$ at the times $t\in \{t_1,t_2,\dots,t_l\}\subset \R_+$
correspond to the pointwise heat data, $PHD$, considered in \cite{KLLY}.
Then, in the diffusion map 
  algorithm one continues by computing the first $J$ eigenfunctions $\phi_j(x)$
of the integral operator having the kernel $h_t(x,x')$, $x,x'\in X$, and uses the eigenfunction evaluation map,
defined as
\begin{eqnarray}\label{eigenmap}
 & &\Phi^{(K,J)}:X\to \R^J,\\ \nonumber
& &\Phi^{(K,J)}(x)=\big(\phi_K(x),\phi_{K+1}(x),\dots,\phi_{K+J-1}(x) \big),
\end{eqnarray}
to construct the set $\Phi^{(K,J)}(X)\subset \R^J$ that  approximates the Riemannian manifold $M_0$. 
Closely related to the present paper and the part I of the paper \cite{KLLY} is the problem of reconstructing an approximation
of the Riemannian manifold $M_0$ in the case when we are given a spatially restricted data. In this case, one has only the local finite interior spectral data,
that is,  $\lambda_j$ and $\phi_j(x_k)$,
where $j\le J$, $k\le K$ and $J,K\in \mathbb Z_+$, or the
local pointwise heat data
$h_{t_\ell}(x_k,x_{k'})=H(x_k,x_{k'},{t_\ell}),$ where  $\ell \le L$ and the sample points $\{x_k: k=1,2,\dots,K\}$ are an $\e$-dense set of points in a possibly small ball $B(x_0,r)$ of $M_0$, not an $\e$-dense set of points in the whole manifold $M_0$. The points in the ball $B(x_0,r)$ can be considered as spatially restricted marker points.

The collapsing of manifolds is encountered in manifold learning when
the data set has several degrees of freedom in different scales \cite{Jones-Marron}. As an motivational example,
one can consider a $d$-dimensional submanifold $M_0\subset \R^n$ that is an embedded image of $N_0\times S^1$
of a map $f:N_0\times S^1\to \R^n$, where $N_0$ is a $(d-1)$-dimensional manifold and 
\beq
f(z,\theta)=f_0(z)+\e \big(\cos\theta\,\nu_1(z)+\sin\theta\,\nu_2(z) \big),\quad z\in N_0,\ \theta \in [0,2\pi],
\eeq
where $f_0:N_0\to \R^n$ is a smooth embedding, and $\nu_1(z)$ and $\nu_2(z)$ are two orthogonal unit
vectors normal to $f_0(N_0)$ at the point $f_0(z)$. The image $f(N_0\times S^1)$ models
a data set that is determined by a $(d-1)$-dimensional variable $z\in N_0$ and an additional 
degree of freedom $\theta\in S^1$ that has a small-scale effect. In the example above on  Cryo-EM,
such  an almost collapsed case is encountered when a large protein (e.g. a virus)
is connected to a small part by a chemical bond that allows a rotation. 
An example of such a bond is the peptide amide linkage between two parts of the molecule, 
see \cite[pp. 29-38]{Bhagavan} for an overview on the degree of freedom in
the 3-dimensional structure of proteins.
In this case, the problem of finding the manifold for the possible EM-images of the molecule means finding a manifold $M$ diffeomorphic to $SO(3)\times S^1$
that is almost collapsed to a lower dimensional manifold.
In data science, the collapsing of manifolds is closely related to multi-scale models
where the intrinsic dimension of the data set is modelled by a function that depends on the
scale (i.e. the size of window) in which the data set is approximated by a smooth structure, see \cite{Wang-Marron}.
Next, we consider a numerical example in the case when we are given the heat kernel data, or distance data, on the whole manifold.


\begin{figure}[h]
\centering 
\hspace{-15mm}
\includegraphics[height=3.5cm]{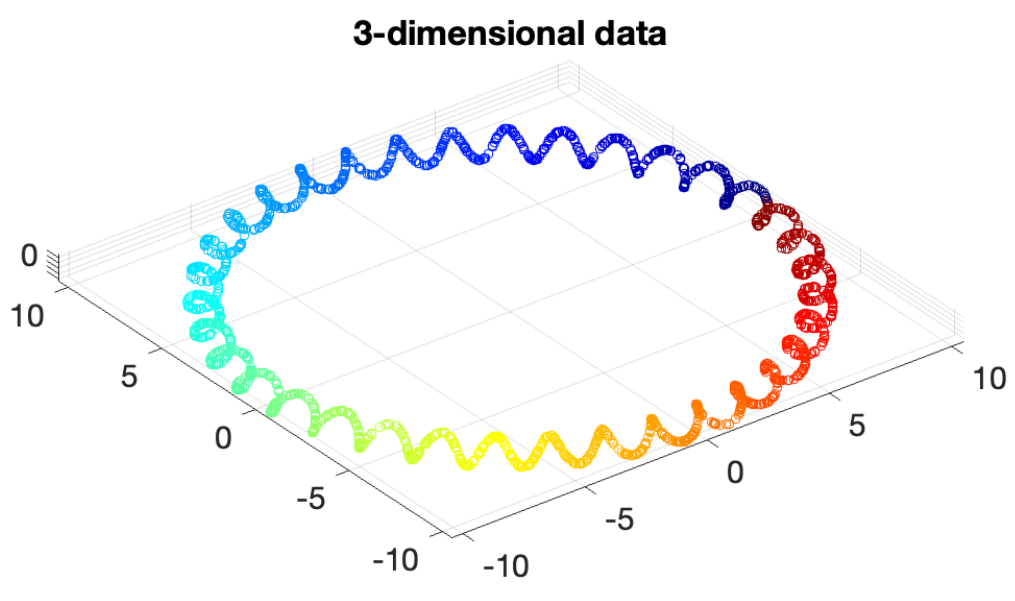}
\includegraphics[height=3.5cm]{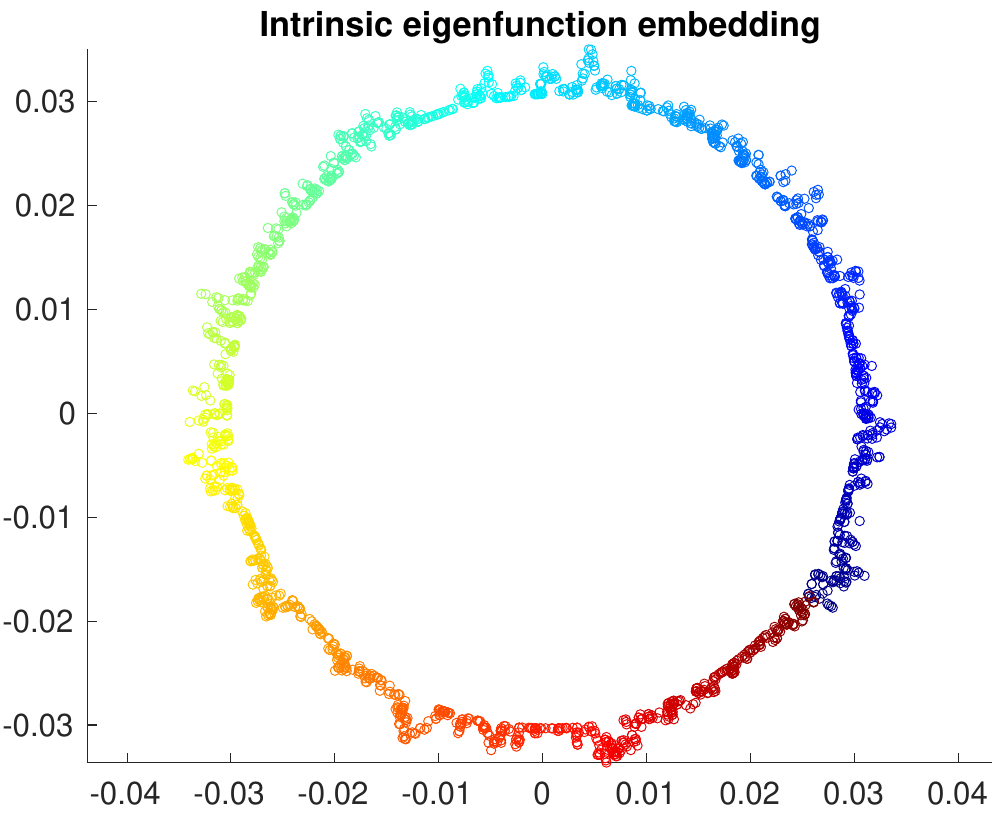}
\includegraphics[height=3.5cm]{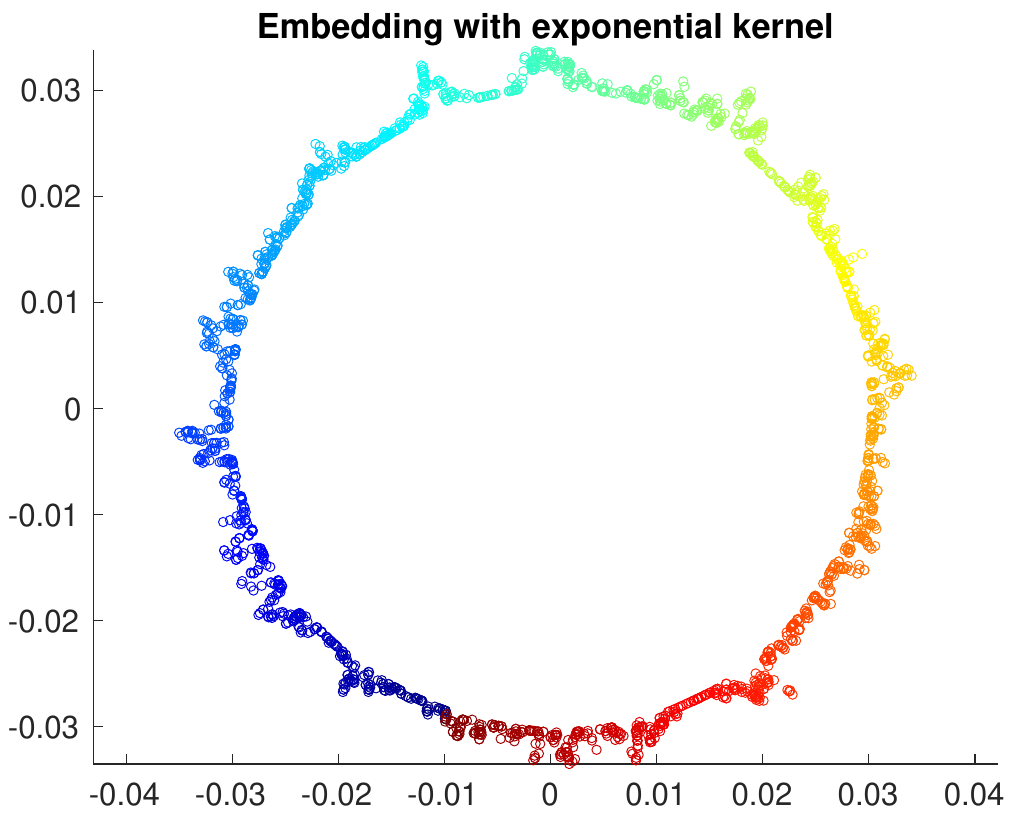}
\hspace{-15mm}
\caption{Applications of Gel'fand's problem in manifold learning. {\bf  Left:}  
Data points $\{x_j\in \R^3:\ j=1,2,\dots, N\}$ with $N=2048$, where $x_j=\gamma^{emb}(\tau_j)$ are randomly sampled points on
a helical curve $\gamma^{emb}\subset M_{R,r}^{emb}$ in a two-dimensional torus $M_{R,r}^{emb}$ having the larger radius $R=10$ and
the smaller radius  $r= \frac 12.$ The induced metric of the embedded manifold $M_{R,r}^{emb}\subset \R^3$
is close to the flat metric of the torus $M_{R,r}=\mathbb S^1_R\times \mathbb S^1_r$.
As $r$ is small, the torus $M_{R,r}$  is `almost collapsed' to the circle $M_R=\mathbb S^1_R$.
{\bf Center:} A low-dimensional representation  $\Phi^{(2,2)}(X)\subset \R^2$ of the manifold $M_{R,r}$, obtained using the
eigenfunctions of the heat kernel $h_t(x_i,x_j)$ that  coincide with the eigenfunctions of the Laplacian. {\bf Right:} A low-dimensional representation  
$\Phi^{(2,2)}(X)\subset \R^2$ of the manifold $M_{R,r}$, obtained using the
eigenfunctions of the integral operator having the exponential kernel $k_t(x_i,x_j)$, see \eqref{eq: Coifman kernel}.
The kernel $k_t(x_i,x_j)$ is an approximation of the heat kernel and can be easily computed using the geodesic distances $d_{M_{R,r}}(x_i,x_j)$ on $M_{R,r}$. In both cases (C1) and (C2),
the image of the 2-dimensional eigenfunction map
$\Phi^{(2,2)}(X)\subset \R^2$ is close to the circle $M_R$.
}

%
%
%
%

\label{fig_embedding}
\end{figure}


\subsubsection{A numerical example on manifold learning with complete
Gel'fand data}

As a simple example, let us consider the case when the 
data are given on the whole manifold, that is, on the ball $B=B_M(x_0,R_0)$
where $R_0>\diam(M)$ so that $B=M$.  In this case, we demonstrate numerically two methods to learn
a manifold, where one method uses the complete Gel'fand data and the other method is based on an approximation of the heat kernel.
The methods developed in this paper for the local Gel'fand data are not yet numerically implemented or analyzed.

We consider  a manifold $M$ that is the torus $M_{R,r}=\mathbb S^1_R\times \mathbb S^1_r$, endowed with the flat product metric, where $\mathbb{S}^1_R$ is the circle of radius $R$. In the numerical test, the larger radius is $R=10$ and the smaller radius is $r=\frac 12$.
Due to visualization purposes, we also consider an embedded
2-dimensional torus
in $\R^3$,  see Figure \ref{fig_embedding} (left), given by
\beq
& &M^{emb}_{R,r}=
\\ \nonumber
&&\Big\{\big((R+r \cos(s_1))\cos(s_2),(R+r\cos(s_1)) \sin(s_2), r\sin(s_1) \big)\in \R^3:\ s_1,s_2\in [0,2\pi] \Big\}.
\eeq
As the ratio $R/r$ is large, the intrinsic metric of the torus $M^{emb}_{R,r}$ inherited from $\R^3$ is relatively close to the flat metric of $M_{R,r}$.

We consider data points $X=\{x_j\in M_{R,r}: j=1,2,\dots,N\}$ with $N= 2048$, where 
$x_j$ are points sampled from the torus $M_{R,r}$. The points $x_j$ could be
sampled randomly on $M_{R,r}$, but to make the visualization of the situation
clearer, we sample $x_j$  randomly on a closed geodesic $\gamma_{x_0,\xi_0}$,
where, in the $(s_1,s_2)$-coordinates, 
$x_0=(0,0)$
and $\xi_0=\frac 1m \p_{s_1}+\p_{s_2}$ with $m=35$. In other words, the sample points are $x_j=\gamma_{x_0,\xi_0}(\tau_j)$, where $\tau_j$ are    independent samples from
the uniform distribution on $[0,2\pi].$ The geodesic $\gamma_{x_0,\xi_0}$ on $M_{R,r}$ is visualized 
in Figure \ref{fig_embedding} (left)
as a helical curve $\gamma^{emb}$ on $M^{emb}_{R,r}$,
$$
\gamma^{emb}(s)=\Big(
(R+r \cos(ms))\cos(s),(R+r\cos(ms))\sin(s), r\sin(ms)\Big)\in \R^3.
$$

 We consider the following two cases:
\begin{enumerate}
\item  [(C1)] (Two-dimensional representation of the manifold using the heat kernel)
We assume that we are given the Gel'fand data, that is, the values of the heat kernel $h_t(x_i,x_j),$ $i,j=1,\dots,N,$
$t=100$.
 We compute a 2-dimensional representation of the manifold using
 the heat kernel of the  manifold. The eigenfunctions of the Laplacian
 operator coincide with the eigenfunctions of the integral operator defined by the heat kernel $h_t(x,y)$,
and we use  these eigenfunctions to represent the manifold $M_{R,r}$  in $\R^2$.

 \item [(C2)] (Two-dimensional representation of the manifold using an exponential kernel)
We assume that we are given the geodesic distances  $d_{M_{R,r}}(x_i,x_j),$ $i,j=1,\dots,N$, on the flat torus $M_{R,r}$.
Using these data we compute the exponenal kernel  $k_t(x_i,x_j)$ defined in \eqref{eq: Coifman kernel},
that is, 
 an approximation of the heat kernel of the  manifold. Using the eigenvectors of the matrix
 associated to  a weighted version of the exponential kernel, we compute
 a 2-dimensional representation of the manifold $M_{R,r}$  in $\R^2$. 
%
\end{enumerate}
Observe that when $r$ is small, the torus $M_{R,r}$
is `almost collapsed' to the limit space that is the circle $M_R=\mathbb S^1_R$. 

\smallskip
In the case (C1) we use a version of the diffusion map algorithm of
 Coifman and Lafon, see \cite{Coifman2,Coifman1}, which we modify below
 so that it uses 
the heat kernel of the manifold.
%
As data, we use the point values of the heat kernel $h_t(x_i,x_j)$ of the Riemannian manifold $(M_{R,r},g)$,
where $g=ds_1^2+ds_2^2$ is the flat product metric of the torus and $t=100$. 
On the flat torus $M_{R,r}$, the values of the heat kernel are computed using the formula 
\beq\label{heat kernel formula}
&&h_t(x_i,x_j)=
\\ \nonumber
& & \sum_{\ell_1,\ell_2\in \Z} \frac 1{4\pi t} \exp
\bigg(-\frac 1{4t} \Big((s_1(x_i)-s_1(x_j)-R\ell_1 )^2+(s_2(x_i)-s_2(x_j)-r\ell_2)^2 \Big)\bigg),
\eeq
where 
$s_1(x)$ and $s_2(x)$ denote the $s_1$ and $s_2$
coordinates of the point $x\in M_{R,r}$. 
(We note that the flat torus was chosen to be our example as its heat kernel can be easily computed 
using formula \eqref{heat kernel formula}.)
First, we define the weighted kernel
\beq
\tilde h_t(x_i,x_j)=\frac 1{h(x_i)}h_t(x_i,x_j),\quad 
h(x_i)=\sum_{j=1}^N h_t(x_i,x_j).
\eeq
Second, 
 we compute the singular value decomposition 
of the matrix $(\tilde h_t(x_i,x_j))_{i,j=1}^N$,
given by $\tilde h_t(x_i,x_j)=\sum_{p=1}^N \omega_p\phi_p(x_i)\psi_p(x_j)$,
where ${\bf v}_p=(\phi_p(x_i))_{i=1}^N$ and ${\bf w}_p=(\psi_p(x_j))_{j=1}^N$ are orthogonal
vectors in $\R^N$, and $\omega_{p+1}\ge \omega_{p}\ge 0$.
We define the map $\Phi^{(K,J)}$, see
\eqref{eigenmap}, using  the right eigenfunctions $\phi_p(x)$,
where $p=K,K+1,\dots, K+J-1$,  in the singular value decomposition. 
Finally, a low-dimensional approximation
of the manifold $M_{R,r}$ is computed by 
the image of the 2-dimensional eigenfunction map
$\Phi^{(2,2)}(X)\subset \R^2$. The result is shown in Figure \ref{fig_embedding} (center).

In the case (C2) we implement a simplified version of the algorithm used above in the case (C1), where
the heat kernel $h_t:X\times X\to \R$ is replaced by an exponential kernel function $k_t:X\times X\to \R$.
We assume that we are given the geodesic distances $d_{M_{R,r}}(x_i,x_j)$, $i,j=1,\dots,N$, and 
we compute the exponential kernel using the formula
 \beq\label{eq: Coifman kernel} 
 k_t(x_i,x_j)=\frac 1{4\pi t}\exp(-\frac {d_{M_{R,r}}(x_i,x_j)^2}{4t})
 \eeq
  with $t=100$. 
  We note that this algorithm is a modified version of the classical diffusion map algorithm where we 
  use the geodesic distances on $M_{R,r}$ instead of  the Euclidean distances on $M^{emb}_{R,r}\subset \R^3$.
  The kernel  $k_t(x_i,x_j)$ can be considered as an  approximation of the heat kernel of the manifold, as due to Varadhan's formula, (\ref{eq: Coifman kernel}) is the leading order asymptotics of the heat kernel as $t\to 0$.
Then, the diffusion map algorithm  is performed by computing the right eigenfunctions $\phi_j(x)$ in the singular value decomposition 
of the matrix $(\tilde k_t(x_i,x_j))_{i,j=1}^N$, where
\bequ
\tilde k_t(x_i,x_j)=\frac 1{k(x_i)}k_t(x_i,x_j),\quad 
k(x_i)=\sum_{j=1}^N k_t(x_i,x_j).
\eequ
 Finally, a low-dimensional approximation
of the manifold $M_{R,r}$ is given by 
the image of the 2-dimensional eigenfunction map
$\Phi^{(2,2)}(X)\subset \R^2$, see Figure \ref{fig_embedding} (right).
In both cases (C1) and (C2), the image of the 2-dimensional eigenfunction map
$\Phi^{(2,2)}(X)\subset \R^2$ in Figure \ref{fig_embedding} is topologically close to a circle, that is, $M_{rep}:=\Phi^{(2,2)}(X)$ is an approximation
of the limit space $M_R$. 

We point out that in the earlier sections of the paper we have studied the local version of
Gel'fand's  problem \ref{problem-collapsing}, where 
the heat kernel $H(x_j,x_{j'},t_\ell)$ are  given at points $x_j$ that do not fill
the whole manifold $M$ but only fill a possibly small metric ball $B=B(x_0,R_0)\subset M$ with $R_0<\diam(M)$, that is, $\{x_j:\ j=1,\dots,N\}\subset B$.
The data missing from the points $x\in M\setminus B(x_0,R_0)$ are compensated by measuring
the heat kernel at several times $t_\ell>0$, or alternatively, measuring
a large number eigenfunctions $\phi_j$ on $B(x_0,R_0)$ and the corresponding eigenvalues,
see Theorem \ref{thm-GH}.

\subsection{Collapsing manifolds in physics}
\label{review-physics}

In modern quantum field theory,
in particular string theory, one often models  the Universe
as a high-dimensional, almost collapsed manifold.
This type of considerations started from the Kaluza-Klein theory in 1921
in which the 5-dimensional Einstein equations are considered on $\R^4\times S^1(\e)$, that is, the
Cartesian product of the standard 4-dimensional space-time with the Minkowski metric and 
a circle $S^1(\e)$ of radius $\e$. As $\e\to 0$, the 5-dimensional Einstein equation yields
a model containing both the Einstein equation and Maxwell's equations.
In this subsection we shortly review this and discuss its relation to the considerations
presented in the main text of this paper.


\begin{figure}[h]
\centering
\includegraphics[height=4.5cm]{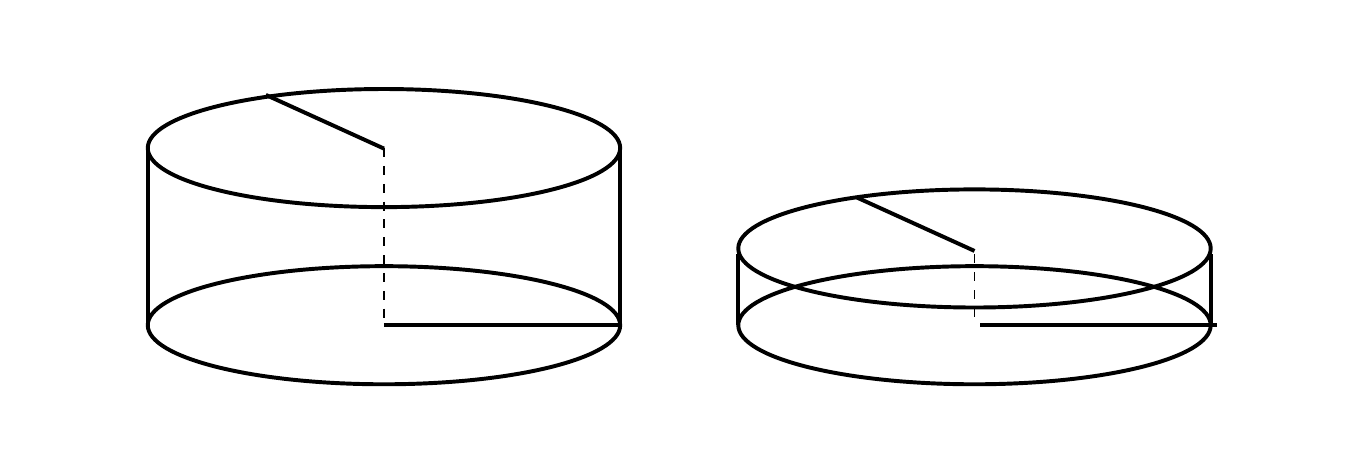}
\includegraphics[height=4.5cm]{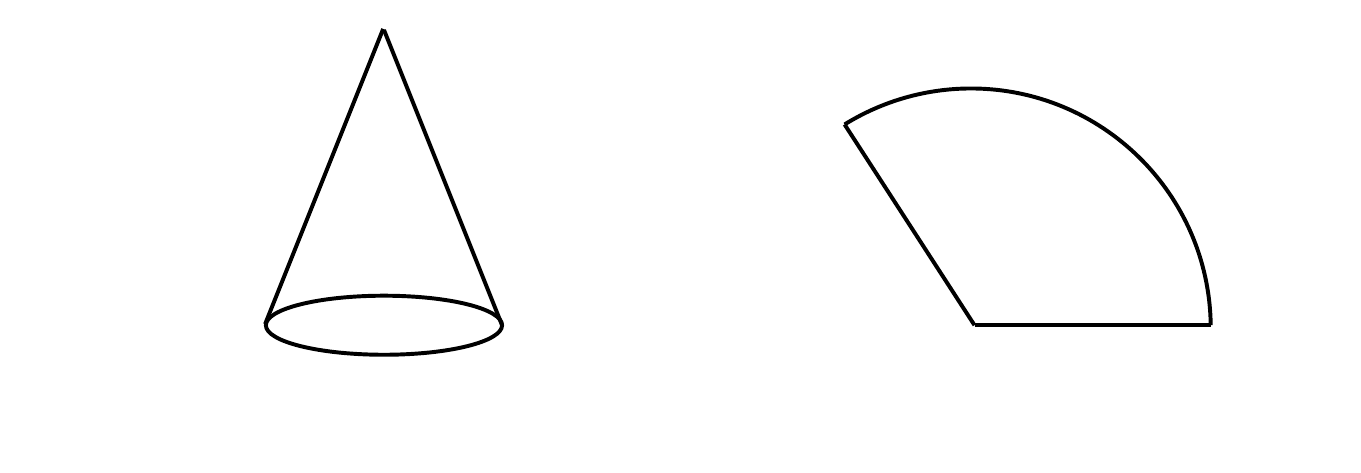}
\vspace{-8mm}
\caption{An example of a product of the 2-dimensional unit disc $D^2\subset \R^2$ with an interval $[0,\epsilon]$, that is, 
a cylinder $D^2\times [0,\epsilon]$. When the top and the bottom of the cylinder are glued together with a  twist of $2\pi/3$ radians,
we obtain a twisted solid torus $M_\epsilon$ (see the top left figure).
As $\epsilon\to 0$, see  the top right figure, this twisted solid torus $M_\epsilon$ collapses
to a 2-dimensional cone (see the bottom left  figure). This cone is a $2\pi/3$ radians sector of a disc, whose radial boundaries are glued together (see the bottom right figure).}
\label{fig_cone}
\end{figure}

First, let us begin with a basic example of collapsing, as illustrated in Figure \ref{fig_cone}.
We start with a product manifold $D^2 \times [0, \epsilon]$, where $D^2$ is the unit disk in $\R^2$ with the Euclidean metric.  Consider the action of the finite group $\Z_m, \,
 m \geq 2$, on $D^2$ by a rotation of angle $2 \pi/m$ around the origin. Then we define $M_{\epsilon}$ by identifying points 
 $({\bf x}, 0) \in D^2 \times \{0\}$ with $(e^{2\pi i/m} \cdot {\bf x}, \epsilon) \in D^2 \times \{\epsilon\}$ in the manifold $D^2 \times [0, \epsilon]$,
 where $e^{2\pi i/m} \cdot {\bf x}$ stands for the rotation by the angle $2 \pi/m$.
 Observe that this gives rise to closed vertical geodesics of length $m \epsilon$ except for the points corresponding to the origin.
 When $\epsilon \to 0$, $M_\epsilon$ collapses to a $2$-dimensional orbifold $X=D^2/\Z_m$, which has conic singular point only at the origin.

\smallskip
In the Kaluza-Klein theory, one starts with a 5-dimensional
manifold $N=\R^4\times S^1$ with the metric $\hat g=\hat g_{jk}$,
$j,k=0,1,2,3,4$, which is a Lorentzian metric of the type $(-,+,+,+,+)$. The "hat"
marks the fact that $\hat g$ is defined on a 5-dimensional manifold.
 We use on $N$ the coordinates $x=(y^0,y^1,y^2,y^3,\theta)$,
where $\theta$ is consider as a variable having values on $[0,2\pi]$.
Let us start with a background metric (or the non-perturbed metric)
$$
\hat {\overline g}_{jk}(x)dx^jdx^k=
\eta_{\nu\mu}(y,\theta)dy^\nu dy^\mu +\e^2d\theta^2,
$$
where $\e>0$ is a small constant and $\nu,\mu,j,k$ are summation indices taking
values $\nu,\mu,j,k \in \{ 0,1,2,3\}$ . Here, we consider
the case when $[\eta_{\nu\mu}(y,\theta)]_{\nu,\mu=0}^3=\diag(-1,1,1,1)$.
Next we consider perturbations of the background metric $\hat {\overline g}$ in the following form
$$
\hat g_{jk}(x)dx^jdx^k=e^{-\sigma/3}
\Big(e^\sigma (d\theta +\kappa A_\mu dy^\mu)^2+g_{\nu\mu}dy^\nu dy^\mu \Big),
$$
where $\kappa>0$ is a constant, $\sigma=\sigma(y,\theta)$ is a function close to the constant $c(\e):=3 \log \e$,
the 1-form $A=A_\mu(y,\theta) dy^\mu$ is small and 
 $g_{\nu\mu}(y,\theta)$ is close to $e^{\sigma(y,\theta)/3}\eta_{\nu\mu}(y,\theta)$.  

Next, assume that $\hat g_{jk}$ satisfies the 5-dimensional Einstein equations
$$
\hbox{Ric}_{jk}(\hat g)-\frac 12(\hat g^{pq}\hbox{Ric}_{pq}(\hat g))\, \hat g_{jk}=0\quad\hbox{ in }\R^4\times S^1,
$$
and write $g=g_{\mu\nu}(y,\theta),$ 
$A=A_{\mu\nu}(y,\theta)$ and $\sigma=\sigma_{\mu\nu}(y,\theta)$ in terms of Fourier series,
\ba
g_{\mu\nu}=\sum_{m=-\infty}^\infty g_{\mu\nu}^{m}(y)e^{im\theta},\quad
A_\mu=\sum_{m=-\infty}^\infty A_{\mu}^{m}(y)e^{im\theta},\quad \sigma=\sum_{m=-\infty}^\infty \sigma^{m}(y)e^{im\theta}.
\ea
Here, the functions $A_{\mu}^{m}(y)$ and $ \sigma^{m}(y)$  correspond to some physical fields. When $\e$ 
is small, the manifold $N$
is almost collapsed in the $S^1$ direction and all these functions
with $m\not=0$ corresponds to physical fields (or particles) of a very high energy which do not
appear in physical observations with a realistic energy. Thus
 one considers only the terms $m=0$ and after suitable approximations
  (see \cite{book1} and \cite[App.\ E]{Wald}) 
one observes that  
 the matrix $g_{\mu\nu}^{0}(y)$, considered as a Lorentzian metric
on $M=\R^4$, the 1-form $A^0(y)=A_{\mu}^0(y)dy^\mu$,
and the scalar function $\phi(y)=\frac 1{\sqrt 3}\sigma^0(y)$ 
satisfy
\beq
\label{KK1} & &\hbox{Ric}_{\mu\nu}(g^0)-\frac 12((g^0)^{pq}\hbox{Ric}_{pq}(g^0))\, g^0_{\mu\nu}=T_{\mu\nu}\quad\hbox{ in }\R^4,\\
\label{KK1b}& &T_{\mu\nu}= \kappa^2 e^{\sqrt 3\, \phi}\big((g^0)^{\a\beta}F_{\a\nu}F_{\beta\mu}-
\frac 14 (g^0)^{\a\beta}(g^0)^{\gamma\delta}F_{\a\gamma}F_{\beta\delta}g^0_{\mu\nu}
\big)
\\ & &\hspace{2cm}\nonumber+\nabla_\mu\phi\, \nabla_\nu\phi-
\frac 12((g^0)^{\a\beta }\nabla_\a\phi\, \nabla_\beta\phi)g_{\mu\nu},\\
\label{KK2} & &d(*F)=0\quad\hbox{ in }\R^4,\\
\label{KK3} & &\square_{g^0}\phi=
\frac 14\kappa^2 e^{\sqrt 3\, \phi}(g^0)^{\a\beta}F_{\a\nu}F_{\beta\mu}
\quad\hbox{ in }\R^4,
\eeq
where $T$ is called the stress energy tensor, $*$ is the Hodge operator with respect to  the metric $g^0_{\mu\nu}(y)$,
$\nabla_\mu\phi=\frac {\p\phi}{\p x^\mu}$ denotes the partial derivative,
$\square_{g^0}$ is the Laplacian 
 with respect to the Lorentzian metric $g^0_{\mu\nu}(y)$
 (i.e.\ the wave operator),
and
$F_{\mu\nu}(y) $ is the exterior derivative of the 1-form
$A^0(y)=A_{\mu}^0(y)dy^\mu$,
$
F=dA^0.
$ 
Physically, if we write $F=E(y)\wedge dy^0+B(y)$, then $E$ corresponds
to the electric field and $B$ the magnetic flux  (see \cite{book2}). 
%
%

In the above, equation  (\ref{KK2}) for  $F=dA^0$ are 
the 4-dimensional formulation of
Maxwell's equations, and (\ref{KK3}) is 
a scalar wave equation corresponding  a mass-less scalar field that interacts
with the $A^0$ field.
Equations (\ref{KK1})-(\ref{KK1b}) are the 4-dimensional Einstein equations in a curved space-time
with stress-energy tensor $T$, which corresponds to the stress-energy of the electromagnetic
field $F$ and scalar field $\phi$. Thus Kaluza-Klein theory unified  the 4-dimensional
Einstein equation and Maxwell's equations. However, as the
scalar wave equation did not correspond to particles observed in physical
experiments, the theory was forgotten for a long time due to
the dawn of quantum mechanics. Later in 1960-1980, it was re-invented in the creation of string theories
when the manifold $S^1$ was replaced by a higher dimensional manifolds, see \cite{book3}.
However,  the Kaluza-Klein theory is still considered as an interesting
simple model close to string theory suitable for testing
ideas.

To consider the relation of the Kaluza-Klein model to the main text in the paper, let
us consider 5-dimensional Einstein equations 
 with some matter model on  manifolds $N_\e=\R\times M_\e$, $\e>0$,
 where $(\{t\}\times M_\e, \,\overline g_\e(\cdot,t))$, $t\in \R$, are compact 4-dimensional Riemannian manifolds.
 Let $\hat {\overline  g}_{\e}=-dt^2+\overline  g_{\e}(\cdot,t)$ be the background
metric on $N_\e$ and  assume that the metric  $\overline g_\e(\cdot,t)$ is independent of the variable
$t$. Moreover, assume that we can make small perturbations to the matter fields
in the domain $\R_+\times M_\e$
that cause the metric to become a small perturbation
 $\hat g_{\e}(\cdot,t;h)=-dt^2+g_{\e}(\cdot,t;h)$ of  the metric  $g_\e(\cdot,t)$,
 where $h>0$ is a small parameter related to the amplitude of the perturbation.
By representing tensors  $g_{\e}(t,x;h)$ for all $h$ at appropriate coordinates (the so-called wave
gauge coordinates), one obtains that
the tensor $\tilde g(t,x)=\p_h \hat  g_{\e}(t,x;h)|_{h=0}$
satisfies the linearized Einstein equations, that is, a wave equation
\beq\nonumber
& &
\hat {\overline \square}\tilde g_{jk}(t,x)+b_{jk}^{lpq}(t,x)\hat {\overline \nabla}_l\tilde g_{pq}(t,x)+
c_{jk}^{pq}(t,x)\tilde g_{pq}(t,x)=\tilde T_{jk}(t,x)\quad\hbox{on }M_\e \times \R,
\\
& &\tilde g_{pq}(t,x)=0\quad\hbox{for $t<t_-$ for some $t_-\in\R$},\label{system 2}
\eeq
see \cite[Ch. 6]{ChBook}, where $\hat {\overline \square}=\square^{\hat{\overline g}}$
is the wave operator, $\hat {\overline \nabla}=\nabla^{\hat{ \overline g}}$ is the 
covariant derivative
with respect to the metric $\hat {\overline g}$, and $\tilde T$ is a source
term corresponding to the perturbation of the stress-energy tensor.
We mention that in realistic physical models,  $\tilde T$ should satisfy
a conservation law but we do not discuss this issue here.

Let us now consider a scalar equation analogous to (\ref{system 2})  for 
a  real-valued function  $U_\e(t,x)$ on $\R\times M_\e,$
\ba
& &
\hat {\overline \square}U_\e(t,x)+B_\e^{\nu}(x){\overline \nabla}_{\nu}U_\e(t,x)+
C_\e(x)U_\e (t,x)=F_\e(t,x)\quad\hbox{on }\R\times M_\e,
\\
& &U_\e(t,x)=0\quad\hbox{for $t<0$}.
\ea
We will apply to the solution $U_\e(t,x)$  the wave-to-heat transformation 
$$
({\cal T}f)(t)=\frac 1 {2\pi}\int_{\R\times \R} e^{-\xi^2t+it'\xi}f(t')\,dt'd\xi
$$
 in the time variable, and denote $u_\e(t,x):=({\cal T}U_\e)(t,x)$ and
  $f_\e(t,x):=({\cal T}F_\e)(t,x)$. 
Then $u_\e$ satisfies the heat equation 
\ba
& &\Big(\frac \p {\p t}+\Delta_{{\overline g}_\e}+B_\e^{\nu}(x)\overline \nabla_{\nu}+C_\e(x) \Big)u(t,x)=f_\e(t,x)\quad\hbox{on }M_\e \times \R_+,\\
& & \;\, u_\e|_{t=0}=0,
\ea
where $\Delta_{{\overline g}_\e}$ is the 3-dimensional (nonnegative definite) Laplace-Beltrami operator
on $(M_\e,{{\overline g}_\e})$. 
Then, if we 
can control the source term $F_\e$ and measure the field $U_\e$
for the wave equation (with a measurement error), we can also produce many sources
$f_\e$ for the heat equation and compute the corresponding fields
$u_\e$. In this paper we have assumed that we are given the values
of the heat kernel, corresponding to measurements with point sources, at the $\delta$-dense points in the
subset $(\delta,\delta^{-1})\times M_\e$ of the space-time with some error.
%
%
Due to the above relation of the heat equation to the wave equation
and the hyperbolic nature of the linearized Einstein equation, the inverse problem for
the pointwise heat data can be considered as a (very much) simplified
version of the question: if the observations of the small perturbations
of physical fields in the subset $\R\times \Omega$ of an almost stationary, almost collapsed
universe $\R\times M_\e$ can be used to find the metric of 
$\R\times M_\e$ in a stable way.
 As $M_\e$ can be considered as a $S^1$-fiber bundle $\pi:M_\e\to M_0$ on a $3$-manifold $M_0$, it is interesting
to ask if the  measurements at the $\delta$-dense subset, where $\delta$ is much larger than $\e$, 
can be used to determine e.g.\ the relative
volume of the
almost collapsed fibers $\pi^{-1}(y)$, $y\in M_0$. Physically,
this means the question if 
 the macroscopic measurements be used to find information on
the possible changes of the parameters of the almost collapsed structures of the universe
in the space-time. 
We emphasize that the questions discussed here in this subsection are
not related to the practical testing of string theory, but more to the philosophical
question: can the properties of the almost collapsed structures in principle be observed
using macroscopic observations, or not.

 
 

\appendix
\section{Remarks on the smoothness in Calabi-Hartman \cite{CH} and Montgomery-Zippin
\cite{MZ}} \label{app:MZ}

The goal of this section is to prove a generalization of Montgomery-Zippin's Theorem \cite{MZ}, Proposition \ref{Pr-MZ}, which is used in part I of the paper \cite{KLLY}.
Let $N$ be either the manifold $M$ or $G \times M$, where $G$ is a Lie group
 of isometries acting on $M$.
We consider Zygmund classes $C^s_*(N)$, $s>0$. To define these spaces,
we cover $N$ by a finite number of coordinate charts, $(U_{a}, \Phi_a)$ with, e.g.,
$\Phi_a(U_a)= Q_{2 r},$ where $Q_r$ is a cube with side $r$,
and assume that 
$$
\bigcup_{a=1}^J \left(\Phi_a^{-1}(Q_r) \right) =N.
$$
Then the definition of the norm in $C^s_*(N)$  is analogous to
the definition of these spaces in Euclidean space, cf. \cite[Section 2.7]{Tri}.

\begin{definition}
Let $\hat s \in \Z_+,\, \tilde s \in (0, 1]$ and $s=\hat s +\tilde s$.
 We say that $f: N \to \R$ is in $C^s_*(N)$, if for
some $\delta <r$,
\bequ \label{Zygmund1}
\|f\|_{C^s_*(N)} = \|f\|_{C(N)}+ \sum_{a=1}^J \left[
\sup_{x \in Q_r} \sup_{|h| < \delta} \sum_{|\beta| \leq \hat s} \frac{1}{|h|^{\tilde s}}
\, \big|(\Delta ^2_h \p^{\beta}f_a)(x) \big| \right] < \infty,
\eequ
where $f_a:=(\varphi_a\,\cdotp f)\circ \Phi_a^{-1}$ and $\varphi_a\in C^\infty_0(U_a)$ are functions
for which $\sum_{a=1}^J \varphi_a(x)=1$, and
$$
(\Delta ^2_h f_a)(x)= f_a(x+h)+f_a(x-h)-2f_a(x).
$$
If condition (\ref{Zygmund1}) is satisfied, it defines the norm of $f$ in 
$C^s_*(N)$. We will also consider maps $F:N\to M$ and denote $F\in C^s_*(N;M)$ if the coordinate representation of $F$ in the local
coordinates are $C^s_*$-smooth.
\end{definition}

Note that even though the norm (\ref{Zygmund1}) depends on the local coordinates used,
the smooth partition of unity $\varphi_a$, and on $\delta$, the resulting norms are
equivalent. Also, when $s \notin \Z_+$, $C^s_*(N)$ coincides with the H\"older spaces $C^{\hat s, \tilde s}(N).$

By \cite[Theorem 2.7.2(2)]{Tri}, the norm involving only terms with the finite differences along the coordinate axis $x^j$
of the partial derivatives  along the same coordinate axis, namely
\bequ \label{Zygmund2}
\|f\|_{C^s_*(N)}^{(1)}=\|f\|_{C(N)}+ \sum_{a=1}^J  \left[
\sup_{x \in Q_r} \sup_{0<\rho < \delta} \sum_{j=1}^n \frac{1}{\rho^{\tilde s}}
\, \left|\left(\Delta ^2_{\rho, j} \frac{\p ^{\hat s}f_a}{(\p x^j)^{\hat s}}\right)(x) \right| \right],
\eequ
is equivalent to (\ref{Zygmund1}), where
$$
(\Delta ^2_{\rho, j} h)(x)= h(x+\rho e_j)+h(x-\rho e_j)-2h(x).
$$

Let  $(M^i,h^i)$, $i=1,2$, be Riemannian manifolds and let $B^i\subset M^i$
be metric balls. We will use estimates presented in  \cite{CH} for a
map $F:M^1\to M^2$ whose restriction to the ball $B^1$  defines 
an isometry $F: (B^1, h^1) \to (B^2, h^2)$. We note that the constants in
these estimates  depend only
on the norms of $h^i, \, i=1,2,$ in the appropriate function classes $C^s_*$ in
some larger balls containing $B^i$ and the radii of these balls.
Thus, if a Lie group $G$ has isometric actions on manifold $M$ 
and $F=F_{g}:M\to M$ is the action of the group element $ g \in G$, then in the case
where $M$ and $G$ are compact, we obtain uniform estimates by covering the manifold and the Lie group
with finite number of balls. In the case where $M$ is compact but $G$  is not, we observe
that the estimates on a finite collection of balls covering $ M$ and one ball ${\cal B}\subset G$ for which $\cup_{g\in G}(g{\cal B})=G$ yield uniform estimates 
 on the space $G\times M$.

Let $(M, h)$ be a compact Riemannian manifold with a metric $h \in C^s_*(M),$ $ s>1,$ i.e., in suitable local coordinates the elements $h_{jk}(x)$ of the metric tensor $h$ are in $C^s_*(M).$
Let $G$ be a Lie  group of transformations acting on $M$ as isometries, i.e., for
any $g \in G$  the action of $g$, denoted by $F_g:M \to M$, is an isometry.
Let us denote
$$
F: G \times M \to M, \quad F(g, x)= F_g(x).
$$
We will use local coordinates $(x^j)_{j=1}^n=(x^1,\dots,x^n)$ of $M$ and 
$(g^\a)_{\a=1}^p=(g^1,\dots,g^p)$  of $G$. Below, we will use  Latin indices $i,k,l$ for 
coordinates on $M$ and Greek indices $\a,\beta,\gamma$ on coordinates on $G$.
Thus e.g.\ for a function $f:G\times M\to \R$ we often denote
  $\p_\a f(g,x)=\frac{\p f}{\p g^\a}(g,x)$ and $\p_j f(g,x)=\frac{\p f}{\p x^j}(g,x)$.

Next, let us assume that
 $h\in C^s_*(M),$ $ s>1$. Our aim is to prove that
 $F:G\times M\to M$ is in $C^{s+1}_*(G\times M;M)$, that is, in local coordinates
 the components $F^m(g,x)$ of $F(g,x)$ are in  $C^{s+1}_*(G\times M)$.
 Let us start with the case when $1 <s \leq 2$.
 Let $g\in G$ be fixed for the moment and denote $F=F_g$. Then $F \in C^2(M;M)$ due to \cite{CH}.
Using local coordinates of $M$ in sufficiently small balls $B_1$ and $B_2$ satisfying  $F(B_1)\subset B_2$, we have by 
 \cite[formula (5.2)]{CH},
\bequ\label{calabi1}
\Gamma^{(1), p}_{i j}(x) \frac{\p F^m}{\p x^p} -\Gamma^{(2), m}_{p q}(F(x))\,
\frac{\p F^p}{\p x^i}\,\frac{\p F^q}{\p x^j}= \frac{\p^2 F^m}{\p x^i \p x^j},
\eequ
where $\Gamma^{(1), p}_{i j}(x)$ and  $\Gamma^{(2), p}_{i j}(x)$ are the Christoffel symbols of the metric $h$ in 
$B_1$ and $B_2$, correspondingly.
As $h \in C^s_*$ and $F \in C^2$, we see easily that all terms in the left side of formula (\ref{calabi1}), except maybe  the term $\Gamma^{(2), m}_{p q}(F(x))$, are in
$C^{s-1}_*$. Next we consider
this term and will show that since 
$\Gamma^{(2), m}_{p q}\in C^{s-1}_*$ and $F \in C^2$, their
composition satisfies 
\bequ \label{2.30.09}
\Gamma^{(2), m}_{p q}\circ F \in C^{s-1}_*(B_1).
\eequ
To show this, observe that since $F \in C^2$, we have
\bequ \label{1.30.09}
\left|F\left(\frac{x+y}{2}\right) -\frac12 F(x)-\frac12 F(y)\right| \leq C|x-y|^2.
\eequ
Let $(s-1)/2< t <s-1$. Then $\Gamma^{(2), m}_{p q}  \in C^t(B_2)$, and for $x,y\in B_1$ we have
$$
\left|\Gamma^{(2), m}_{p q}\left(F(\frac{x+y}{2})\right)- 
\Gamma^{(2), m}_{p q}\left( \frac12 F(x)+\frac12 F(y)\right) \right| \leq c |x-y|^{2t}
\leq c'|x-y|^{s-1}.
$$
Moreover, we see that
\bfo
& &\left|\Gamma^{(2), m}_{p q}(F(x))+\Gamma^{(2), m}_{p q}(F(y))-2\Gamma^{(2), m}_{p q}\left(F(\frac{x+y}{2})\right)\right| \leq
\\ \nonumber
& &\left|\Gamma^{(2), m}_{p q}(F(x))+\Gamma^{(2), m}_{p q}(F(y))-
2\Gamma^{(2), m}_{p q}\left( \frac12 F(x)+\frac12 F(y)\right)
\right|+c'|x-y|^{s-1}
\\ \nonumber
& & \leq c|F(x)-F(y)|^{s-1}+c'|x-y|^{s-1} \leq C|x-y|^{s-1}.
\efo
which proves the claim (\ref{2.30.09}).

By \cite{Hor3}, the space $C^{s-1}_*(M)$  with $s>1$ is an algebra, that is, the pointwise
multiplication satisfies $C^{s-1}_*(M) \cdot C^{s-1}_*(M) \subset C^{s-1}_*(M)$. Thus
it follows from (\ref{calabi1}) that
$\frac{\p^2 F^m}{\p x^i \p x^j} \in C^{s-1}_*(B_1)$. This shows that
$F_g \in C^{s+1}_*(M;M)$ for $g\in G$.

Differentiating further (\ref{calabi1})
with respect to variables $x^j$ and
repeating the above considerations, we see that if $h\in C^{s}_*(M)$ with $ s \in (2, 3]$ then $F_g \in C^{s+1}_*(M,M)$. Iterating this construction, we obtain the following result.

\begin{lemma} \label{C-H}
Let $(M, h)$, $h \in C^s_*(M)$, $s>1$, be a compact 
Riemannian manifold.
Let $G$ be a Lie  group of transformations acting on $M$ as isometries, i.e., for
any $g \in G,$  the corresponding action $F_g:M \to M$ is an isometry.
Then, 
for each $g\in G$  the map $F_g$ is in $C^{s+1}_*(M;M)$, and the norm
 of $F_g$  in $C^{s+1}_*(M;M)$ is
uniformly bounded with respect to $g \in G$.
\end{lemma}

We now turn to \cite{MZ}. The corresponding result in \cite{MZ} which we need
is as follows (see Theorem on p. 212, sec. 5.2, \cite{MZ}).

\begin{theorem} \label{Th-MZ} (Montgomery-Zippin)
Let $M$ be a differentiable manifold of class $C^k$, $k \in \Z_+$.
Let $G$ be a Lie group of transformations acting on $M$ so that $F_g(\cdot) \in C^k(M;M)$ uniformly in $g\in G$, where $F_g$ is the action of $g\in G$. 
Define $F:G\times M\to M$ by $F(g,x)=F_g(x)$.
Then, 
in   the local real-analytic coordinates 
of $G$ and the $C^{k}$-smooth coordinates of $M$,  we have
$$
F \in C^k(G \times M;M).
$$
\end{theorem}
Note that as $G$ is a Lie group it has an analytic structure.

\smallskip
Our goal is to extend Theorem \ref{Th-MZ} to  spaces $C^s_*$.

\begin{proposition} \label{Pr-MZ} (Generalization of  Montgomery-Zippin's theorem)

\noindent Let $s>1$ and $(M, h)$ be a compact Riemannian manifold
with  the $C^{s}_*$-smooth coordinates.
Let $G$ be a Lie  group of transformations acting on $M$
so that $F_g\in C^{s}_*(M;M)$ uniformly in $g\in G$.
Then, in  the local real-analytic coordinates 
of $G$ and the $C^{s}_*$-smooth coordinates of $M$, 
$$
F \in C^{s}_*(G \times M;M).
$$
\end{proposition}

\begin{proof}
Let $1< s'< s''<s$.
Let us show that the map
\bequ \label{cont_wrt_g}
{\mathcal F}:G \to C^{s'}_*(M;M),\quad g \mapsto F_g(\cdot)
\eequ
is continuous.  Essentially, this follows from the facts that $s'<s$ and our assumption that
 $F_g\in C^{s}_*(M;M)$ uniformly in $g\in G$, and that
$F(g, x)$ is uniformly continuous with respect to $(g, x)\in G\times M$. 
Let us explain this in details.
First, we consider the case when $1<s \leq 2$ and  a scalar, uniformly continuous function $f:G\times M\to \R$.
Denoting $f_g(x)=f(g, x)$ and assuming
that $g\mapsto f_g$ is continuous map $G\to C^{s'}_*(M)$,
we have, for $h,\rho\in \R$, $|h|<1$, $|\rho|\leq 1$,
\beq \label{13.08.1}
& &\bigg|\frac{\p(f_g-f_{g'})(x+h)}{\p x^i} - \frac{\p(f_g-f_{g'})(x)}{\p x^i}
\\ \nonumber
& &-\frac{1}{\rho}\bigg((f_g-f_{g'})(x+h+\rho e_i)-(f_g-f_{g'})(x+h)\\
& &\quad\quad \nonumber
+(f_g-f_{g'})(x+\rho e_i)-(f_g-f_{g'})(x)  \bigg)\bigg| 
\leq
C(s'') \rho^{s''},
\eeq
where $C(s'')$ is uniform with respect to $g, g'\in G$.

On the other hand,
\bequ \label{13.08.2}
\frac{1}{|h|^{s'}}\left|\frac{\p (f_g-f_{g'})}{\p x^i}(x+h)- \frac{\p (f_g-f_{g'})}{\p x^i}(x) \right|
\leq C(s'') |h|^{s''-s'}.
\eequ
Thus, for any $\e >0$,  we can find $h_0>0$ such that the left hand side of (\ref{13.08.2}) is less than $\e/2$ for $|h| <h_0$.
Moreover, we can find $\rho_0>0$ such that the right hand side of (\ref{13.08.1}) is less than $\e/2$ for
$|\rho| \leq \rho_0$. As $f:G\times M\to \R$ is continuous and, therefore uniformly continuous,
any $g\in G$ has a neighborhood $V_\e(g)$ in  $G$, such that for $g'\in V_\e(g)$ and $x\in M$,
$$
|(f_g-f_{g'})(x)| < \frac18 \rho_0 \e h_0^{s'}.
$$
Combining the above considerations we see that if $g' \in V_\e(g)$, then
$$
\|f_g-f_{g'}\|_{C^{s'}_*(M)} < \e.
$$
Applying the above for the coordinate representation of $F:G\times M\to M$, we obtain (\ref{cont_wrt_g}) in the case when $1<s'<s\leq 2$. Analyzing the higher derivatives similarly,
we obtain (\ref{cont_wrt_g}) for all $s>s'>1$.

\smallskip
Next, let $g^{\a},\, \a=1, 2,\dots, p$, be real-analytic coordinates on
$G$ near the identity element $\hbox{id}$, for which the coordinates
of the identity element are $(0,0,\dots,0)$.
Let $e_\a=\p_{g_\a}|_{g=\textrm{id}}$, and 
$t e_\a,$ $ t \in \R$, denote the elements of the one-parameter subgroup in $G$ generated by $e_\a$.
 Let
 $x^j,$ $j=1,2,\dots,n$, be the local
coordinates in an open set $B\subset M$ and  $x_0\in B$. Near $(\hbox{id},x_0)\in G\times M$, we represent $F$ in these coordinates
as $F(g,x)=(F^m(g,x))_{m=1}^n$.
As noted before, we will use  Latin indices $i,k,l$ for 
coordinates on $M$ and Greek indices $\a,\beta,\gamma$ on coordinates on $G$.
In particular, we will use this to indicate derivatives with respect to
$g^\a$ and $x^j$.

Consider now the identity given in \cite[Lemma A a), p. 209]{MZ},
 \bequ \label{MZ1}
F^m(\rho e_{\a}, x)-F^m(0, x)=\sum_{j=1}^n \left(\int_0^1 {\p_j F^m(t\rho e_{\a}, x)} \,dt \right)
\p_\a F^j(0, x),
\eequ
where $\rho>0$  is sufficiently small, and as noted before, 
$$
\p_j F^m(g,x)=\frac {\p F^m}{\p x^j}(g,x)\quad\hbox{and}\quad \p_\a F^j(g,x)=\frac{\p F^j}{\p g^{\a}}(g,x).
$$ 
Since  
$$
\p_j F^m(0, x)=\delta^m_j,\quad x \in M,
$$
it follows from the uniform continuity of $\nabla_x F(g,x)$ with respect to $(g,x)$ that the matrix
\bequ \label{MZ3}
\left[\int_0^1 \p_j F^m(t\rho e_{\a}, x) \,dt \right]_{j, m=1}^n
\eequ
is invertible for sufficiently small $\rho$ (note that $\rho>0$  can be chosen uniformly  with respect to $x\in M$). 
Denoting the inverse matrix (\ref{MZ3}) by
$\phi^m_j(\rho e_{\a}, x)$, we obtain from (\ref{MZ1}) the identity
\bequ \label{MZ2}
\p_\a F^m(0, x)= \phi^m_j(\rho e_{\a}, x) \,
\left[ F^j(\rho e_{\a}, x)-F^j(0, x) \right],
\eequ
when $\rho >0$ is sufficiently small. In the following, we can choose $\rho_0>0$
so that (\ref{MZ2}) is valid for all $0<\rho<\rho_0$  and  $x \in M$. 
Using our assumption that
 $F_g\in C^{s}_*(M;M)$, uniformly in $g\in G$, we see that the matrix
(\ref{MZ3}) is in $C^{s-1}_*(M)$ and thus its inverse matrix satisfies $\phi^m_j(\rho_0 e_{\a}, \cdot) \in C^{s-1}_*(M)$.
As $F_g \in C^{s}_*(M;M)$ uniformly with respect to $g\in G$, formula (\ref{MZ2}) implies
that
\bequ\label{MZ4}
\p_\a F^m(0, x)\in C^{s-1}_*(M).
\eequ

Our next goal is to show that 
$$
\p_\a F^m(g, x)=\frac{\p F^m(g, x)}{\p g^{\a}} \in C^{s-1}_*(G),
$$
uniformly with respect to $x \in M$.
To this end, let $g_0\in G$ and $x\in M$ be fixed, and consider an element $g\in G$ which varies 
in a neighborhood of $g_0^{-1}$.
Let us consider function $\tilde g(g):=g_0\, g^{-1}$, $\tilde g:G\to G$ and
$\tilde y(g)=F_g(x)$, $\tilde y:G\to M$.
The group $G$ has real-analytic local coordinates $(g^\a)_{\a=1}^p=
(g^1,\dots,g^p)$ near $g_0$. 
As the group operation $(g',g'')\mapsto g'\,\cdotp (g'')^{-1}$ is real-analytic, the local coordinates of $\tilde g(g)$,
denoted by $\tilde g^\a(g^1,\dots ,g^p)$, are also real-analytic. 
Then,
$$
F(\tilde g(g), \tilde y(g))= F(g_0\, g^{-1}, F_g(x))=F(g_0\, g^{-1} \, g, x)=F(g_0, x).
$$
Thus, $F(\tilde g(g), \tilde y(g))$ is independent of $g$, and using the
chain rule, we obtain
\beq \label{MZ5}
0&=& \frac{\p }{\p g^{\a}}(F(\tilde g(g), \tilde y(g)))\\
\nonumber 
&=&\p_\beta F^m(\tilde g(g), \tilde y(g)) \frac{\p \tilde g^{\beta}(g)}{\p g^{\a}} 
+
\p_j F^m(\tilde g(g), \tilde y(g)) \frac{\p \tilde y^j(g)}{\p g^{\a}}.
\eeq

Let us next use that fact that for all $z\in M$  we have  $\p_j F^m (g,z)|_{g=id}=\delta^m_j$. 
As derivatives of the map $(g,z)\mapsto \p_j F^m (g, z)$
 is continuous, the matrix $(\p_j F^m (g, z))_{j,m=1}^n$ is invertible for $(g,z)$ near $(\hbox{id},z)$.
 Let us  denote this inverse matrix by
by $\Phi^m_j(g,z)$. Composing this function with $\tilde g(g)$ and $\tilde y(g)$, we see that  the matrix
$(\p_j F^m(\tilde g(g), y(g)))_{j,m=1}^n$ has an inverse
$\Phi^m_j(\tilde g(g), y(g))$ when $g$ is sufficiently near to $g_0$.
Then we obtain from (\ref{MZ5}),
\beq \label{MZ6}
\p_\a F^k (g, x)&=&\p_\a \tilde y^m(g, x)\\
\nonumber &=&-\Phi^m_k(\tilde g(g),F(g, x)) \cdot
\p_\beta F^k(\tilde g(g), F(g, x)) \cdot 
\p_\a \tilde g^{\beta}(g).
\eeq
Now we evaluate (\ref{MZ6}) at $g=g_0$. As $\Phi^m_k(\tilde g(g_0),F(g_0, x))=\delta^m_k$, and in the local
coordinates of $G$ we have
$\tilde g(g_0)=0$, we obtain 
\bequ\label{MZ aux}
\p_\a F^m (g_0, x)
=-
\p_\beta F^m (0, F(g_0, x)) \cdot 
\p_\a \tilde g^{\beta}(g_0).
\eequ
Note that $g_0$ and $x$ above were arbitrary, even though we considered those as fixed parameters.
Next we change our point of view and will consider those as variables. For this
end, we use variable $g\in G$ instead of $g_0$ so that (\ref{MZ aux}) reads as
\bequ\label{MZ aux B}
\p_\a F^m (g, x)
=-
\p_\beta F^m (0, F(g, x)) \cdot 
\p_\a \tilde g^{\beta}(g).
\eequ
Here $g\mapsto \p_\a \tilde g^{\beta}(g)$  is real-analytic and by formula (\ref{MZ4}), 
the map $z\mapsto \p_\beta F^m (0, z)$ is in $C^{s-1}_*(M)$.

Let us next show that 
\bequ \label{MZ7}
\p_{\beta} F^m (0, F(g, x)) \in C^{s-1}_*(G \times M).
\eequ
Once this is shown, the fact that
$\p_\a \tilde g^{\beta}(g)$  is real-analytic  and the equation
(\ref{MZ aux B}) imply that  $\p_\a F^m(g, x) \in C^{s-1}_*(G \times M).$

To show (\ref{MZ7}),
we start by considering the case when $1<s < 2$. Then, using (\ref{MZ4})
and that fact that  $F \in C^1(G \times M)$ by Theorem \ref{Th-MZ}, we obtain in local coordinates
\beq\label{Holder estimate M}
& &\left|\p_\beta F^m(0, F(g+h,x))- \p_\beta F^m(0, F(g,x)) \right| 
\\ \nonumber
& \leq& C|F(g+h, x)-F(g, x)|^{s-1}
\leq C' |h|^{s-1},
\eeq
where $C, C'>0$ are uniform with respect to $(g, x)\in G\times M$. 
Hence (\ref{MZ7}) is valid, i.e., $\p_\a F^m(g, x) \in C^{s-1}_*(G \times M).$
Combining this with (\ref{MZ aux B}) we see that
for any $x\in M$,  the function $g\mapsto \p_\a F^m(g, x)$  
is in $C^{s-1}_*(G)$ and its norm in  $C^{s-1}_*(G)$  is
bounded by a constant which is independent of $x\in M$.
By assumption,
functions $x\mapsto \p_j F^m(g, x)$
are in $C^{s-1}_*(M)$, uniformly in $g\in G$. 
Thus, by  (\ref{Zygmund2}), i.e., \cite[Thm. 2.7.2(2)]{Tri}, the map 
$F:G\times M\to M$  is $C^{s}_*$-smooth with respect to
both $g$ and $x$, that is,
\bequ \label{MZZ4}
F(g,x) \in C^{s}_*(G \times M).
\eequ

Next, let us consider the case when  $s=2$.
To apply (\ref{Zygmund2}), we need to estimate 
$$
T_\beta^k :=\p_\beta F^k(0, F(g, x))
$$
and in particular its finite difference
$$
\Delta^2_h \left(T_\beta^k\right)=\p_\beta F^k(0, F(g+h, x))
 +\p_\beta F^k (0, F(g-h, x)) -
 2\p_\beta F^k (0, F(g, x)).
$$
 Using (\ref{MZZ4}), we observe that for $h=(h^\a)_{\a=1}^p$,
$$
|F^k(g+h, x)- F^k(g, x)- \p_\a F^k(g, x) \cdot h^\a| \leq C_s |h|^{s'}, \quad \hbox{for any}\,\,
s' <2. 
$$
On the other hand, by (\ref{MZZ4}), we have
$$
\p_\beta F^k (0, F(g, x)) \in C^t(G \times M), 
\quad \hbox{for any}\,\,
t <1.
$$
Let us combine these two  observations in the 
case when $s'=3/2$ and $t=2/3$. Then
\beq \label{MZ8}
 \Delta^2_h \left(T_\beta^k \right)
&=&\p_\beta F^k\big(0, F(g, x)+ (\p_\a F^k)(g, x) \,h^\a+O(|h|^{s'}))\\
& &+ \nonumber
 \p_\beta F^k (0, F(g, x)- (\p_\a F^k)(g, x) \, h^\a+O(|h|^{s'})) 
 \\ \nonumber
 & &-
 2\p_\beta F^k (0, F(g, x)) 
 \\ \nonumber
&=&O(  (|h|^{s'})^ t)= O(|h|).
\eeq
Applying (\ref{MZ4}) with $s=2$ to estimate the finite
differences of $T_\beta^k$ in the  $x^j$ directions
and  (\ref{MZ8})
to estimate the finite
differences of $T_\beta^k$  in the  $g^\a$ directions
 in the formula (\ref {Zygmund2}) for the Zygmund norm,
 we obtain  $T_\beta^k=\p_\beta F^k(g, x) \in C^1_*(G \times M)$. 
Moreover, by our assumption,
functions $x\mapsto \p_j F^m(g, x)$
are in $C^{1}_*(M)$ uniformly in $g\in G$. 
Thus, by applying (\ref{Zygmund2}) again,  we see that $F(g, x) \in C^2_*(G \times M)$.  

\smallskip
Next, we consider the case when $2<s \leq 3$.  By differentiating formula (\ref{MZ aux B}) 
with respect to $g^\beta$, we obtain
\beq \label{MZZ5}
\p_\a\p_\beta F^k(g, x)&=&-\frac{\p}{\p g^\beta} \bigg( \p_{\gamma} F^k (0, F(g, x)) \, 
\p_\a \tilde g^\gamma(g)\bigg)
\\ \nonumber
& =&(\p_\gamma\p_j F^k (0, F(g, x))\,\p_\beta F^j(g, x)\,\p_\a \tilde g^\gamma (g)
\\ \nonumber
& &+\p_\gamma F^k (0, F(g, x)) \,\frac{\p^2 \tilde g^\gamma}
{\p g^{\a} \p g^\beta}(g).
\eeq
We see easily that all terms in the right hand side of 
equation (\ref{MZZ5}), except maybe the term    $\p_\gamma\p_j F^k (0, F(g, x))$,
are in $C^1_*(G \times M)$. To analyze this term, we
observe that by (\ref{MZ4}) we have $\p_\gamma\p_j F^k (0, y) \in C^{s-2}_*(M)$. 
Thus considerations similar to those leading to 
the inequality (\ref{MZ7}) show that
$$
\p_\gamma\p_j F^k (0, F(g, x)) \in C^{s-2}_*(G \times M).
$$
Hence by (\ref{MZZ5}),
\bequ\label{alphabeta estimate}
\p_{\a}\p_\beta F^k(g, x)\in C^{s-2}_*(G),
\eequ
uniformly with respect to $x\in M$.
By assumption,
$F_g(\cdot) \in C^{s}_*(M;M)$ uniformly in $g\in G$,
so that
\bequ\label{ij estimate}
\p_{i}\p_j F^k(g, x)\in C^{s-2}_*(M),
\eequ
uniformly with respect to $g\in G$.

Applying (\ref{Zygmund2}) with  (\ref{alphabeta estimate}) and (\ref{ij estimate}),
we see that $F \in C^{s}_*(G \times M)$.
Iterating the procedures above we obtain (\ref{MZ7}) for all $s>1$.
\end{proof}

Combining Lemma \ref{C-H} and Proposition \ref{Pr-MZ}, we obtain the following result.

\begin{coro} \label{CH+MZ}
Let $M$ be a compact, $C^{s+1}_*$-smooth differentiable manifold, $ s>1$, and $h \in C^s_*(M)$ be a Riemannian
metric on $M$.
Let $G$ be a Lie group  acting on $(M,h)$ for which the actions $F_g: M\to M$ of the elements $g\in G$ are isometries. Define $F:G\times M\to M$ by $F(g,x)=F_g(x)$.
Then $F \in C^{s+1}_*(G \times M;M)$.
\end{coro}

\section{Definition of orbifolds}
\label{appendix_orbifold}

Orbifolds were first introduced under the name "V-manifolds" in \cite{Satake} and later studied in \cite{Thurston} under its current name.
Here we review the definition in \cite{Satake}.

\begin{definition}[Smooth orbifold] \label{def-orbifold}
An ($d$-dimensional) orbifold chart on a topological space $X$ is the triple $(\tilde{U},G,\varphi)$ consisting of a connected open set $\tilde{U}\subset \mathbb{R}^d$, a finite group $G$ of diffeomorphisms of $\tilde{U}$, and a continuous map $\varphi$ from $\tilde{U}$ to an open set $U\subset X$ inducing a homeomorphism from the quotient space $\tilde{U}/G$ to $U$.

Let $(\tilde{U},G,\varphi)$ and $(\tilde{U}',G',\varphi')$ be two orbifold charts for $U$ and $U'$ with $U\subset U'$.
By an smooth embedding $\iota:(\tilde{U},G,\varphi)\to (\tilde{U}',G',\varphi')$ we mean a smooth embedding $\iota:\tilde{U}\to \tilde{U}'$ such that $\varphi=\varphi'\circ \iota$.

An orbifold of dimension $d$ consists of a Hausdorff topogical space $X$ and a family $\mathcal{A}$ of $d$-dimensional orbifold charts for open subsets in $X$ satisfying the following conditions.

\smallskip
(1) Every point $x\in X$ is contained in at least one orbifold chart in $\mathcal{A}$ (i.e., there exists an open neighborhood $U$ containing $x$ such that there is an orbifold chart for $U$ in $\mathcal{A}$).
If $x$ is contained in two orbifold charts in $\mathcal{A}$ for $U_1$ and $U_2$, then there exists an open set $U_3\subset U_1\cap U_2$ containing $x$ such that there is an orbifold chart for $U_3$ in $\mathcal{A}$.

\smallskip
(2) If $(\tilde{U},G,\varphi)$ and $(\tilde{U}',G',\varphi')$ for $U$ and $U'$ are two orbifold charts in $\mathcal{A}$ such that $U\subset U'$, then there exists a smooth embedding from $(\tilde{U},G,\varphi)$ to $(\tilde{U}',G',\varphi')$.
\end{definition}

A point $x$ of an orbifold is called singular, if there exists an orbifold chart $(\tilde{U},G,\varphi)$ for some open set $U$ containing $x$ such that the isotropy group of $\varphi^{-1}(x)\in \tilde{U}$ in $G$ is nontrivial. Points that are not singular are called regular, and the set of regular points has the structure of an open manifold.
A closed manifold is an orbifold with the group $G$ in orbifold charts being the trivial group.
The quotient space $\mathbb{R}^2/\mathbb{Z}_m$, $m\geq 2$, is a cone with cone angle $2\pi/m$, where $\mathbb{Z}_m$ acts on $\mathbb{R}^2$ by a rotation of angle $2\pi/m$ around the origin.
A manifold with boundary can be given an orbifold structure: every point on the boundary has a neighborhood homeomorphic to an open set in the half-plane $\mathbb{R}^d/\mathbb{Z}_2$, where $\mathbb{Z}_2$ acts on $\mathbb{R}^d$ by reflection along the hyperplane that models the boundary.
If $M$ is a manifold and $G$ is a group acting properly discontinuously on $M$, then the quotient $M/G$ has an orbifold structure.
An orbifold is said to be \emph{good} if it is the quotient of a manifold by an action of a discrete group.

\begin{definition}[Riemannian orbifold]
A Riemannian metric on an orbifold is an assignment to each orbifold chart $(\tilde{U},G,\varphi)$ of a $G$-invariant Riemannian metric $g_{\tilde{U}}$ on $\tilde{U}$ satisfying that each embedding in Definition \ref{def-orbifold}(2) is an isometry.
A Riemannian orbifold is a smooth orbifold equipped with a Riemannian metric.
\end{definition}

\end{document}